\documentclass[reqno,11pt]{amsart}
\usepackage{color}
\usepackage{hyperref}

\usepackage{amsfonts}
\usepackage{amsthm}
\usepackage{amssymb}
\usepackage[utf8]{inputenc}
\usepackage{enumerate}

\newtheorem{Theorem}{Theorem}[section]
\newtheorem{Proposition}[Theorem]{Proposition}
\newtheorem{Lemma}[Theorem]{Lemma}
\newtheorem{Definition}[Theorem]{Definition}
\newtheorem{Corollary}[Theorem]{Corollary}
\newtheorem{theoremmain}{Theorem}

\theoremstyle{definition}
\newtheorem{Remark}[Theorem]{Remark}
\newtheorem{Example}[Theorem]{Example}
\newtheorem{Notation}[Theorem]{Notation}

\newtheorem*{Agradecimientos}{Acknowledgements}

\title{Cuspidal dicritical foliations and analytical invariants of cusps}
\author{Nuria Corral}
\address{Nuria Corral. Departamento de Matemáticas, Estadística y Computación. Universidad de Cantabria. Avda. de los Castros s/n, 39005 -- Santander, SPAIN}
\email{nuria.corral@unican.es}
\author{Oziel Gómez-Martínez}
\address{Oziel Gómez-Martínez. CIMAT (Centro de Investigación en Matemáticas), Guanajuato, México. }
\email{oziel.gomez@cimat.mx}
\author{David Senovilla-Sanz}
\address{David Senovilla-Sanz.}
\email{d.ssc37@gmail.com}
\thanks{The authors are supported by the Spanish research project PID2022-139631NB-I00 funded by the Agencia Estatal de Investigación - Ministerio de Ciencia e Innovación. The second author is also supported by Papiit (Dgapa, UNAM) IN10123 and SECIHTI through Estancias Posdoctorales en México 2022(1)}
\begin{document}

\begin{abstract}
In this paper we study dicritical foliations having a family of cusps as invariant curves. We give conditions to assure that all the curves of the family have the same semimodule of differential values. We also give an expression to compute the conductor of a cuspidal semimodule.
\end{abstract}
\maketitle
\tableofcontents
\section{Introduction}

Dicritical foliations in $(\mathbb{C}^2,0)$ are foliations with infinitely many invariant curves. Dicriticalness in dimension two is equivalent to the existence of non-invariant exceptional divisors in the sequence of blowing-ups desingularizing $\mathcal{F}$. The invariant curves of the foliation which are transversal to a non-invariant exceptional divisor are equisingular curves but they are not analytically equivalent in general. The objective of this paper is to give conditions to guarantee that all the curves in this family have the same semimodule of differential values.

Consider a foliation $\mathcal{F}$ in $(\mathbb{C}^2,0)$ and let $\pi: M \to (\mathbb{C}^2,0)$ be the minimal reduction of singularities  of $\mathcal{F}$. The foliation $\mathcal{F}$ is dicritical if there exists a component $E$ of the exceptional divisor $\pi^{-1}(0)$ which is not invariant by the strict transform $\pi^*\mathcal{F}$ of the foliation $\mathcal{F}$ by $\pi$. In this case, there exists a curve ${C}_P'$ invariant by $\pi^* \mathcal{F}$ and transversal to the divisor $E$ at each point $P \in E$ which is not a singular point of the divisor $\pi^{-1}(0)$. The projection $C_P=\pi(C_P')$ is a germ of curve in $(\mathbb{C}^2,0)$. This family of curves will be denoted $\operatorname{Curv}_E(\mathcal{F})$.

The objective of this paper is to give conditions over the foliation $\mathcal{F}$ to assure that the curves in $\operatorname{Curv}_E(\mathcal{F})$ have the same semimodule of differential values. We will focus on the case where the curves in $\operatorname{Curv}_E(\mathcal{F})$ are cusps, that is, curves with only one Puiseux pair. These foliations will be called {\em cuspidal dicritical foliations.} The results in this paper generalize the results obtained by the second author in \cite{Gom} where he gave conditions to assure that the curves  in  $\operatorname{Curv}_E(\mathcal{F})$ have the same Zariski invariant.

\medskip
Dicritical foliations appear in a natural way when studying the analytic classification of plane curves. The set of values $\Lambda_C$ given by the contact of 1-forms with an irreducible curve $C$ is an important analytic invariant of the curve $C$. The set $\Lambda_C$ is a $\Gamma_C$-semimodule, where $\Gamma_C$ is the semigroup of the curve $C$. The 1-forms whose differential values provide the basis of  $\Lambda_C$ as $\Gamma_C$-semimodule define dicritical foliations (see \cite{Can-C-SS-2023,Cor-H-H,For-R}). Let us detail all these concepts.

Given a primitive parametrization $\gamma(t)$ of the irreducible curve $C$ and a 1-form $\omega$, the differential value $\nu_C(\omega)$ is given by
$$\nu_C(\omega)= \text{ord}_t(\alpha(t))+1
$$
where $\gamma^*\omega=\alpha(t)dt$. The set of differential values $\Lambda_C$ is defined by
$$
\Lambda_C=\{\nu_C(\omega) \ : \ \omega \in \Omega^1_{(\mathbb{C}^2,0)}\}.
$$
The importance of the semimodule $\Lambda_C$ in the analytic classification of plane curve was pointed out by O. Zariski in his seminal book  \cite{Zar}, where he started the study of the analytic classification of plane curves in a systematic way. The semimodule $\Lambda_C$ is also the principal invariant in the complete classification of irreducible plane curves given by A. Hefez and M. E. Hernandes in \cite{Hef-H,Hef-H-2021}.

Let us describe some properties of the semimodule $\Lambda_C$. There is a unique basis $\{\lambda_i\}_{i=-1}^s$ of $\Lambda_C$, that is, $$\Lambda_C=\bigcup_{i=-1}^s (\lambda_i + \Gamma_C)$$
with $\lambda_j \not \in \bigcup_{i=-1}^{j-1}(\lambda_i + \Gamma_C)$. A list of 1-forms $(\omega_{-1},\omega_0,\omega_1,\ldots,\omega_s)$ such that
$$
\nu_C(\omega_i)=\lambda_i, \quad \text{ for } \quad i=-1,0,1,\ldots,s,
$$
is called a {\em minimal standard basis\/} for the curve $C$.

Properties concerning the structure of the semimodule $\Lambda_C$ are described by several authors when $C$ is a cusp (see for instance \cite{Del,Alb-A-MF,Alm-M-2021,Can-C-SS-2023}).

Assume now that $C$ is a cusp with Puiseux pair $(m,n)$ with $1 < n < m$ and $\gcd (n,m)=1$ and fix a  minimal standard basis $(\omega_{-1},\omega_0,\omega_1,\ldots,\omega_s)$ of $C$. Let us  consider  the foliation $\mathcal{G}_i$ in $(\mathbb{C}^2,0)$ defined by $\omega_i=0$ for $i=1,2,\ldots,s$. Geometrical properties of the foliations $\mathcal{G}_i$, for $i=1,\ldots,s$, were studied in  \cite{Can-C-SS-2023}.  In particular, the foliations $\mathcal{G}_i$ are dicritical in
the triple point $E$ of the resolution dual graph of the curve $C$. Moreover, the semimodule of differential values of a curve  in $\operatorname{Curv}_E(\mathcal{G}_i)$  is equal to
$$
\Lambda_{i-1}=\bigcup_{k=-1}^{i-1} (\lambda_k + \Gamma_C).
$$
The behaviour of these foliations motivates the work in this paper. Let us explain the main results.

Consider the minimal reduction of singularities $\pi_C : M \to (\mathbb{C}^2,0)$ of the cusp $C$ and let $E$ be the ``cuspidal divisor'', that is, the irreducible component $E$ of $\pi_C^{-1}(0)$  such that the curve $C'$ intersects $E$, with $C'$ the strict transform of $C$ by $\pi_C$.  Associated to the divisor $E$, there is a divisorial value $\nu_E( \  )$ defined for functions and forms (see Section~\ref{sec:div-ord}). In adapted coordinates, the divisorial value for a cuspidal divisor behaves as a weighted monomial order, that is, for a monomial $\nu_E(x^i y^j)=ni+mj$. Note that $\nu_E(\omega) \leq \nu_C(\omega)$ for any 1-form.

A foliation $\mathcal{F}$ is called a {\em totally $E$-dicritical foliation} if the strict transform $\pi^*_C \mathcal{F}$ by $\pi_C$ is transversal to $E$ and has normal crossings with $\pi_C^{-1}(0)$ at each point of $E$. The foliations $\mathcal{G}_i$ are totally $E$-dicritical and the divisorial value $\nu_E(\omega_i)=t_i$ is determined in terms of the combinatorial properties of the semimodule $\Lambda_C$ (see \cite{Can-C-SS-2023}). Let us recall the definition of the \textit{axes} $u_i,\tilde{u}_i$ and \textit{critical values} $t_i,\tilde{t}_i$ of the semimodule $\Lambda_C$.  For $i=0,1,\ldots,s,s+1$, the axes $u_i^n,u_i^m,u_i, \tilde{u}_i$ of $\Lambda_C$ are defined as
\begin{equation*}
\begin{aligned}
   u_i^n &= \min\{\lambda_{i-1} + m \ell \in \Lambda_{i-2} \ : \ell \geq 1\}; \quad &u_i^m & = \min \{ \lambda_{i-1} + m \ell \in \Lambda_{i-2} \ : \ell \geq 1 \}; \\
   u_i & =\min\{u_i^n,u_i^m\}; & \tilde{u}_i &=\min\{u_i^n,u_i^m \}.
\end{aligned}
\end{equation*}
Put $t_{-1}=\lambda_{-1}=n$, $t_0=\lambda_0=m$, and for $i=1,\ldots,s,s+1$,  the critical values are given by
\begin{equation*}
\begin{aligned}
   t_i^n &= t_{i-1} +u_i^n - \lambda_{i-1}; \quad &t_i^m & = t_{i-1} +u_i^m - \lambda_{i-1}; \\
   t_i & =\min\{t_i^n,t_i^m\}; & \tilde{t}_i &=\min\{t_i^n,t_i^m \}.
\end{aligned}
\end{equation*}
The combinatorial properties of the semimodule described by the above parameters will be key in the proof of the results in this paper (see Section~\ref{sec:cuspidal-semimodules} and Appendix~\ref{sec:apendix-combinatoria}). In particular, we prove that the conductor $c(\Lambda_C)$ of $\Lambda_C$ can be computed as
$$
c(\Lambda_C)=\tilde{u}_{s+1}-n-m+1.
$$
(see Appendix~\ref{ap:conductor}).

 Consider now any totally $E$-dicritical foliation $\mathcal{F}$ in $(\mathbb{C}^2,0)$ defined by $\omega_\mathcal{F}=0$ and such that $C \in \operatorname{Curv}_E(\mathcal{F})$. Given another curve $\widetilde{C} \in \operatorname{Curv}_E(\mathcal{F})$, we would like to determine when $\Lambda_{\widetilde{C}}=\Lambda_{{C}}$.
The differential value $\nu_{\widetilde{C}} (\omega)$ of a 1-form $\omega$ can be computed as
\begin{equation}\label{eq:introduccion1}
\nu_{\widetilde{C}} (\omega)=\nu_E(\omega \wedge \omega_{\mathcal{F}}) - \nu_E(\omega_{\mathcal{F}}) + i_P(\mathcal{J}',\widetilde{C}')
\end{equation}
where $P$ is the infinitely near point of $\widetilde{C}$ in the divisor $E$, the curves $\widetilde{C}'$ and  $\mathcal{J}'$ are the strict transforms of $\widetilde{C}$ and $\mathcal{J}$ by $\pi$ with $\mathcal{J}$ the jacobian curve  of the foliations $\mathcal{F}$ and the foliation defined by $\omega=0$ (see Theorem~\ref{Th:nu_C} and  \cite[Theorem 3.8]{Gom}).

Recall that, given two foliations $\mathcal{F}$ and $\mathcal{G}$ in $(\mathbb{C}^2,0)$, the {\em jacobian curve} $\mathcal{J}_{\mathcal{F},\mathcal{G}}$ of $\mathcal{F}$ and $\mathcal{G}$ is the curve of tangency of the foliations defined by
$$
\omega_{\mathcal{F}}\wedge \omega_{\mathcal{G}}=0
$$
where $\omega_\mathcal{F}$, $\omega_{\mathcal{G}}$ are 1-forms defining $\mathcal{F}$ and $\mathcal{G}$ respectively.

The expression given in \eqref{eq:introduccion1} leads us to introduce next definition. We say that the foliations $\mathcal{F}$ and $\mathcal{G}$ have the {\em $E$-transversality property} if $\mathcal{J}'_{\mathcal{F},\mathcal{G}}$ does not intersect $E$, where $\mathcal{J}'_{\mathcal{F},\mathcal{G}}$ is the strict transform of the jacobian curve $\mathcal{J}_{\mathcal{F},\mathcal{G}}$ by $\pi_C$.  In this situation, given two curves $\widetilde{C}_1,\widetilde{C}_2 \in \operatorname{Curv}_E(\mathcal{F})$, the differential values $\nu_{\widetilde{C}_1}(\omega_{\mathcal{G}})=\nu_{\widetilde{C}_2}(\omega_{\mathcal{G}})$ since
$$
\nu_{\widetilde{C}}(\omega_{\mathcal{G}})=\nu_E(\omega_{\mathcal{F}}\wedge \omega_{\mathcal{G}}) - \nu_E(\omega_\mathcal{F}).
$$
for all $\widetilde{C}\in \operatorname{Curv}_E(\mathcal{F})$. In particular, if $\mathcal{F}$ has the $E$-transversality property with the foliations $\mathcal{G}_i$, for $i=-1,0,1, \ldots,\ell$, with $\ell \in \{1,2,\ldots,s\}$, we obtain that
$$
\nu_{\widetilde{C}}(\omega_i)=\nu_C(\omega_i)=\lambda_i \quad \text{ for } \quad  i=-1,0,1, \ldots,\ell.
$$
Then the values $\lambda_{-1},\lambda_0,\lambda_1,\ldots,\lambda_\ell$ are elements in $\Lambda_{\widetilde{C}} \smallsetminus \Gamma_{\widetilde{C}}$.
We say that a foliation $\mathcal{F}$ satisfies the {\em $\ell$-transversality property} for the divisor $E$ if $\mathcal{F}$ and $\mathcal{G}_\ell$ have the $E$-transversality property. We prove that the $\ell$-transversality property implies the $j$-transversality property for $1 \leq j \leq \ell$ (Proposition~\ref{Prop:l-transv-j-transv}) and  we obtain the following properties concerning the semimodules of differential values of the curves in $\operatorname{Curv}_E(\mathcal{F})$ (see Theorem~\ref{th:valores-dif-tilde-C})
\begin{theoremmain}
Let $\mathcal{F}$ be a totally $E$-dicritical foliation with $C$ as invariant curve and assume that $\mathcal{F}$ satisfies the $\ell$-transversality property for some $\ell \in \{1,2,\ldots,s\}$. Let $\widetilde{C} \in \operatorname{Curv}_E(\mathcal{F})$ and $\mathcal{B}_{\widetilde{C}}=(\lambda_{-1}^{\widetilde{C}},\lambda_{0}^{\widetilde{C}},\lambda_{1}^{\widetilde{C}},\ldots,\lambda_{\ell}^{\widetilde{C}}, \lambda_{\ell+1}^{\widetilde{C}}, \ldots, \lambda_{\tilde{s}}^{\widetilde{C}})$ be the basis of the semimodule $\Lambda_{\widetilde{C}}$ of $\widetilde{C}$. Then $\tilde{s} \geq \ell$ and
$$
\lambda_{j}^{\widetilde{C}} = \lambda_{j}, \qquad \text{ for } j=-1,0,1,\ldots, \ell.
$$
In particular, $\Lambda_j^{\widetilde{C}}=\Lambda_j$ for $j=-1,0,1,\ldots, \ell$, and a minimal standard basis for the curve $\widetilde{C}$ is given by
$$
(\omega_{-1},\omega_0,\omega_1,\ldots,\omega_\ell,\omega_{\ell+1}^{\widetilde{C}}, \ldots, \omega_{\tilde{s}}^{\widetilde{C}} ).
$$
\end{theoremmain}
In \cite{Gom}, the second author showed a particular case of this result proving that all the curves in $\operatorname{Curv}_E(\mathcal{F})$ have the same Zariski invariant provided that $\mathcal{F}$ has the 1-transversality property.

From the previous result we deduce that, if $\mathcal{F}$ satisfies the $\ell$-transversality property for $\ell=-1,0,1,\ldots,s$, then $\Lambda_C \subset \Lambda_{\widetilde{C}}$ for any $\widetilde{C}\in \operatorname{Curv}_E(\mathcal{F})$. Hence, we say that
the foliation $\mathcal{F}$ satisfies {\em total transversality property} when $\mathcal{F}$ and $\mathcal{G}_\ell$ have the $E$-transversality property for $\ell=-1,0,1,\ldots,s$. Note that this property does not depend on the minimal standard basis for $C$. Moreover, we can check if a foliation satisfies the total transversality property in terms of the divisorial value of the 1-form defining the foliation (see Theorem~\ref{Th:total-transv-property})
\begin{theoremmain}
Let $\mathcal F$ be a totally $E$-dicritical foliation with $C$ as invariant curve and
$$
\nu_E(\omega_\mathcal{F})\in [t_{s+1},t_{s+1}+\tilde{t}_{s+1}-t_s),
$$
then the foliation $\mathcal{F}$ satisfies the total transversality property.
\end{theoremmain}
The upper bound in the result above is sharp as shown in Example~\ref{ex:sharp-bound}.

The previous results only determine part of the semimodule of differential values of the curves in $\operatorname{Curv}_E(\mathcal{F})$ as shown in Example~\ref{ex-G-tilde-no-lambda-cte1}. Section~\ref{sec:constant-semimodule} is devoted to determine conditions to assure that a foliation $\mathcal{F}$ is {\em $\Lambda$-constant}, that is, if $\Lambda_{C}=\Lambda_{\widetilde{C}}$ for $C,\widetilde{C} \in \operatorname{Curv}_E(\mathcal{F})$. We give conditions on the divisorial value $\nu_E(\omega_\mathcal{F})$ in terms of the combinatorial of the semimodule $\Lambda_C$ which allow to assure that the foliation $\mathcal{F}$ is $\Lambda$-constant (Theorem~\ref{th:lambda-cte}). More precisely,

\begin{theoremmain}\label{th:main-3}
Let $\mathcal{F}$ be a totally $E$-dicritical foliation in $(\mathbb{C}^2,0)$ with $C$ as invariant curve. If  $\mathcal{F}$ satisfies the total transversality property and
$$\nu_E(\omega_\mathcal{F}) \leq c(\Gamma_C)-c(\Lambda_C)+2(n+m)-1,$$
then the foliation $\mathcal{F}$ is $\Lambda$-constant.
\end{theoremmain}
Note that $c(\Gamma_C)-c(\Lambda_C)+2(n+m)-1=t_{s+1}+n+m+1$. In particular, if we denote $t_\ast=\min \{\tilde{t}_{s+1}-t_s,n+m+1\}$,  we get that the foliation $\mathcal{F}$ is $\Lambda$-constant provided that $\nu_E(\omega_\mathcal{F}) \in [t_{s+1},t_{s+1}+t_*)$.

At the end of Section~\ref{sec:constant-semimodule}, we  discuss different situations where the hypothesis in Theorem~\ref{th:main-3} are satisfied  and  we construct families of $\Lambda$-constant foliations.

In the appendices at the end of the article we include some technical results concerning the structure of increasing cuspidal semimodules that we use along the paper.

\begin{Agradecimientos} We would like to thank  Prof. Felipe Cano for fruitful conversations and suggestions during the preparation of this work.
\end{Agradecimientos}
\section{Local invariants of curves and foliations}

Let $C$ be  an irreducible plane curve $C$ in $(\mathbb{C}^2,0)$. In this section, we will introduce some invariants associated to the curve $C$ and we will see how these invariants are related to invariants associated to foliations in $(\mathbb{C}^2,0)$.

Consider   a primitive parametrization $\gamma(t)$ of $C$. Given a holomorphic function $h \in \mathcal{O}_{\mathbb{C}^2,0}$,
the {\em differential value\/} $\nu_C(h)$ of $h$ is defined as
$$\nu_C(h)=\text{ord}_t(h(\gamma(t))).
$$ Note that, if $h(0)=0$ and we consider the curve $H$ in $(\mathbb{C}^2,0)$ defined by $h=0$, then the intersection multiplicity $i_0(C,H)$ of the curves $C$ and $H$ at the origin is equal to $i_0(C,H)=\nu_C(h)$.

The {\em semigroup} $\Gamma_C$ of the curve $C$ is defined by
$$
\Gamma_C=\{ \nu_C(h) \ : \ h \in \mathcal{O}_{\mathbb{C}^2,0}\}.
$$
The semigroup $\Gamma_C$ is a complete topological invariant of an irreducible curve $C$ (see \cite{Zar}). 
Recall that $\mathbb{N} \smallsetminus \Gamma_C$ is finite and hence there exists an integer $c({\Gamma_C})$ such that any non-negative integer greater or equal to $c({\Gamma_C})$ belongs to $\Gamma_C$. The number $c({\Gamma_C})$ is the {\em conductor} of the semigroup. Recall that $c({\Gamma_C})$ coincides with the Milnor number $\mu(C)$ of the curve $C$.

Let $\omega \in \Omega^1_{(\mathbb{C}^2,0)}$ be a 1-form  in $(\mathbb{C}^2,0)$. The {\em differential value\/} of $\omega$ is defined as
$$\nu_C(\omega)=\text{ord}_t(\alpha(t))+1,
$$
where we write $\gamma^*\omega=\alpha(t) dt$. In particular, given $h \in \mathcal{O}_{\mathbb{C}^2,0}$ with $h(0)=0$, we have that $\nu_C(h)=\nu_C(dh)$.

The {\em set of differential values} $\Lambda_C$ of the curve $C$ is  the set
$$
\Lambda_C=\{ \nu_C(\omega) \ : \ \omega \in \Omega^1_{(\mathbb{C}^2,0)}\}.
$$
By the properties of the semigroup, we obtain that $\Gamma_C \smallsetminus \{0\} \subset \Lambda_C$. Moreover,
we have that $\nu_C(h \omega)=\nu_C(h) + \nu_C(\omega)$ for $h \in \mathcal{O}_{\mathbb{C}^2,0}$ and a 1-form  $\omega \in \Omega^1_{(\mathbb{C}^2,0)}$. Hence, we get $\Lambda_C +\Gamma_C \subset \Lambda_C$ and consequently, the set of differential values $\Lambda_C$ is a {\em $\Gamma_C$-semimodule}. Since $\Lambda_C \smallsetminus \Gamma_C$ is a finite set, then $\Lambda_C$ has also a conductor $c(\Lambda_C)$ such that any non-negative integer greater or equal to $c({\Lambda_C})$ belongs to $\Lambda_C$.

As we mention before, the semimodule of differential values $\Lambda_C$ is the most relevant discrete analytic invariant of the curve $C$.

Let us see how the above invariants of plane curves are related with local invariants of foliations defined in $(\mathbb{C}^2,0)$. Recall that a foliation $\mathcal{F}$ in $(\mathbb{C}^2,0)$ is defined by $\omega=0$, where  $\omega$ is a saturated 1-form in $\Omega^1_{(\mathbb{C}^2,0)}$.  A curve $C$ is a {\em separatrix\/} (or an {\em invariant curve}) of the foliation $\mathcal F$ if and only if $\gamma^* \omega=0$, where $\gamma(t)$ is a primitive parametrization of $C$. A foliation $\mathcal{F}$ in $(\mathbb{C}^2,0)$ is {\em dicritical} if $\mathcal{F}$ has infinitely many invariant curves.

\begin{Remark}\label{rmk:dicritical}
In dimension two,  dicriticalness can be stated in terms of the minimal reduction of singularities of the foliation (see \cite{Sei,Can-C-D}). Given a foliation $\mathcal{F}$ in $(\mathbb{C}^2,0)$, there is a morphism $\pi: M \to (\mathbb{C}^2,0)$ composition of a finite number of  blow-ups centered at points such that the strict transform $\pi^* \mathcal{F}$ of $\mathcal{F}$ by $\pi$ satisfies that
\begin{itemize}
  \item each irreducible component $E$ of the exceptional divisor $\pi^{-1}(0)$ is either invariant by $\pi^*\mathcal{F}$ or transversal to $\pi^*\mathcal{F}$. In the second case, $E$ is a dicritical component;
  \item all the singular points of $\pi^*\mathcal{F}$ are simple and do not belong to a dicritical component of the exceptional divisor.
\end{itemize}
A foliation $\mathcal{F}$ in $(\mathbb{C}^2,0)$ is non-dicritical if all the irreducible components of the exceptional divisor are invariant ones. Otherwise, the foliation $\mathcal{F}$ is a dicritical foliation.

For the sake of completeness, let us recall the notion of simple singularity. The origin is a {\em simple singularity} of a foliation $\mathcal{F}$  in $(\mathbb{C}^2,0)$ if there are local coordinates
$(x, y)$ at $(\mathbb{C}^2,0)$ such that $\mathcal{F}$ is defined by a 1-form
$$y(\lambda  + a(x, y))dx -x(\mu + b(x, y))dy =0$$
with $a(0) = b(0) = 0$, $\mu \neq 0$ and $\lambda/\mu \in \mathbb{Q}_{>0}$. If $\lambda = 0$, the singularity is called a
{\em saddle-node} singularity.

Foliations without saddle-node singularities after reduction of singularities are called {\em generalized curve foliations} (see \cite{Cam-S-LN}).
\end{Remark}

Let us introduce some invariants of the foliation $\mathcal{F}$ relative to an irreducible curve $C$ in $(\mathbb{C}^2,0)$. Let $(x,y)$ be coordinates in $(\mathbb{C}^2,0)$. The foliation $\mathcal{F}$ is defined by a 1-form $\omega=0$, where $\omega=A(x,y)dx+ B(x,y)dy$  with $A,B \in \mathbb{C}\{x,y\}$ and $\gcd(A,B)=1$. The {\em multiplicity} $\nu_0(\mathcal{F})$ of the foliation  $\mathcal{F}$ at the origin is the minimum of the orders $\nu_0(A)$, $\nu_0(B)$  at the origin
of the coefficients of $\omega$. The {\em Milnor number} $\mu_0(\mathcal{F})$ is given by
$$\mu_0(\mathcal{F})=i_0(A,B).$$

Consider now a primitive parametrization $\gamma(t)=(x(t),y(t))$ of the curve $C$. If $C$ is an invariant curve of $\mathcal{F}$,
the {\em Milnor number} $\mu_0({\mathcal F},C)$ of $\mathcal F$ {\em along\/} $C$ is given by
\begin{equation}
\label{multiplicidaderelativa1}
 \mu_0({\mathcal F},C) =
\begin{cases}
 {\rm ord}_{t}(B(\gamma(t))) -  {\rm ord}_{t}(x(t)) + 1 \ \ \ \mbox{if $x(t) \neq 0$}   \\
 {\rm ord}_{t}(A(\gamma(t))) -  {\rm ord}_{t}(y(t)) + 1 \ \ \ \mbox{if $y(t) \neq 0$}
\end{cases}
\end{equation}
(this number is also called {\em multiplicity of the vector field $\mathbf{v}=B(x,y) \frac{\partial}{\partial x} - A(x,y) \frac{\partial}{\partial y}$ along the curve $C$}, see~\cite[p. 152-153]{Cam-S-LN}). 
If the curve  $C$ is not  invariant by  $\mathcal{F}$, we can consider the {\em tangency order} $\tau_0({\mathcal F},C)$ defined by
\begin{equation}\label{def-tangencia}
\tau_0({\mathcal F},C)= \text{ord}_t (\alpha(t))
\end{equation}
where $\gamma^* \omega= \alpha(t) dt$ (see \cite[p. 167 ]{Cam-S-LN} when the curve $C$ is a non-singular curve or \cite{Can-C-M} for the general
case). From the definition of differential value, we get that
\begin{equation}\label{valordif-tangencia}
\nu_C(\omega)=\tau_0({\mathcal F},C)+1.
\end{equation}
The equality $\nu_C(h)=\nu_C(dh)$ can be stated in terms of the invariants associated to foliations as follows. Assume that $h \in \mathbb{C}\{x,y\}$ is reduced and consider the hamiltonian foliation $\mathcal{G}_h$ defined by $dh=0$. Let $S=S_{\mathcal{G}_h}$ be the curve defined by $h=0$ which is the curve of separatrices of the foliation $\mathcal{G}_h$. Then we have
$$
i_0(S,C) =\tau_0(\mathcal{G}_h,C) +1.
$$
In fact, the previous equality is a particular case of a more general result. If $\mathcal{F}$ is a non-dicritical foliation  and  $S_\mathcal{F}$ is the curve of separatrices of $\mathcal{F}$, then we have
$$
i_0(S_\mathcal{F},C)\leq \tau_0(\mathcal{F},C)+1
$$
and the equality holds if and only if $\mathcal{F}$ is a second type foliation (see \cite[Corollary 1]{Can-C-M}). Consequently, if we are interested in the description of the set $\Lambda_C \smallsetminus \Gamma_C$, we need to study the
values $\nu_C(\omega)$ for 1-forms such that the foliation defined by $\omega=0$ is either
dicritical or it is not a second type foliation (see \cite{Cor-Handbook}).

From the results in \cite{Cor-2025}, we obtain that the invariants introduced in \eqref{multiplicidaderelativa1}  and \eqref{def-tangencia} are related with the jacobian curve of two foliations. Given two germs of foliations $\mathcal F$ and $\mathcal G$
 in $({\mathbb C}^2,0)$,
defined by the 1-forms $\omega=0$ and $\eta=0$, the {\em jacobian curve} ${\mathcal J}_{{\mathcal F},{\mathcal G}}$   of $\mathcal F$ and $\mathcal G$ is the curve given by
$$J_{\mathcal{F},\mathcal{G}}=0$$
where $\omega \wedge \eta= J_{\mathcal{F},\mathcal{G}} (x,y) dx \wedge dy$. Note that the jacobian curve is the curve of tangency between the foliations $\mathcal{F}$ and $\mathcal{G}$.

Next proposition relates the intersection multiplicity $i_0(\mathcal{J}_{\mathcal{F},\mathcal{G}},C)$ of the jacobian curve ${\mathcal J}_{{\mathcal F},{\mathcal G}}$ with a separatrix $C$ of one of the foliations with the invariants defined in \eqref{multiplicidaderelativa1}  and \eqref{def-tangencia}.
\begin{Proposition}[Proposition B.1, \cite{Cor-2025}]\label{prop:int-sep}
Let $\mathcal{F}$ and $\mathcal{G}$ be two foliations in $({\mathbb C}^2,0)$. Assume that $\mathcal F$ and $\mathcal G$ have no common separatrix.
If $C$ is an irreducible separatrix of $\mathcal G$, we have
\begin{equation}\label{ec:i-mu-tau}
i_0(\mathcal{J}_{\mathcal{F},\mathcal{G}},C) = \mu_0({\mathcal G},C)+\tau_0({\mathcal F},C).
\end{equation}
\end{Proposition}
The result above also holds when $\mathcal{F}$ and $\mathcal{G}$ are dicritical foliations.

\medskip
From equality~\eqref{valordif-tangencia}, if $\omega$ is a 1-form which defines a foliation $\mathcal{F}$ and we consider any foliation $\mathcal{G}$ having $C$ as invariant curve, we can compute the differential value $\nu_C(\omega)$ as
$$
\nu_C(\omega)=i_0(\mathcal{J}_{\mathcal{F},\mathcal{G}},C)- \mu_0(\mathcal{G},C) +1.
$$
 Consequently, the study of the behaviour of jacobian curves of foliations seems interesting when dealing with the differential values of a curve $C$.

\medskip

Let us  consider a reduced equation $f=0$  of $C$ with $f \in \mathbb{C}\{x,y\}$ and the hamiltonian foliation $\mathcal{G}_f$ given by $df=0$. If we assume that $x(t) \neq 0$ and $m_0(C)=\text{ord}_t(x(t))$,  from Proposition~\ref{prop:int-sep}, we get that
\begin{align*}
  i_0(\mathcal{J}_{\mathcal{F},\mathcal{G}_f},C) & =\mu_0(\mathcal{G}_f,C)+\tau_0(\mathcal{F},C) \\
  & =  \text{ord}_t\left(\tfrac{\partial f}{\partial y}(\gamma(t))\right) -m_0(C)+ 1 + \tau_0(\mathcal{F},C) \\
   & = \mu_0(C)+  \tau_0(\mathcal{F},C)
\end{align*}
since $\text{ord}_t\left(\tfrac{\partial f}{\partial y}(\gamma(t))\right)=\mu_0(C)+m_0(C)-1$ where $\mu_0(C)$ denotes the Milnor number of $C$ given by $\mu_0(C)=dim_\mathbb{C} \mathbb{C}\{x,y\}/(\tfrac{\partial f}{\partial x},\tfrac{\partial f}{\partial y})$ (see also Remark 4.2 in \cite{Cor-H-H}). In particular, we obtain that
$$
\nu_C({J}_{\mathcal{F},\mathcal{G}_f}) = \mu_0(C)+  \nu_C(\omega)-1.
$$
Note that if $\mathcal{G}$ is a generalized curve foliation such that $C=S_{\mathcal{G}}$ is the curve of separatrices of  $\mathcal{G}$ with $C$ irreducible,  by Corollary 1.3.11 of \cite{Cor-Handbook}, we obtain that $\mu_0(\mathcal{G},C)=\mu_0(\mathcal{G}_f,C)$ and hence
$$
i_0(\mathcal{J}_{\mathcal{F},\mathcal{G}},C) =\mu_0(\mathcal{G})+  \tau_0(\mathcal{F},C),
$$
where $\mu_0(\mathcal{G})$ denotes the Milnor number of the foliation $\mathcal{G}$.

\section{Divisorial values}\label{sec:div-ord}
Let us consider  a sequence $\mathcal{S}$ of point blow-ups
\begin{equation}\label{eq:suc-explosiones}
(M_0,P_0) \overset{\pi_1}{\longleftarrow} (M_1,P_1) \overset{\pi_{2}}{\longleftarrow} \cdots \overset{\pi_{N-1}}{\longleftarrow} (M_{N-1},P_{N-1}) \overset{\pi_N}{\longleftarrow} M_{N}
\end{equation}
where $(M_0,P_0)=(\mathbb{C}^2,0)$ and the morphism $\pi_k$ is the blow-up with center at $P_{k-1}$. We denote the exceptional divisor of $\pi_k$ as $E_k^k=\pi_k^{-1}(P_{k-1})$.
For any $1 \leq j <k$, we denote $E_{j}^k \subset M_k$ the strict transform of $E_{j}^{k-1}$ by $\pi_k$ and $D_k=E_1^k \cup E_2^k \cup \cdots \cup E_k^k$. We assume that $P_k \in D_k$ for $k=1,2, \ldots, N-1$.

We consider also the intermediary morphisms
\begin{equation}\label{eq:explosiones-intermedias}
\sigma_k : M_k  \to (M_0,P_0), \  \qquad \rho_k: M_N \to (M_k,P_k)
\end{equation}
given by $\sigma_k=\pi_1 \circ \pi_2 \circ \cdots \circ \pi_k$ and $\rho_k= \pi_{k+1} \circ \pi_{k+2} \circ \cdots \circ \pi_N$, for any $k=1,2,\ldots, N-1$. Note that $D_k=\sigma_k^{-1}(0)$.

Hence, from the sequence $\mathcal{S}$ in \eqref{eq:suc-explosiones},   we obtain a morphism
\begin{equation}\label{eq:pi}
\pi: (M,D) \to (\mathbb{C}^2,0)
\end{equation} where $\pi=\pi_1 \circ \pi_2 \circ \cdots \pi_N$, $M=M_N$ and $D=D_N$.
Let us  denote $E=E_N^N$.

This section is devoted to introduce the divisorial value relative to the divisor $E$. We fix the notation above along all the section.
\medskip

\paragraph{\bf Divisorial value of a function.} Consider a holomorphic function $h$ in $(M,D)$ defined globally in $E \subset D$. Take a point $P \in E$ and choose a reduced local equation $u=0$ of the germ $(E,P)$, the {\em divisorial value} $\nu_E(h)$ of $h$ is given by
$$
\nu_E(h)=\max\{\ell \in \mathbb{Z} \ : \ u^{-\ell} h \in \mathcal{O}_{M,P}\}.
$$
Consider now a germ of holomorphic function $h \in {\mathcal O}_{M_k,P_k}$ with $k \in \{0,1,\ldots, N-1\}$. Then $\rho_{k}^* h$ is a germ of function in $(M,D)$ globally defined in $E$. We define the {\em divisorial value} $\nu_E(h)$  by $\nu_E(h)=\nu_E(\rho_k^* h)$. Note that this divisorial value coincides with the divisorial valuation associated to the divisor $E$.

A {\em curvette} $\tilde{S}$ of the divisor $E$ is a non-singular curve transversal to $E$ at a non-singular point of $\pi^{-1}(0)$. The projection $S=\pi(\tilde{S})$ is a germ of plane curve in $(\mathbb{C}^2,0)$ and we say that $S$ is an $E$-{\em curvette}. We will denote by $\text{Curv}(E)$ the set of $E$-curvettes. Note that all the curves in $\text{Curv}(E)$ have the same topological type. The divisorial value can be computed in terms of  intersection multiplicities as follows:
\begin{Lemma}[Theorem 7.2, \cite{Spiv-1990}] \label{lemma:nu_E_funcion} Consider $h \in \mathcal{O}_{\mathbb{C}^2,0}$ and let $H$ be the germ of curve given by $h=0$ in $(\mathbb{C}^2,0)$. Then $\nu_E(h)$ is given by
$$
\nu_E(h)=\min\{i_0(H,C) \ : \ C \in \text{\rm Curv}(E)\}.
$$
\end{Lemma}
Hence,  using Noether's formula (see for instance Theorem 3.3.1 of \cite{Cas-2000}), we get
$$
\nu_E(h)=\sum_{i=1}^{N-1} m_{P_i}(H^i) m_{P_i}(C^i)
$$
where  $C \in \text{Curv}(E)$, the curves $H^i$,  ${C}^i$ are the strict transforms of the curves ${H}$ and ${C}$ by the morphism $\sigma_i$ at $P_i$ and $m_{P_i}(C^i)$ is the multiplicity of the curve $C^i$ at $P_i$.

\begin{Remark}\label{Rm-nuC-nuE}
With the notations above, if $C \in \text{Curv}(E)$ and $P$ is the infinitely near point of $C$ in the divisor $E$, we get that
$$
\nu_C(h)=\nu_E(h) + i_P(C',H'),
$$
where $C'$ and $H'$ are the strict transforms of $C$ and $H$ by $\pi$ respectively. Consequently, $\nu_E(h) \leq \nu_C(h)$.
\end{Remark}

\medskip

\paragraph{\bf Divisorial value of a 1-form.}
Let $\omega$ be a 1-form defined globally in $E \subset D$ and choose a reduced equation $u=0$ of $E$ at $P$ as before. The {\em divisorial value} $\nu_E(\omega)$ is given by
$$
\nu_E(\omega)=\max\{\ell \in \mathbb{Z}\ : \ u^{-\ell} \omega \in \Omega_{M,P}^1[\text{log} E]\}.
$$
where $\Omega_{M,P}^1[\text{log} E]$ is the $\mathcal{O}_{M,P}$-module of germs of logarithmic 1-forms along $E$ at $P$.

Given a germ of 1-form $\omega \in \Omega_{M_k,P_k}^1$ with $k \in \{0,1,\ldots, N-1\}$, we define the {\em divisorial value} $\nu_E(\omega)$  by $\nu_E(\omega)=\nu_E(\rho_k^* \omega)$.

\begin{Remark}
Note that if $h  \in \mathcal{O}_{\mathbb{C}^2,0}$  with $h(0)=0$, then we have  $\nu_E(h)=\nu_E(dh)$ (see \cite[Corollary 3.5]{Can-C-SS-2023}).
\end{Remark}
\medskip
Consider a sequence  of blow-ups as the one given in \eqref{eq:suc-explosiones}. For any $P \in D_k=E_1^k \cup E_2^k \cup \cdots \cup E_k^k$, we denote $e(P)=\sharp \{j \ : \ P \in E_j^k\}$. Note that $e(P) \in \{1,2\}$. We say that $P$ is a {\em free point} if $e(P)=1$ and $P$ is a {\em corner point} if $e(P)=2$.

Given a curve $C \in \text{Curv}(E)$, the morphism $\pi: M \to (\mathbb{C}^2,0)$ is a reduction of singularities of the curve $C$ and the points $P_1,\ldots, P_{N-1}$ are infinitely near points of the curve $C$.

 Assume that $P_1,P_2,\ldots, P_f$ are infinitely near points of $C$ which are free points  and $P_{f+1}$ is a corner point. We say that a non-singular branch $Y$ in $(\mathbb{C}^2,0)$ {\em has maximal contact with} $C$ if and only if $P_k$ is an infinitely near point of $Y$ for $j=1,2,\ldots,f$. A system of coordinates $(x,y)$ in $(\mathbb{C}^2,0)$ is  {\em adapted to the divisor} $E$ if and only if the curve $y=0$ has maximal contact with any curve $C \in \text{Curv}(E)$.

\begin{Remark}\label{rm:parametrizacion-coord-adapt}
Let $\{\beta_0,\beta_1,\ldots,\beta_g\}$ be the characteristic exponents of a curve $C \in \text{Curv}(E)$. If $(x,y)$ are coordinates in $(\mathbb{C}^2,0)$ adapted to the divisor $E$, a primitive parametrization  $\gamma(t)=(x(t),y(t))$ of $C$ is given by
\begin{equation*}
  \left\{
\begin{aligned}
  x(t) & = t^{\beta_0} \\
  y(t) & = \sum_{i \geq \beta_1} a_i t^{i}
\end{aligned}
\right.
\end{equation*}
Note that the number $f$ of free infinitely near points  of $C$ in $\mathcal{S}$  is equal to $\left[\tfrac{\beta_1}{\beta_0}\right]$. In particular, if $Y$ is a non-singular curve in $(\mathbb{C}^2,0)$ with maximal contact with $C$ then
$$
i_0(C,Y)=\beta_1.
$$
Moreover, we get that
\begin{align*}
\nu_E(x)&=\nu_E(dx)=\beta_0, \quad \nu_E(y)=\nu_E(dy)=\beta_1, \\
\nu_C(x)&=\nu_C(dx)=\beta_0, \quad \nu_C(y)=\nu_C(dy)=\beta_1.
\end{align*}
\end{Remark}

Let $\omega$ be a 1-form at $(\mathbb{C}^2,0)$. Take  coordinates $(x,y)$ at $(\mathbb{C}^2,0)$ adapted to $E$ and  write the differential 1-form $\omega=A(x,y) dx + B(x,y) dy$.  By Proposition 3.4 in \cite{Can-C-SS-2023}, the  divisorial value $\nu_E(\omega)$ is equal to
\begin{equation}\label{ec:nu_E-omega}
\nu_E(\omega)=\min \{\nu_E(xA),\nu_E(yB)\}
\end{equation}
 (see Section 3.1 of \cite{Can-C-SS-2023} for a more detailed study of the divisorial value of a 1-form).
\begin{Remark}
From the definitions of differential and divisorial values, if $C \in \text{Curv}(E)$, we have
  $$
  \nu_E(\omega) \leq \nu_C(\omega)
  $$
  for any 1-form $\omega \in \Omega^1_{(\mathbb{C}^2,0)}$.
\end{Remark}

\medskip

\paragraph{\bf Divisorial value of a 2-form.} Consider now a 2-form $\eta$ defined globally in $E \subset D$ and choose a reduced equation $u=0$ of $E$ at $P$ as before. The {\em divisorial value} $\nu_E(\eta)$ is given by
$$
\nu_E(\eta)=\max\{\ell \in \mathbb{Z}\ : \ u^{-\ell} \eta \in \Omega_{M,P}^2[\text{log} E]\}.
$$
where $\Omega_{M,P}^2[\text{log} E]$ is the $\mathcal{O}_{M,P}$-module of germs of logarithmic 2-forms along $E$ at $P$. As before, given a 2-form $\eta \in \Omega_{M_k,P_k}^2$ with $k \in \{0,1,\ldots, N-1\}$, we define the {\em divisorial value} $\nu_E(\eta)$  by $\nu_E(\eta)=\nu_E(\rho_k^* \eta)$. Next result gives the computation of the divisorial value of a 2-form in adapted coordinates.

\begin{Lemma}
Let $(x,y)$ be coordinates in $(\mathbb{C}^2,0)$  adapted to $E$. Given  a holomorphic 2-form  $\eta$ in $(\mathbb{C}^2,0)$, the divisorial value $\nu_E(\eta)$ is given by
$$
\nu_E(\eta)=\nu_E(xyh) = \nu_E(h) + \nu_E(dx) + \nu_E(dy).
$$
where  $\eta = xy h(x,y) \ \frac{dx}{x} \wedge \frac{dy}{y}$ with $h \in \mathbb{C}\{x,y\}$.
\end{Lemma}
\begin{proof}
We prove the result by induction on the number $N$ of blow-ups  in the sequence $\mathcal{S}$ given in~\eqref{eq:suc-explosiones}.

Assume $N=1$ and take coordinates $(x_1,y_1)$ in the first chart of the blow-up such that $\pi_1(x_1,y_1)=(x_1,x_1 y_1)$ and the divisor $E=E_1$ is given by $x_1=0$. Hence, we have
$$
\pi_1^* \eta= x_1^2 y_1 h(x_1,x_1y_1) \frac{dx_1}{x_1} \wedge \frac{dy_1}{y_1}
$$
and we get
$$
\nu_E(\eta)=\nu_E(\pi_1^*\eta)=\nu_E (\pi_1^* (xy h))
$$
as wanted. Assume now $N >1$ and write $\pi= \pi_1 \circ \rho_1$ with $\pi_1$ and $\rho_1$ given in \eqref{eq:explosiones-intermedias}.  Then, we can write
$$
\pi^* \eta= \rho_1^* \pi_1^* \eta
$$
By definition, we have
$$\nu_E(\eta)= \nu_E(\pi_1^*\eta).
$$
The induction hypothesis gives $\nu_E(\pi_1^*\eta)=\nu_E(xyh)$ which finishes the proof.
\end{proof}
As a consequence of the previous lemma, we get the following result.
\begin{Corollary} \label{lema-2-div-valor}
Let $\mathcal{F}$ and $\mathcal{G}$ be  two foliations  defined by 1-forms $\omega_{\mathcal{F}}=0$ and $\omega_{\mathcal{G}}=0$ and write $\omega_\mathcal{F} \wedge \omega_{\mathcal{G}}=J_{\mathcal{F},\mathcal{G}} (x,y)  dx \wedge dy$, then
$$
\nu_E(\omega_\mathcal{F} \wedge \omega_\mathcal{G})= \nu_E({J}_{\mathcal{F},\mathcal{G}}) + \nu_E(dx) + \nu_E(dy).
$$
\end{Corollary}

\subsection{Divisorial value for a cuspidal divisor}
We say that the sequence of blow-ups $\mathcal{S}$ in \eqref{eq:suc-explosiones}  is a \textit{cuspidal sequence} if any curve $C \in \text{Curv}(E)$ is an irreducible plane curve with one Puiseux pair (a cusp) and $\pi$ is the minimal reduction of singularities of $C$.

The sequence $\mathcal{S}$ is a cuspidal sequence if $P_k \in E_k^k$ for any $k=1,2,\ldots, N-1$ and $e(P_{k-1}) \leq e(P_k)$ for $2 \leq k \leq N-1$. The divisor $E=E_N^N$ is called a {\em cuspidal divisor} (see \cite{Can-C-SS-2023}). If any curve $C \in \text{Curv}(E)$ is a cusp with Puiseux pair $(m,n)$  where $2 \leq n < m$ and $\gcd(n,m)=1$, we also say that $E$ is a $(m,n)$-cuspidal divisor. We will also say that $\mathcal{S}$ is a $(m,n)$-cuspidal sequence. When $E$ is a cuspidal divisor, we will also denote $\operatorname{Cusp}(E)=\operatorname{Curv}(E)$.

In particular, for cuspidal divisors, we get that the divisorial value of a function can be computed as follows
\begin{Lemma}[\cite{Can-C-SS-2023}]\label{lemma:nu_E_ij}
Let $E$ be a $(m,n)$-cuspidal divisor and consider $(x,y)$  coordinates  in $(\mathbb{C}^2,0)$ adapted to the divisor $E$. Given a germ $h \in \mathbb{C}\{x,y\}$, the divisorial value is equal to
$$
\nu_E(h)=\min \{n i + mj \ : \ h_{ij} \neq 0\}
$$
where $h(x,y)=\sum_{i,j} h_{ij} x^i y^j$.
\end{Lemma}

\medskip
Consider a 1-form $\omega \in \Omega_{(\mathbb{C}^2,0)}^1$. Given coordinates $(x,y)$ in $(\mathbb{C}^2,0)$, the 1-form can be written as
\begin{equation}\label{eq:omega-ij}
\omega=\sum_{i,j} \omega_{ij} \qquad \text{ with }
\qquad \omega_{ij}=A_{ij}x^{i-1}y^j dx + B_{ij} x^i y^{j-1}dy.
\end{equation}
 We denote $\Delta(\omega)=\{(i,j) \ : \omega_{ij}\neq 0\}$. The {\em Newton polygon} $\mathcal{N}(\omega;x,y)=\mathcal{N}(\omega)$ is given by the convex envelop of $\Delta(\omega)+(\mathbb{R}_{\geq 0})^2$. If $\mathcal{F}$ is a foliation defined by a 1-form $\omega=0$, we define the Newton polygon $\mathcal{N}(\mathcal{F})=\mathcal{N}(\mathcal{F};x,y)=\mathcal{N}(\omega;x,y)$.

Given a rational number $\alpha \in \mathbb{Q}$, the {\em initial part} of $\omega$ with {\em weight} $\alpha$ is given by
$$
\text{In}_{\alpha}(\omega;x,y)= \sum_{i + \alpha j=k} \omega_{ij}
$$
where $i + \alpha j=k$ is the equation of the first line of slope $-1/\alpha$ which intersects the Newton polygon of $\omega$ in the coordinates $(x,y)$ (see \cite{Cor-2003,Cor-2025} for more details).

Let $E$ be a $(m,n)$-cuspidal divisor and  take $(x,y)$ coordinates in $(\mathbb{C}^2,0)$ adapted to the divisor $E$. From equality in~\eqref{ec:nu_E-omega} and Lemma~\ref{lemma:nu_E_ij}, we get that
$$
\nu_E(\omega)=\min \{ n i + m j \ : \ \omega_{ij} \neq 0\}
$$
where $\omega$ is written as in expression~\eqref{eq:omega-ij}.
In particular, we obtain
$$
\text{In}_{\tfrac{m}{n}}(\omega;x,y)= \sum_{ni + m j=\nu_E(\omega)} \omega_{ij}.
$$
In a similar way, given a holomorphic function $h \in \mathbb{C}\{x,y\}$, we can write $h(x,y)=\sum_{i,j} h_{ij} x^i y^j$ and we define $\Delta(h)=\{(i,j) \ : \ h_{ij} \neq 0\}$. The {\em Newton polygon} $\mathcal{N}(h;x,y)=\mathcal{N}(h)$ is the convex envelop of $\Delta(h) + (\mathbb{R}_{\geq 0})^2$. If $h=0$ is a reduced equation of a curve $H$, then  the Newton polygon of the curve $H$ is given by $\mathcal{N}(H)=\mathcal{N}(h)$.

\begin{Remark} \label{monomio-nm}
If $\nu_E(\omega) <nm$, then there are unique $a,b \in \mathbb{Z}_{\geq 0}$, with $ab \neq 0$, such that $\nu_E(\omega)=n a + m b$. Consequently, the initial part of $\omega$ is given by
$$
\text{In}_{\tfrac{m}{n}}(\omega;x,y)=x^a y^b \left( \lambda \frac{dx}{x} + \mu \frac{dy}{y} \right).
$$
A similar argument holds for the case of a function $f \in \mathbb{C}\{x,y\}$.
Therefore, if $\nu_E(h)< nm$, we get $\nu_E(h)=\nu_C(h)$. However, for 1-forms we can have $\nu_E(\omega)<nm$ and $\nu_E(\omega) < \nu_C(\omega)$; in this case, the initial part of $\omega$ is given by
\begin{equation}\label{eq:1-forma-dicritica}
\text{In}_{\tfrac{m}{n}}(\omega;x,y)=\mu x^a y^b \left( m \frac{dx}{x} -n \frac{dy}{y} \right)
\end{equation}
with $\mu \neq 0$, and $a,b \geq 1$ such that $\nu_E(\omega)=n a + mb$.
\end{Remark}

A differential 1-form with initial part given by the expression in \eqref{eq:1-forma-dicritica} is called a {\em resonant differential 1-form}. Properties of foliations defined by resonant 1-forms will be described in Section~\ref{sec:E-dicritical}.
\medskip

\section{Totally dicritical foliations}\label{sec:E-dicritical}
Consider a sequence $\mathcal{S}$ of point blow-ups
\begin{equation}\label{eq:suc-explosiones-totaldicritical}
(\mathbb{C}^2,0)=(M_0,P_0) \overset{\pi_1}{\longleftarrow} (M_1,P_1) \overset{\pi_{2}}{\longleftarrow} \cdots \overset{\pi_{N-1}}{\longleftarrow} (M_{N-1},P_{N-1}) \overset{\pi_N}{\longleftarrow} M_{N}
\end{equation}
where  $\pi_k$ is the blow-up with center at $P_{k-1}$. The exceptional divisor of $\pi_k$  is denoted by $E_k^k=\pi_k^{-1}(P_{k-1})$, the strict transform of $E_j^{k-1}$ by $\pi_k$ is denoted by $E_j^k$ with $1 \leq j < k$ and $D_k=E_1^k \cup E_2^k \cup \cdots \cup E_k^k$ for $k=1,2,\ldots, N-1$. We also denote $M=M_N$, $D=D_N$, $E=E_N^N$ and $\pi: M \to (\mathbb{C}^2,0)$ with $\pi=\pi_1 \circ \cdots \circ \pi_N$.

This section is devoted to prove some properties of totally $E$-dicritical foliations. These foliations where introduced in \cite{Can-C-SS-2023} when studying the behaviour of the foliations defined by the 1-forms of a standard basis for a cusp, but this notion can be extended to non-cuspidal  divisors.

\begin{Definition}
 A foliation $\mathcal F$  is a {\em totally $E$-dicritical foliation} if the strict transform $\pi^*\mathcal{F}$ by $\pi$ is transversal to $E$ and has normal crossings with $D$ at each point of $E$.
\end{Definition}
Hence, for each free point $P \in E$, there exists a curve $C_P \in \operatorname{Curv}(E)$ such that $C_P$ is an invariant curve of $\mathcal{F}$. We will denote $\operatorname{Curv}_E(\mathcal{F})$ the set of these curves.
\begin{Remark}
Note that the morphism $\pi$ gives a reduction of singularities of any of the curves $C \in \operatorname{Curv}_E(\mathcal{F})$ but $\pi$ can be different from the minimal reduction of singularities of $C$. In the case that $\pi$ is equal to the minimal reduction of singularities of one curve $C$ in $\operatorname{Curv}_E(\mathcal{F})$, then $\pi$ is also the minimal reduction of singularities of all the curves in $\operatorname{Curv}_E(\mathcal{F})$.
\end{Remark}

In the first part of this section we recall some results showing the relationship between resonant 1-forms and totally dicritical foliations when $E$ is a cuspidal divisor. In the second part, we describe some properties of totally $E$-dicritical foliations which hold for any divisor $E$. In particular, given a curve $\widetilde{C} \in \operatorname{Curv}_E(\mathcal{F})$ with $\mathcal{F}$ a totally $E$-dicritical foliation and any 1-form $\omega$,  we prove a formula which allows to compute the differential value $\nu_{\widetilde{C}}(\omega)$ in terms of divisorial values relatives to the divisor $E$ and the jacobian curve of the foliation $\mathcal{F}$ and the one defined by $\omega=0$ (see Theorem~\ref{Th:nu_C}). This result will be key to describe the semimodule of differential values of the curves $\widetilde{C} \in \operatorname{Curv}_E(\mathcal{F})$.

\subsection{Cuspidal dicritical foliations}
Let us first introduce the notion of cuspidal dicritical foliations (see also \cite{Gom,Can-C-SS-2023}).

\begin{Definition}\label{def:cuspidal-dicritical}
A foliation $\mathcal{F}$ is a {\em cuspidal dicritical foliation} if $\mathcal{F}$ is  totally $E$-dicritical for a cuspidal divisor $E$.
\end{Definition}

Let $\mathcal{F}$ be a totally $E$-dicritical foliation with $E$ a $(m,n)$-cuspidal divisor. For each free point $P \in E$, there exists a cusp ${C}_P$ in $(\mathbb{C}^2,0)$ whose strict transform by $\pi$ cuts $E$ at $P$ and such that ${C}_P$ is an invariant curve of $\mathcal{F}$. We say that ${C}_P$ is an $E$-{\em cusp passing through} $P$ and we denote by $\operatorname{Cusp}_E(\mathcal{F})$ the set of all $E$-cusps of $\mathcal{F}$. Note that $\operatorname{Cusp}_E(\mathcal{F}) \subset \text{Curv}(E)=\text{Cusp}(E)$.

In \cite[Section 4]{Can-C-SS-2023},  cuspidal dicritical foliations are characterized in terms of the Newton polygon of the foliation. Assume that $E$ is  a $(m,n)$-cuspidal divisor and  take $(x,y)$ coordinates in $(\mathbb{C}^2,0)$ adapted to the divisor $E$. Consider a foliation $\mathcal{F}$ in $(\mathbb{C}^2,0)$ defined by a 1-form $\omega=0$. From the results in \cite{Can-C-SS-2023} we deduce the following statements
\begin{Lemma}
If $\omega$ is resonant and $\nu_E(\omega) < nm$, then the foliation $\mathcal{F}$ is totally $E$-dicritical.
\end{Lemma}
Next result describes the initial part of a 1-form defining a totally $E$ dicritical foliation
\begin{Lemma}\label{lemma:In-total-dicritical}
If the foliation $\mathcal{F}$ is totally $E$-dicritical then, up to multiply by a constant,
$$
\operatorname{In}_{\tfrac{m}{n}}(\omega;x,y)= x^a y^b \left(m \frac{dx}{x} - n \frac{dy}{y} \right),
$$
where $\nu_E(\omega) = n a + m b$.
\end{Lemma}

\begin{Remark}  Let  $\mathcal{F}$ be a foliation in $(\mathbb{C}^2,0)$ defined by a 1-form $\omega_{\mathcal{F}}=0$ and $E$ a cuspidal divisor. If the foliation  $\mathcal{F}$ is totally $E$-dicritical, then
\begin{itemize}
  \item[(i)]  $\nu_C(\omega_{\mathcal{F}})> \nu_E(\omega_{\mathcal{F}})$ for any $C \in \operatorname{Curv}(E)$;
  \item[(ii)]  the Newton polygon $\mathcal{N}(\mathcal{F})$ in adapted coordinates to $E$ fulfills
$$
\mathcal{N}(\mathcal{F}) \cap \{(i,j) \ : \ n i + m j =\nu_E(\omega_{\mathcal{F}})\} = \{(a,b)\}.
$$
with $\nu_E(\omega_{\mathcal{F}})=na+mb$.
\end{itemize}

\end{Remark}

\begin{Notation}
From now on, we will only consider  initial parts with weight $\tfrac{m}{n}$. Hence, we will denote $\operatorname{In} (\omega)= \operatorname{In}_{\tfrac{m}{n}}(\omega)$ for any $1$-form $\omega$ or $\operatorname{In} ( h)= \operatorname{In}_{\tfrac{m}{n}}(h)$ for any $h \in \mathbb{C}\{x,y\}$.
\end{Notation}

\begin{Remark}\label{rm-nu_C-nu_E-no-resonante}
Note that if $\eta$ is a non-resonant 1-form, then
$$\nu_E(\eta)=\nu_C(\eta)
$$
for any $C \in \operatorname{Curv}(E)$.
\end{Remark}
In general, given two 1-forms $\omega$ and $\eta$, we have that
$$
\nu_E(\omega \wedge \eta) \geq \nu_E(\omega) + \nu_E(\eta).
$$
Next lemma shows a particular case, that we will need later, where the equality holds.
\begin{Lemma}\label{lema_curva-contacto-no-resonante}
Let $\mathcal{F}$ be a totally $E$-dicritical foliation having $C$ as invariant curve. If $\omega$ is a non-resonant 1-form with
$$
\operatorname{In}(\omega)=x^\alpha y^\beta \left( \mu_1 \frac{dx}{x} + \mu_2 \frac{dy}{y} \right),
$$
then
$$
\nu_E(\omega_\mathcal{F} \wedge \omega)= \nu_E(\omega_\mathcal{F}) + \nu_E(\omega)=\nu_E(\omega_\mathcal{F}) + \nu_C(\omega).
$$
\end{Lemma}
\begin{proof}
Taking into account that $\omega_\mathcal{F}$ is resonant and $\omega$ is non-resonant, a simple computation shows that $\operatorname{In}(\omega_\mathcal{F})\wedge \operatorname{In}(\omega) \neq 0$. Hence, if we write $\omega_\mathcal{F} \wedge \omega=J(x,y) \frac{dx}{x} \wedge \frac{dy}{y}$, then
$$
\operatorname{In}(J) \frac{dx}{x} \wedge \frac{dy}{y} =\operatorname{In}(\omega_\mathcal{F})\wedge \operatorname{In}(\omega)
= - \mu x^{a+\alpha} y^{b + \beta} (n\mu_1 + m \mu_2) \frac{dx}{x} \wedge \frac{dy}{y}
$$
where we assume that $\operatorname{In}(\omega_\mathcal{F})$ is given by an expression as in~\eqref{eq:1-forma-dicritica}.
From the previous equality, we get  $\nu_E(\omega_\mathcal{F} \wedge \omega)= \nu_E(\omega_\mathcal{F}) + \nu_E(\omega)$.
The second equality in the statement is consequence of the fact that $\nu_E(\omega)=\nu_C(\omega)$ since $\omega$ is non-resonant.
\end{proof}
\subsection{Totally $E$-dicritical foliations} In this section, we give some properties of $E$-dicritical foliations  without assuming that the curves in $\operatorname{Curv}(E)$ are cusps.

\begin{Lemma}\label{lema-milnor-adaptado} Let $\mathcal{F}$ be a totally $E$-dicritical foliation. If $S,\widetilde{S} \in \operatorname{Curv}_E(\mathcal{F})$, then
$$\mu_0(\mathcal{F},S)=\mu_0(\mathcal{F},\widetilde{S}).$$
\end{Lemma}
\begin{proof}
Let us prove the result by induction on the number $N$ of blow-ups in the sequence~\eqref{eq:suc-explosiones-totaldicritical} above.

Assume that $N=1$. Up to a analytic change of coordinates, the foliation $\mathcal{F}$ is defined by $ydx-xdy=0$ and the curves in $\operatorname{Curv}_E(\mathcal{F})$ are given by $\gamma_{[a:b]}=(at,bt)$ with $[a:b] \in \mathbb{P}_{\mathbb{C}}^1$. We obtain that $\mu_0(\mathcal{F},S)=1$ for any $S \in \operatorname{Curv}_E(\mathcal{F})$.

Assume that $N >1$ and consider the first blow-up $\pi_1: (M_1,P_1) \to (\mathbb{C}^2,0)$ in the sequence~\eqref{eq:suc-explosiones-totaldicritical}. Take $S,\widetilde{S} \in \operatorname{Curv}_E(\mathcal{F})$. Let $\mathcal{F}_1, S_1,\widetilde{S}_1$ be the strict transforms of $\mathcal{F}, S,\widetilde{S}$ by the blow-up  $\pi_1$. Note that $S_1 \cap E_1^1=\widetilde{S}_1 \cap E_1^1=\{P_1\}$. Let us denote $\nu_0(\mathcal{F})$ the multiplicity of $\mathcal{F}$ at the origin and $m_0(S)$ the multiplicity of the curve $S$ at the origin. The behaviour of the Milnor number $\mu_0(\mathcal{F},S)$ along $S$ under blow-up is given by
$$
\mu_0(\mathcal{F},S)=\left\{
                       \begin{array}{ll}
                         \mu_{P_1}(\mathcal{F}_1,S_1) + m_0(S) \cdot \nu_0(\mathcal{F}), & \hbox{ if $\pi_1$ is dicritical} \\
                         \mu_{P_1}(\mathcal{F}_1,S_1) + m_0(S) \cdot (\nu_0(\mathcal{F})-1), & \hbox{ if $\pi_1$ is not dicritical.}
                       \end{array}
                     \right.
$$
By induction hypothesis $\mu_{P_1}(\mathcal{F}_1,S_1)=\mu_{P_1}(\mathcal{F}_1,\widetilde{S}_1)$. Since $m_0(S)=m_0(\widetilde{S})$, we get the result.
\end{proof}
Let us fix coordinates $(x,y)$ in $(\mathbb{C}^2,0)$ adapted to $E$. If a 1-form defines a totally $E$-dicritical foliation, the divisorial value of the 1-form can be computed as follows:

\begin{Lemma}\label{lemma-Milnor-valorDiff}
Let $\mathcal{F}$ be a totally $E$-dicritical foliation defined by  $\omega_{\mathcal{F}}=0$ and take $S \in \operatorname{Curv}_E(\mathcal{F})$. Then
$$
\nu_E(\omega_\mathcal{F})=\mu_0(\mathcal{F},S) + \nu_E(dx) + \nu_E(dy)-1.
$$
\end{Lemma}
\begin{proof}
Assume that $S$ is a non-singular curve, then all the centers $P_i$ of the blow-ups in the sequence~\eqref{eq:suc-explosiones-totaldicritical} are free points and a primitive parametrization of $S$ is given by $\gamma(t)=(t,a t^N+ \cdots)$. From the results in \cite{Can-C-SS-2023}, we can use the statement in Lemma~\ref{lemma:In-total-dicritical}  when $(m,n)=(N,1)$ and then the initial part of $\omega_\mathcal{F}$ is given by
$$
\operatorname{In}_{N}(\omega_\mathcal{F})=\mu x^{a-1} y^{b-1} (Nydx-xdy)
$$
where $\nu_E(\omega_\mathcal{F})=a+Nb$. The computation of the Milnor number $\mu_0(\mathcal{F},S)$ along $S$ gives $\mu_0(\mathcal{F},S)=a+N(b-1)$ and we get the result since $\nu_E(dx)=1$ and $\nu_E(dy)=N$.

Assume now that $S$ is a singular curve ($m_0(S)>1$) and consider $\gamma(t)=(x(t),y(t))$ a primitive parametrization of $S$. Then $x(t) \neq 0$ and $y(t) \neq 0$ and  $m_0(S)=\text{ord}_t(x(t)) <\text{ord}_t(y(t))$ since the coordinates are adapted to $E$.
Moreover, the curve $S$ is an invariant curve of $\mathcal{F}$, then $\gamma^* \omega_\mathcal{F}\equiv 0$, that is,
$$
A(\gamma(t)) \cdot \dot{x}(t) + B(\gamma(t)) \cdot \dot{y}(t)\equiv 0,
$$
where we write $\omega_\mathcal{F}=A(x,y)dx+B(x,y)dy$ with $A,B \in \mathbb{C}\{x,y\}$.
Consequently,  we get
$$
\text{ord}_t(A(\gamma(t))) +  \text{ord}_t(x(t)) = \text{ord}_t (B(\gamma(t))) + \text{ord}_t(y(t)).
$$
By equation~\eqref{ec:nu_E-omega} and Lemma~\ref{lemma:nu_E_funcion}, the divisorial value $\nu_E(\omega_\mathcal{F})$ is given by
\begin{align}
  \nu_E(\omega_\mathcal{F}) & = \min \{\nu_E(xA),\nu_E(yB)\} \notag \\
   & \leq \text{ord}_t(A(\gamma(t))) +  \text{ord}_t(x(t))  \label{eq:nu_E-mu_0}\\
   & = \mu_0(\mathcal{F},S) + \text{ord}_t(x(t)) + \text{ord}_t(y(t))-1 \notag\\
   & = \mu_0(\mathcal{F},S) + \nu_E(dx) + \nu_E(dy)-1.\notag
\end{align}
Let us see that the inequality in \eqref{eq:nu_E-mu_0} is an equality.  Assume  that $\min \{\nu_E(xA),\nu_E(yB)\} =\nu_E(xA)$ (the other case works similar) and that there is a curve $C^* \in \operatorname{Curv}(E)$ with
$$
\nu_E(xA)=\text{ord}_t(A(\gamma^*(t))) +  \text{ord}_t(x^*(t))<\text{ord}_t(A(\gamma(t))) +  \text{ord}_t(x(t))
$$
where $\gamma^{*}(t)=(x^*(t),y^*(t))$ is a primitive parametrization of $C^*$ with $m_0({C^*})=\text{ord}_t(x^*(t))$. Denote $\mathcal{A}$ the curve defined by $A=0$ in $(\mathbb{C}^2,0)$. Since $S$ and ${C^*}$ are equisingular curves, the inequality above implies
$$
\nu_E(A)=\nu_{C^*}(A)=\text{ord}_t(A(\gamma^*(t))) < \text{ord}_t(A(\gamma(t)))= \nu_S(A) = \nu_E(A)+i_P(S',\mathcal{A}')
$$
where $S',\mathcal{A}'$ are the strict transforms of $S$ and $\mathcal{A}$ by $\pi$ respectively, and $P$ is the infinitely near point of $S$ in the divisor $E$. If $Q$ is the  infinitely near point of $C^*$ in the divisor $E$, then $i_Q(C',\mathcal{A}')=0$. Since $\mathcal{F}$ is a totally $E$-dicritical foliation, there exists a curve $S_Q \in \operatorname{Curv}_E(\mathcal{F})$ such that $Q$ is the infinitely near point of $S_Q$ in $E$. Hence, $i_Q(S'_Q,\mathcal{A}')=0$ and then $\nu_E(A)=\nu_{S_Q}(A)$.

By Lemma~\ref{lema-milnor-adaptado}, we have $\mu_0(\mathcal{F},S)=\mu_0(\mathcal{F},S_Q)$ which implies
$$
\text{ord}_t(A(\gamma(t))-\text{ord}_t(y(t))+1 = \text{ord}_t(A(\gamma_Q(t))-\text{ord}_t(y_Q(t))+1.
$$
 Since the coordinates $(x,y)$ are adapted to $E$, the curve $y=0$ has maximal contact with any curve in $\operatorname{Curv}(E)$ and hence $\text{ord}_t(y(t))=\text{ord}_t(y_Q(t))$. Then we get $\nu_S(A)=\text{ord}_t(A(\gamma(t))=\text{ord}_t(A(\gamma_Q(t))=\nu_{S_Q}(A)$ which  is not possible.
\end{proof}


The results above concerning the divisorial value together with the results concerning invariants associated to foliations introduced in the previous sections allow us to generalize Theorem 3.8 in \cite{Gom} (see also \cite{Gom-tesis}).
\begin{Theorem}\label{Th:nu_C}
Let $\mathcal{F}$ and $\mathcal{G}$ be two foliations in  $({\mathbb C}^2,0)$ such that $\mathcal F$ and $\mathcal G$ have no common separatrix. Assume that $\mathcal{F}$ is a totally $E$-dicritical foliation and take $C \in \text{Curv}_E(\mathcal{F})$. Then
\begin{equation}\label{ec:muC-muE-2-formas}
\nu_C(\omega_{\mathcal{G}})= \nu_E(\omega_{\mathcal{F}}\wedge \omega_{\mathcal{G}})-\nu_E(\omega_\mathcal{F})+ i_P(\mathcal{J}_{\mathcal{F},\mathcal{G}}',C')
\end{equation}
where $P$ is the infinitely near point of $C$ in the divisor $E$ and $C'$, $\mathcal{J}_{\mathcal{F},\mathcal{G}}'$ are the strict transforms of $C$ and $\mathcal{J}_{\mathcal{F},\mathcal{G}}$ by $\pi$.
\end{Theorem}
\begin{proof}
From Proposition~\ref{prop:int-sep} and Lemma~\ref{lemma-Milnor-valorDiff}, we get
\begin{align*}
\tau_0(\mathcal{G},C) & =i_0(\mathcal{J}_{\mathcal{F},\mathcal{G}},C) -\mu_0({\mathcal F},C) \\
& = i_0(\mathcal{J}_{\mathcal{F},\mathcal{G}},C) -\nu_E(\omega_\mathcal{F}) + \nu_E(dx) + \nu_E(dy) -1
\end{align*}
By Remark~\ref{Rm-nuC-nuE}, we have
$$
i_0(\mathcal{J}_{\mathcal{F},\mathcal{G}},C)=\nu_C({J}_{\mathcal{F},\mathcal{G}})=\nu_E({J}_{\mathcal{F},\mathcal{G}})+ i_P(\mathcal{J}_{\mathcal{F},\mathcal{G}}',C')
$$
with ${J}_{\mathcal{F},\mathcal{G}}=0$ an equation of the jacobian curve $\mathcal{J}_{\mathcal{F},\mathcal{G}}$.
Taking into account Corollary~\ref{lema-2-div-valor} and the equality $\nu_C(\omega_\mathcal{G})=\tau_0(\mathcal{G},C)+1$, we obtain
$$
\nu_C(\omega_\mathcal{G})= \nu_E(\omega_{\mathcal{F}}\wedge \omega_{\mathcal{G}})-\nu_E(\omega_\mathcal{F})+ i_P(\mathcal{J}_{\mathcal{F},\mathcal{G}}',C')
$$
as wanted.
\end{proof}
The previous result will motivate the definition of the transversality property given in Section~\ref{sec:transv-polar}.

\begin{Remark}\label{Th:nu_C_1-forma-no-foliacion}
Theorem~\ref{Th:nu_C}  has been proved using indices associated to foliations. However, in some computations (see Proposition~\ref{prop:nu_C-nu_tilde_C}), we will need to compute the differential value $\nu_C(\eta)$ for a 1-form $\eta$ which is not saturated.

Let us write $\eta=h \eta_1$ with $h \in \mathcal{O}_{\mathbb{C}^2,0}$ and $\eta_1$ a saturated 1-form. Denote by $\mathcal{G}$ the foliation defined by $\eta_1=0$. Consider a foliation $\mathcal F$  defined by a 1-form $\omega_\mathcal{F}$ such that $C$ is an irreducible invariant curve of $\mathcal{F}$.
Let $\mathcal{J}_{\eta,\omega_{\mathcal{F}}}$ be the curve defined by $J_{\eta, \omega_{\mathcal{F}}}=0$ with $\eta \wedge \omega_{\mathcal{F}}=J_{\eta, \omega_{\mathcal{F}}} dx \wedge dy$ and $\mathcal{H}$ be  the curve given by $h=0$. Hence,   we have
$$
\mathcal{J}_{\eta,\omega_{\mathcal{F}}} =  \mathcal{J}_{\mathcal{F},\mathcal{G}} \cup \mathcal{H}
$$
since
$$
\eta \wedge \omega_{\mathcal{F}}=h \eta_1 \wedge \omega_{\mathcal{F}}.
$$
Let us compute $\nu_C(\eta)$. We have that
$$
\nu_C(\eta) =\nu_C(h \eta_1)  =\nu_C(h) + \nu_C(\eta_1)
$$
and taking into account Theorem~\ref{Th:nu_C} and Remark~\ref{Rm-nuC-nuE}, we obtain
\begin{align*}
\nu_C(\eta) & = \nu_C(h) + \nu_C(\eta_1) \\
& = \nu_E(h) + i_P(C',\mathcal{H}') + \nu_E(  \eta_1  \wedge \omega_\mathcal{F}) - \nu_E( \omega_\mathcal{F}) +i_P(\mathcal{J}_{\mathcal{F},\mathcal{G}}',C')
\\
&=
\nu_E(h  \eta_1 \wedge \omega_\mathcal{F}) - \nu_E( \omega_\mathcal{F}) +i_P(\mathcal{J}_{\mathcal{F} ,\mathcal{G}}' \cup \mathcal{H}',C')
\end{align*}
Then, we obtain the following equation for a non-saturated 1-form
$$
\nu_C(\eta) = \nu_E(  \eta \wedge \omega_\mathcal{F}) - \nu_E( \omega_\mathcal{F}) + i_P(\mathcal{J}_{\eta,\omega_{\mathcal{F}}}',C')
$$
which extends equation~\eqref{ec:muC-muE-2-formas} to non-saturated 1-forms.
\end{Remark}


\section{Cuspidal semimodules} \label{sec:cuspidal-semimodules}
Let $C$ be an irreducible plane curve in $(\mathbb{C}^2,0)$ with one Puiseux pair $(m,n)$ where $2 \leq n <m$ and $\gcd(n,m)=1$. Then the semigroup $\Gamma_C$ is generated by $n, m$, that is, $\Gamma_C=\{a n + b m \ : \ a,b \in \mathbb{Z}_{\geq 0}\}$. We write $\Gamma_C=\langle n,m\rangle$ and we say that $\Gamma_C$ is a {\em cuspidal semigroup}. In this case, the conductor $c({\Gamma_C})=(n-1)(m-1)$.

Let $\Lambda_C$  be the semimodule of differential values of $C$. Since $\Lambda_C$ is a $\Gamma_C$-semimodule, we say that $\Lambda_C$ is a {\em cuspidal semimodule}.
There exists a unique basis of $\Lambda_C$, that is, a strictly increasing sequence $\mathcal{B}=(\lambda_{-1},\lambda_0,\lambda_1,\ldots,\lambda_s)$ of elements of $\Lambda_C$, with $s$ minimal, such that
$$
\Lambda_C=\bigcup_{i=-1}^s (\lambda_i + \Gamma_C).
$$
Note that $\lambda_{-1}=n$ and $\lambda_0=m$.

\begin{Remark}
In \cite{Zar}, O. Zariski introduces the analytic invariant $\lambda=\min (\Lambda_C \smallsetminus \Gamma_C)-n$ and he proves that $\Lambda_C \smallsetminus \Gamma_C=\emptyset$ if and only if the curve $C$ is analytically equivalent to the curve $y^n-x^m=0$. The exponent $\lambda$ is called the {\em Zariski invariant} of $C$.  Note that $\lambda_1=\lambda+n$.
\end{Remark}
A list of $1$-forms $(\omega_{-1},\omega_0,\omega_1,\ldots,\omega_s)$ satisfying $\nu_C(\omega_i)=\lambda_i$ for $-1\leq i \leq s$ is called a {\em minimal standard basis} for the curve $C$. Since each $\lambda_i$ is an element of a minimal set of generators of $\Lambda_C$, then any 1-form $\omega_i$ with $\nu_C(\omega_i)=\lambda_i$ defines a foliation in $(\mathbb{C}^2,0)$.

For the rest of this section, we fix a minimal standard basis  $(\omega_{-1},\omega_0,\omega_1,\ldots,\omega_s)$ for $C$ and we denote $\mathcal{G}_i$ the foliation in $(\mathbb{C}^2,0)$ defined by $\omega_i=0$ for $1 \leq i \leq s$.

Let $\pi: M \to (\mathbb{C}^2,0)$ be the minimal reduction of singularities of $C$ and let $E$ be the cuspidal divisor of the sequence $\mathcal{S}$ of point blow-ups that gives $\pi$. By the results of \cite{Can-C-SS-2023} we have the following properties:
\begin{Lemma}[\cite{Can-C-SS-2023}]\label{lemma_G_i_dicritica}
For $1 \leq i \leq s$, the foliations $\mathcal{G}_i$ are totally $E$-dicritical. Moreover, if $(x,y)$ are coordinates in $(\mathbb{C}^2,0)$ adapted to $E$, then
$$
\operatorname{In} (\omega_i)= \mu_i x^{a_i} y^{b_i} \left(m \frac{dx}{x} - n \frac{dy}{y}\right), \quad \mu_i \in \mathbb{C}\smallsetminus \{0\},
$$
with $\nu_E(\omega_i)=n a_i + m b_i=t_i <nm$ for $1 \leq i \leq s$. Note also that $(a_1,b_1)=(1,1)$ and
$$
1=a_1 \leq a_2 \leq \cdots \leq a_s, \quad 1=b_1 \leq b_2 \leq \cdots \leq b_s.
$$
\end{Lemma}
The values $t_i=\nu_E(\omega_i)$, with $-1 \leq i \leq s$, correspond to the critical values of the semimodule $\Lambda_C$ introduced below (see expression given in \eqref{eq:t_i}). In fact, the 1-forms whose differential value is one of the elements of the basis $\mathcal{B}=(\lambda_{-1},\lambda_0,\lambda_1,\ldots,\lambda_s)$ of $\Lambda_C$ can be characterized in terms of its divisorial value as shown in Theorem~\ref{th_caracterizacion-omega-i}.
\begin{Remark}\label{rm:omega-1-0}
If $(x,y)$ are coordinates in $(\mathbb{C}^2,0)$ adapted to $E$, we have
$$
\operatorname{In} (\omega_{-1})=\mu_{-1} dx; \qquad \operatorname{In} (\omega_{0})=\mu_{0} dy,
$$
with $\mu_{-1},\mu_0 \in \mathbb{C}\smallsetminus \{0\}$.
\end{Remark}

Let us introduce the axes and the critical values of the cuspidal semimodule $\Lambda_C$ (see \cite{Can-C-SS-2023,Can-C-SS-2026} for more details). Denote $\Lambda=\Lambda_C$ and
$\Lambda_i=\bigcup_{k=-1}^i (\lambda_k + \Gamma)$, for $i=-1,0,1,\ldots,s$, where $\Gamma=\Gamma_C$.
Thus we have a {\em decomposition chain} of $\Lambda$:
$$
\Lambda_{-1} \subset \Lambda_0 \subset \Lambda_1 \subset \cdots \subset \Lambda_s=\Lambda
$$
where each $\Lambda_i$ is a $\Gamma$-semimodule with basis $\mathcal{B}_i=(\lambda_{-1},\lambda_0,\lambda_1,\ldots,\lambda_i)$ for $i=-1,0,1,\ldots,s$. The number $s$ is called the {\em length} of the semimodule $\Lambda$.

The {\em axes} $u_i^n$, $u_i^m$, $u_i$ and $\tilde{u}_i$ of $\Lambda$ are defined as
\begin{align*}
  u_i^n & = \min \{ \lambda_{i-1} + n \ell \in \Lambda_{i-2} \ : \ \ell \geq 1\}; \\
  u_i^m & = \min \{ \lambda_{i-1} + m \ell \in \Lambda_{i-2} \ : \ \ell \geq 1\}; \\
  u_i & = \min \{u_i^n, u_i^m\}; \\
  \tilde{u}_i & = \max \{u_i^n,u_i^m\}.
\end{align*}
for $1 \leq i \leq s+1$. We write $u_i^n = \lambda_{i-1}+n \ell^n_i$ and $u_i^m=\lambda_{i-1}+m \ell_i^m$. Note that $1 \leq \ell_i^n < m$ and $1 \leq \ell_i^m <n$ (see \cite[Remark 2.2]{Can-C-SS-2026}). The integers $\ell_i^n$ and $\ell_i^m$ are called {\em limits} of $\Lambda$.
\begin{Remark}
Since $\Lambda_C$ is the semimodule of differential values of a cusp $C$, then the semimodule $\Lambda_C$ is {\em increasing} which means that $\lambda_i > u_i$ for $1 \leq i \leq s$ (see \cite{Del}).
\end{Remark}
The {\em critical values} $t_i^n$, $t_i^m$, $t_i$ and $\tilde{t}_i$ of $\Lambda$ are given as follows: $t_{-1}=\lambda_{-1}=n$ and $t_0=\lambda_0=m$. For $1 \leq i \leq s+1$, we define
\begin{align}
  t_i^n & = t_{i-1} + n \ell_i^n= t_{i-1} + u_i^n-\lambda_{i-1}; \notag\\
  t_i^m & = t_{i-1} + m \ell_i^m = t_{i-1} + u_i^m-\lambda_{i-1}; \notag \\
  t_i & = \min \{ t_i^n,t_i^m\} = t_{i-1} + u_i - \lambda_{i-1}; \label{eq:t_i} \\
  \tilde{t}_i  & = \max \{ t_i^n,t_i^m\} = t_{i-1} + \tilde{u}_i - \lambda_{i-1}. \label{eq:tilde_t_i}
\end{align}
\begin{Remark}\label{Rm-def-critical-axes}
From the definition of the critical values in equations \eqref{eq:t_i} and \eqref{eq:tilde_t_i}, we obtain
\begin{equation}\label{eq:suma_u-t}
  u_i + \tilde{t}_i = \tilde{u}_i + t_i \qquad \text{ for  } 1 \leq i \leq s+1.
\end{equation}
Moreover,  we will show that the previous sum does not depend on the index $i$ and it is given by $\tilde{u}_i + t_i=nm+n+m$ for  $i =1,\ldots,s,s+1$ (see  Lemma~\ref{lemma:ui+ti}).
\end{Remark}
\begin{Notation}
If we need to work with the above invariants for different curves, we will denote them by $\Lambda_i^C$, $\mathcal{B}_C=(\lambda_{-1}^C,\lambda_0^C,\lambda_1^C, \ldots, \lambda_{s(C)}^C)$ or $(\omega_{-1}^C,\omega_0^C,\omega_1^C,\ldots,\omega_{s(C)}^C)$ where $s(C)$ is the length of the semimodule $\Lambda_C$.

The axes and the critical values will be denoted by $u_i^C$ and $t_i^C$ respectively.
\end{Notation}
In the rest of the section we will show some properties of the axes and critical values that will be useful later in the proofs of the results concerning the transversality property of foliations.
\begin{Lemma}[\cite{Can-C-SS-2023,Can-C-SS-2026}]\label{lema_t_i_u_i} For an increasing cuspidal semimodule $\Lambda_C$, the axes and critical values satisfy
\begin{itemize}
\item[(i)] $u_i>u_k$ and $t_i > t_k \quad $ for $1 \leq k < i \leq s+1$.
\item[(ii)] $\tilde{u}_i < \tilde{u}_k$ and $\tilde{t}_i < \tilde{t}_k \quad $ for $1 \leq k < i \leq s+1$.

\item[(iii)] $\lambda_i - \lambda_k > t_i - t_k \quad $ for $-1 \leq k < i \leq s$.
\end{itemize}
Therefore, we have that
\begin{align*}
n+m=t_1&< t_2< \cdots < t_s < t_{s+1} < \tilde{t}_{s+1} < \tilde{t}_s < \cdots < \tilde{t}_1=nm, \\
n+m=u_1&< u_2< \cdots < u_s < u_{s+1} < \tilde{u}_{s+1} < \tilde{u}_s < \cdots < \tilde{u}_1=nm.
\end{align*}
\end{Lemma}
The properties given in items $(i)$ and $(ii)$ are proved in Lemma 2.29 in \cite{Can-C-SS-2026} and assertion in item $(iii)$ corresponds to Lemma 7.10 in \cite{Can-C-SS-2023}.

Moreover, we have the following property
\begin{Lemma}[see \cite{Can-C-SS-2023}, Lemma 6.9 and \cite{Can-C-SS-2026}, Lemma 2.4]\label{Lemma:a-b-l}
If $\lambda_i + na + m b \in \Lambda_{i-1}$ with $a,b \in \mathbb{Z}_{\geq 0}$, then $a \geq \ell_{i+1}^n$ or $b \geq \ell_{i+1}^m$.
\end{Lemma}

 The value $\tilde{u}_{s+1}$ is related to the conductor $c(\Lambda_C)$ of the semimodule $\Lambda_C$. In \cite{Can-C-SS-2026}, it is proved that $\tilde{u}_{s+1} \geq c (\Lambda_C) + n$. Next lemma gives  the precise relationship between $\tilde{u}_{s+1}$ and the conductor $c(\Lambda_C)$.

\begin{Lemma}\label{lema_conductor_u_s+1}
We have
$$c(\Lambda_C)=\tilde{u}_{s+1}-n-m+1.
$$
\end{Lemma}
The proof of this result will be given in Appendix~\ref{ap:conductor}.

\medskip
The critical values associated to the semimodule $\Lambda_C$ correspond to the divisorial values of the 1-forms of a minimal standard basis $(\omega_{-1},\omega_0,\omega_1,\ldots,\omega_s)$ for $C$. More precisely, we have the following result
\begin{Theorem}[\cite{Can-C-SS-2023}, Theorem 7.13] \label{th_caracterizacion-omega-i}
For any $i \in \{1,2,\ldots, s\}$, the following statements hold
\begin{itemize}
  \item[(1)] $\lambda_i=\sup\{\nu_C(\omega) \ : \ \nu_E(\omega)=t_i\}$.
  \item[(2)] If $\nu_C(\omega)=\lambda_i$, then $\nu_E(\omega)=t_i$.
  \item[(3)] For each 1-form $\omega$ with $\nu_C(\omega) \not \in \Lambda_{i-1}$, there is a unique pair $a, b \geq 0$ such that $\nu_E(\omega)=\nu_E(x^a y^b \omega_i)$. In particular, we have that $\nu_C(\omega) \geq \lambda_i + na +mb$.
  \item[(4)] The semimodules $\Lambda_i$ are increasing since $\lambda_i > u_i$.
  \item[(5)] $\lambda_i + na + mb \not \in \Lambda_{i-1}$ if and only if, for any $\omega$ with $\nu_C(\omega)= \lambda_i + na + mb$, we have that $\nu_E(\omega)\leq \nu_E(x^a y^b \omega_i)$.
\end{itemize}
\end{Theorem}
Moreover, the other critical values $t_{s+1}$ and $\tilde{t}_i$, $1 \leq i \leq s+1$, are also the divisorial values of 1-forms which define foliations with interesting properties that we will describe in Section~\ref{sec:standard-systems}. Let us now introduce these 1-forms. In \cite[Proposition 8.3]{Can-C-SS-2023}, it is proved the existence of a 1-form $\omega_{s+1}$ such that
\begin{itemize}
  \item $\nu_E(\omega_{s+1})=t_{s+1}$;
  \item the cusp $C$ is an invariant curve of the foliation defined by $\omega_{s+1}=0$, that is, $\nu_C(\omega_{s+1})=\infty$.
\end{itemize}
We say that $\mathcal{E}=(\omega_{-1},\omega_0,\omega_1,\ldots,\omega_s,\omega_{s+1})$ is an {\em extended standard basis\/} for $C$ when $(\omega_{-1},\omega_0,\omega_1,\ldots,\omega_s)$ is a minimal standard basis for $C$ and $\omega_{s+1}$ satisfies the properties above. Note also that the foliation $\mathcal{G}_{s+1}$ defined by $\omega_{s+1}=0$ is also a totally $E$-dicritical foliation (see \cite{Can-C-SS-2023}).

Moreover, there exists a family $\widetilde{\mathcal{E}}=(\widetilde{\omega}_1,\widetilde{\omega}_2,\ldots,\widetilde{\omega}_s,\widetilde{\omega}_{s+1})$ of 1-forms satisfying
$$
\nu_E(\widetilde{\omega}_j)=\tilde{t}_j,\qquad \nu_C(\widetilde{\omega}_j)=\infty, \qquad \text{ for } 1 \leq j \leq s+1.
$$
(see \cite[Section 4]{Can-C-SS-2026}). We say that $(\mathcal{E},\widetilde{\mathcal{E}})$ is a {\em standard system\/} for the cusp $C$.

We finish this section with a technical lemma that will be useful later. Recall that, given two 1-forms $\omega$ and $\eta$, we say that $\omega$ is {\em reachable from} $\eta$ if and only if there exists  non-negative integer numbers $a,b$ and a constant $\mu \in \mathbb{C}$ such that
$$
\nu_E(\omega-\mu x^a y^b \eta) > \nu_E(\omega)
$$
Note that the pair $(a,b)$ and the constant $\mu$ are unique.

\begin{Lemma}\label{lema_eta_resonant-omega_i}
Let $\eta$ be a resonant 1-form with $\nu_E(\eta) \leq t_\ell$ for some $\ell \in \{1, \ldots, s\}$. Then there exist $a,b \geq 0$ and $i \in \{1,\ldots, \ell\}$ such that $\nu_E(\eta)=\nu_E(x^a y^b \omega_i)$ and $\nu_C(\eta) \leq \nu_C(x^a y^b \omega_i)$.
\end{Lemma}
\begin{proof}
Assume that for each $i \in \{1,\ldots,\ell\}$ and any $a,b \geq 0$, a 1-form $\alpha=x^a y^b \omega_i$ with $\nu_E(\alpha)=\nu_E(\eta)$ satisfies that $\nu_C(\alpha) < \nu_C(\eta)$.

Since $\eta$ is resonant and $\nu_E(\eta) \leq t_\ell$, we can define
$$j=\max \{ i \in \{1,\ldots, \ell\} \ : \ \eta \text{ is reachable from } \omega_i\}.$$
 Take $\alpha=x^a y^b \omega_j$ with $\nu_E(\alpha)=\nu_E(\eta)$ and hence $\operatorname{In}(\eta)=\mu \operatorname{In}(\alpha)$. By hypothesis we have that $\nu_C(\alpha)<\nu_C(\eta)$. Consequently the 1-form
$$
\beta=\eta- \mu \alpha
$$
satisfies $\nu_E(\beta) > \nu_E(\alpha)$ and $\nu_C(\beta)=\nu_C(\alpha)$. By statement (5) in Theorem~\ref{th_caracterizacion-omega-i}, we get that
\begin{equation}\label{Lambda_j-1}
\nu_C(\alpha)=\lambda_j + na + mb \in \Lambda_{j-1}.
\end{equation}
Let us see that this implies that  $\eta$ is reachable from $\omega_{j+1}$ or from $\widetilde{\omega}_{j+1}$. The expression given in~\eqref{Lambda_j-1} implies that $a \geq \ell_{j+1}^n$ or $b\geq \ell_{j+1}^m$ (see Lemma~\ref{Lemma:a-b-l}).
Let us assume that $t_{j+1}=t_{j+1}^n=t_j + n \ell_{j+1}^n$ (the other possibility is similar).
Using Delorme's decomposition given in Theorem~\ref{Th:descomp-Delorme}, we can write $\omega_{j+1}$ as
$$
\omega_{j+1}=\sum_{i=-1}^j h_i \omega_i
$$
with $\operatorname{In}(\omega_{j+1})=\operatorname{In}(h_j \omega_j)=\operatorname{In}(h_j) \operatorname{In}(\omega_j)$.
Since $\nu_E(\omega_{j+1})=t_{j+1}$, then $\operatorname{In}(\omega_{j+1})=x^{\ell_{j+1}^n} \operatorname{In}(\omega_j)$ up to multiply by a constant. In a similar way, we get that   $\operatorname{In}(\widetilde{\omega}_{j+1})=y^{\ell_{j+1}^m} \operatorname{In}(\omega_j)$.

Hence, if $a \geq \ell_{j+1}^n$, then $\eta$ is reachable from $\omega_{j+1}$. In case that $b \geq \ell_{j+1}^m$, we obtain that $\eta$ is reachable from $\widetilde{\omega}_{j+1}$.

Let us see that none of these possibilities can happen:
\begin{itemize}
  \item $\eta$ reachable from $\omega_{j+1}$ which contradicts the definition of $j$;
  \item $\eta$ reachable from $\widetilde{\omega}_{j+1}$: we have that $\nu_E(\eta) \geq  \tilde{t}_{j+1} > t_{s+1} > t_\ell$ (see Lemma~\ref{lema_t_i_u_i}) which is a contradiction with the hypothesis $\nu_E(\eta) \leq t_\ell$.
\end{itemize}
\end{proof}

\section{Foliations associated to a Standard System}\label{sec:standard-systems}
Let $C$ be an irreducible plane curve in $(\mathbb{C}^2,0)$ with one Puiseux pair $(m,n)$ where $2 \leq n <m$ and $\gcd(n,m)=1$. Consider a standard system $(\mathcal{E},\widetilde{\mathcal{E}})$ for $C$, where $\mathcal{E}=(\omega_{-1},\omega_0,\omega_1,\ldots,\omega_s,\omega_{s+1})$
is extended standard basis for $C$ and $\widetilde{\mathcal{E}}=(\widetilde{\omega}_{1},\widetilde{\omega}_2,\ldots,\widetilde{\omega}_s,\widetilde{\omega}_{s+1})$ (see Section~\ref{sec:cuspidal-semimodules}). Recall that
these 1-forms satisfy the following properties
\begin{equation*}
\begin{aligned}
  \nu_E(\omega_i) & = t_i, \  -1 \leq i \leq s+1, \qquad & \nu_E(\widetilde{\omega}_i) & = \tilde{t}_i, \ \   1 \leq i \leq s+1 \\
  \nu_C(\omega_i) & = \lambda_i,   \  -1 \leq i \leq s,  \qquad
  & \nu_C(\omega_{s+1}) & = \infty, \\
  \nu_C(\widetilde{\omega}_i) & = \infty,  \ \   1 \leq i \leq s+1
\end{aligned}
\end{equation*}
where $\mathcal{B}=(\lambda_{-1},\lambda_0,\lambda_1,\ldots, \lambda_s)$ is the basis of the semimodule $\Lambda_C$.

Once a standard system as above is fixed, we will denote by $\mathcal{G}_i$, resp. $\widetilde{\mathcal{G}}_i$, the foliation defined in $(\mathbb{C}^2,0)$ by the 1-form $\omega_i=0$, resp. $\widetilde{\omega}_i=0$, for $1 \leq i \leq s+1$.  Properties of these foliations are described in \cite{Can-C-SS-2023} and \cite{SS-2025}. In this section we will recall some of these properties which will be useful later to study totally $E$-dicritical foliations.

Consider the minimal reduction of singularities $\pi: M \to (\mathbb{C}^2,0)$ of $C$ and let $E$ be the cuspidal divisor. By the results in Section~\ref{sec:cuspidal-semimodules}, we know that the foliations $\mathcal{G}_i$ are totally $E$-dicritical for $i=1,2,\ldots,s,s+1$. Moreover, we can describe the semimodule of the cusps in $\operatorname{Cusp}_E(\mathcal{G}_i)$ which are invariant curves of $\mathcal{G}_i$:
\begin{Theorem}[\cite{Can-C-SS-2023}, Theorem 8.8]\label{Th:semirraices}
Fix $i \in \{1,2,\ldots,s,s+1\}$ and consider  any cusp  $C^i \in \operatorname{Cusp}_E(\mathcal{G}_i)$, then
\begin{itemize}
  \item[(a)] $\Lambda_{i-1}$ is the semimodule of differential values of $C^i$.
  \item[(b)] $\mathcal{E}_i=(\omega_{-1},\omega_0,\omega_1,\ldots,\omega_{i})$ is an extended standard basis for $C^i$.
\end{itemize}
\end{Theorem}
If we denote by $C_P^i$ the cusp in $\operatorname{Cusp}_E(\mathcal{G}_i)$ passing through a point $P \in E$, we say that $C_P^i$ is the $i$-{\em analytic semiroot} of $C$ through $P$ with respect to the extended standard basis $\mathcal{E}$. Observe that the result above does not imply that the curves $C_P^i$ and $C_Q^i$ are analytically equivalent for any pair of free points $P, Q \in E$ (see Example 8.13 in \cite{Can-C-SS-2023}).

The foliations $\widetilde{\mathcal G}_i$ are also totally $E$-dicritical for $i=2,\ldots,s,s+1$ (see \cite{SS-2025}, Chapter 6) but in general the semimodule of a curve $\widetilde{C}^i \in \operatorname{Cusp}_E(\widetilde{\mathcal{G}}_i)$ is not determined from $\Lambda_C$, we can only determine the first elements of its basis. More precisely, we have
\begin{Theorem}[\cite{SS-2025}, Theorem 7.6]\label{Th:separatrices-tilde}
Let $\widetilde{C}^i$ be a cusp in $\operatorname{Cusp}_E(\widetilde{\mathcal{G}}_i)$ with $i \in \{2,\ldots,s,s+1\}$. Then
$$
\nu_{\widetilde{C}^i}(\omega_\ell)=\nu_C(\omega_\ell)=\lambda_\ell, \qquad  \text{ for }  \ \ \ell= -1,0,1,\ldots, i-1.
$$
Consequently, $\Lambda_{i-1} \subset \Lambda_{\widetilde{C}^i}$. Moreover, the basis of the semimodule $\Lambda_{\widetilde{C}^i}$ is given by
$$
{\mathcal B}_{\widetilde{C}^i}=(\lambda_{-1},\lambda_0,\lambda_1,\ldots,\lambda_{i-1},\lambda_i^{\widetilde{C}^i},\ldots, \lambda_{s(\widetilde{C}^i)}^{\widetilde{C}^i})
$$
and there is a minimal standard basis for the curve $\widetilde{C}^i$ given by $$(\omega_{-1},\omega_0,\omega_1,\ldots, \omega_{i-1},\omega_i^{\widetilde{C}^i},\ldots, \omega_{s(\widetilde{C}^i)}^{\widetilde{C}^i}).$$
\end{Theorem}
In \cite[Example 7.7]{SS-2025}, it is shown that the inclusion $\Lambda_{i-1} \subset \Lambda_{\widetilde{C}^i}$ can be strict (see also Example \ref{ex-G-tilde-no-lambda-cte1}).

One of the objectives of this paper is to study under what conditions all the cusps in $\operatorname{Cusp}_E(\mathcal{F})$ have the same semimodule of differential values, for  a totally $E$-dicritical foliation $\mathcal{F}$ with $C$ as invariant curve. Let us recall some facts concerning the foliations with $C$ as invariant curve.

Let $\Omega^1(C)$ be the $\mathcal{O}_{\mathbb{C}^2,0}$-module of 1-forms $\omega \in \Omega^1_{(\mathbb{C}^2,0)}$ such that the curve $C$ is an invariant curve of the foliation defined by $\omega=0$. Note that if $f=0$ is a reduced equation of $C$ with $f \in \mathcal{O}_{\mathbb{C}^2,0}$, then
$$
\Omega^1(C)=\{ \omega \in \Omega^1_{(\mathbb{C}^2,0)} \ : \ f \text{ divides } \omega \wedge df\}.
$$
In \cite{Sai-1980}, K. Saito proved that $\Omega^1(C)$ is a free module of rank two and a basis of  $\Omega^1(C)$ is called a {\em Saito basis} for $C$. The module $\Omega^1(C)$ is now know as the {\em Saito module} of $C$. This module contains relevant information about the curve $C$. In particular, the {\em Saito number} of $C$ is defined as
$$
\mathfrak{s}(C)=\min_{\omega \in \Omega^1(C)} \{\nu_0(\omega)\}
$$
where $\nu_0(\omega)$ stands for the multiplicity of the 1-form $\omega$ at the origin. Note that $\frak{s}(C)$ is also equal to $\min \{\nu_0(\eta_1),\nu_0(\eta_2)\}$ with $\{\eta_1,\eta_2\}$ a Saito basis for $C$. In \cite{Gen-2022}, Y. Genzmer proved that the Saito number $\frak{s}(C)$ is an analytic invariant of the curve $C$.

Next result gives a simple criterion to check if a pair of 1-forms gives a Saito basis for the curve $C$.
\begin{Theorem}[Saito criterion \cite{Sai-1980}]\label{th:Saito-criterion}
Let $\eta_1,\eta_2 \in \Omega^1(C)$. The set $\{\eta_1,\eta_2\}$ is a Saito basis for $C$ if and only if
$$
\eta_1 \wedge \eta_2= u f dx \wedge dy
$$
with $u$ a unit in $\mathcal{O}_{\mathbb{C}^2,0}$.
\end{Theorem}
Note that if $\{\eta_1,\eta_2\}$ is a Saito basis for $C$, then the jacobian curve ${\mathcal J}_{\mathcal{F}_{\eta_1},\mathcal{F}_{\eta_2}}$ is equal to $C$, where $\mathcal{F}_{\eta_i}$ denotes the foliation defined by $\eta_i=0$ in $(\mathbb{C}^2,0)$.

In some recent works, Saito bases play an important role in the study of the analytic classification of plane curves and other new analytic invariants of a plane curve are defined from a Saito basis for the curve (see for instance \cite{Gen-2022,Gen-H-2020,Can-C-SS-2026}).

Next result describes a way to obtain a Saito basis for an irreducible curve with only one Puiseux pair from the data of a standard system for the curve.
\begin{Theorem}[\cite{Can-C-SS-2026}, Theorem 4.2]\label{Th:Saito_basis}
Consider $C$  a curve  in $(\mathbb{C}^2,0)$ with only one Puiseux pair. Let
$t_{s+1}$ and $\tilde{t}_{s+1}$  be the last critical values of the semimodule $\Lambda_C$.
If  $\eta$, $\widetilde{\eta}$ are 1-forms with $C$ as an invariant curve and such that
$$
\nu_E(\eta)=t_{s+1}, \qquad \nu_E(\widetilde{\eta})=\tilde{t}_{s+1},
$$
then the set $\{\eta, \widetilde{\eta}\}$ is a Saito basis for $C$.
\end{Theorem}
In particular, if $(\mathcal{E},\widetilde{\mathcal{E}})$ is a standard system for $C$ with $\mathcal{E}=(\omega_{-1},\omega_0,\omega_1,\ldots,\omega_s,\omega_{s+1})$ and $\widetilde{\mathcal{E}}=(\widetilde{\omega}_{1},\widetilde{\omega}_2,\ldots,\widetilde{\omega}_s,\widetilde{\omega}_{s+1})$, then $\{\omega_{s+1},\widetilde{\omega}_{s+1}\}$ is a Saito basis for $C$. Moreover, given two 1-forms $\eta$, $\widetilde{\eta}$ with $\nu_E(\eta)=t_{s+1}$,  $\nu_E(\widetilde{\eta})=\tilde{t}_{s+1}$, there exists a standard system  $(\mathcal{E},\widetilde{\mathcal{E}})=((\omega_{-1},\omega_0,\omega_1,\ldots,\omega_s,\omega_{s+1}), (\widetilde{\omega}_{1},\widetilde{\omega}_2,\ldots,\widetilde{\omega}_s,\widetilde{\omega}_{s+1}))$ for the curve $C$ with $\omega_{s+1}=\eta$ and $\widetilde{\omega}_{s+1}=\widetilde{\eta}$ (see Proposition 4.9 in \cite{Can-C-SS-2026}).

\section{Cuspidal Transversality Property}\label{sec:transv-polar}
By the results in Section~\ref{sec:E-dicritical}, it is interesting to know the behaviour of the jacobian curve of two foliations  since this curve appears in the expression given in Theorem~\ref{Th:nu_C} which allows to compute the
differential value of a 1-form  in terms of some divisorial values.

Along this section we will consider a cuspidal sequence $\mathcal{S}$ and a morphism \linebreak $\pi: M \to (\mathbb{C}^2,0)$ composition of point blow-ups as in equation~\eqref{eq:pi} with $E=E^{N}_N$ the cuspidal divisor. The results in Section~\ref{sec:E-dicritical} and
Theorem~\ref{Th:nu_C}  motivate next definition (see also Section 4 in \cite{Gom}).

\begin{Definition}
Let $\mathcal{F}$ and $\mathcal{G}$ two foliations in $(\mathbb{C}^2,0)$.
We say that the foliations $\mathcal{F}$ and $\mathcal{G}$ have the {\em transversality property for the divisor} $E$ (or the {\em $E$-transversality property}) if
$\mathcal{J}_{\mathcal{F},\mathcal{G}}'$ does not intersect $E$, where $\mathcal{J}_{\mathcal{F},\mathcal{G}}'$ is
 the strict transform of the jacobian curve $\mathcal{J}_{\mathcal{F},\mathcal{G}}$ by $\pi$.
\end{Definition}

If $\mathcal{F}$ and $\mathcal{G}$  have the $E$-transversality property and they satisfy the  hypothesis of Theorem~\ref{Th:nu_C} for a
curve $C \in \text{Curv}(E)$, then  the differential value $\nu_{C}(\omega_\mathcal{G})$ is equal to
$$
\nu_C(\omega_{\mathcal{G}})= \nu_E(\omega_{\mathcal{F}}\wedge \omega_{\mathcal{G}})-\nu_E(\omega_\mathcal{F}).
$$

\begin{Remark}
If the foliation $\mathcal{G}$ is non singular, the jacobian curve ${\mathcal J}_{\mathcal{F},\mathcal{G}}$ coincides with the polar curve of the foliation $\mathcal{F}$.
Properties concerning the equisingularity type of polar curves of foliations have been studied for instance in \cite{Cor-2003, Rou}. In particular,
if $\mathcal{F}$ is a generalized curve foliation such that the curve $S_{\mathcal{F}}$  of separatrices of $\mathcal{F}$ is a cusp with $S_{\mathcal{F}} \in \text{Curv}(E)$, then the strict transform of generic polar curve of $\mathcal{F}$ by $\pi$ does not intersect $E$ (see \cite{Cor-2003}).  Consequently, $\mathcal{F}$ and the foliations $\mathcal{G}_{[a:b]}$ have the $E$-transversality property for a generic $[a:b] \in \mathbb{P}_\mathbb{C}^1$, where $\mathcal{G}_{[a:b]}$ is the foliation defined by $b dx - a dy=0$.
\end{Remark}

Let us fix a cusp $C \in \text{Curv}(E)$ with Puiseux pair $(m,n)$, $2 \leq n < m$ and $\gcd(n,m)=1$. Consider the basis $\mathcal{B}=(\lambda_{-1},\lambda_0,\lambda_1,\ldots,\lambda_s)$ of the semimodule $\Lambda_C$ and fix a minimal standard basis $(\omega_{-1},\omega_0,\omega_1,\ldots,\omega_s)$ for $C$. We denote by $\mathcal{G}_\ell$ the foliation defined by $\omega_\ell=0$ for $\ell=-1,0,1,2,\ldots,s$. Recall that the foliations $\mathcal{G}_\ell$ are totally $E$-dicritical for $\ell=1,2,\ldots,s$ (see Lemma~\ref{lemma_G_i_dicritica}).

\begin{Definition}
Let $\mathcal{F}$ be a totally $E$-dicritical foliation in $(\mathbb{C}^2,0)$ with $C$ as invariant curve. The foliation $\mathcal{F}$ satisfies the $\ell$-{\em transversality property (relative to a minimal standard basis $(\omega_{-1},\omega_0,\omega_1,\ldots,\omega_s)$) for the divisor} $E$ if $\mathcal{F}$ and $\mathcal{G}_\ell$ have the transversality property for the divisor $E$.
\end{Definition}

Note that the definition of $\ell$-transversality property depends on the chosen minimal standard basis $(\omega_{-1},\omega_0,\omega_1,\ldots,\omega_s)$ for $C$. At the end of the section we will prove that this definition is independent of the choice of the minimal standard basis (see Corollary~\ref{Cor:indep-basis}). Moreover, since we also fix the divisor $E$ along the section, we will say that $\mathcal{F}$ satisfies the $\ell$-transversality property when we refer to the notion introduced in the previous definition if no confusion is possible. We will also take coordinates $(x,y)$ in $(\mathbb{C}^2,0)$ adapted to $E$.
\begin{Remark}\label{Rm:-1-0-transv}
Let us consider a totally $E$-dicritical foliation $\mathcal{F}$ defined by a $1$-form $\omega_\mathcal{F}=0$. Let us write
$$
\omega_\mathcal{F} \wedge \omega_{-1} = J_{{-1}}(x,y) dx \wedge dy; \qquad \omega_\mathcal{F} \wedge \omega_{0} = J_{0}(x,y) dx \wedge dy,
$$
where we denote $J_{{\ell}}(x,y)=J_{\mathcal{F},\mathcal{G}_{\ell}}(x,y)$  once the foliation $\mathcal{F}$ is fixed. Moreover, we will denote the jacobian curve $\mathcal{J}_\ell=\mathcal{J}_{\mathcal{F},\mathcal{G}_\ell}$ given by $J_\ell=0$  if there is no possible confusion.
By Lemmas~\ref{lemma:In-total-dicritical}, \ref{lema_curva-contacto-no-resonante} and Remark~\ref{rm:omega-1-0}, up to multiply by a constant, we have
$$
\operatorname{In}(J_{-1})=x^a y^{b-1}; \qquad \operatorname{In}(J_{0})=x^{a-1} y^{b}
$$
where $\nu_E(\omega_\mathcal{F})=na + mb$. Consequently, the strict transforms $\mathcal{J}'_{{-1}}$ and $\mathcal{J}'_{{0}}$ by $\pi$ of the curves  $\mathcal{J}_{-1}$ and $\mathcal{J}_{0}$ do not intersect the divisor $E$. Hence, any totally $E$-dicritical foliation satisfies the $\ell$-transversality property for the divisor $E$ when $\ell=-1,0$.

From the previous computations and Theorem~\ref{Th:nu_C}, we deduce that
\begin{equation}\label{eq:Rm:-1-0-transv}
\nu_E(\omega_\mathcal{F}\wedge \omega_{-1})=\nu_E(\omega_{\mathcal{F}}) + \lambda_{-1}; \qquad \nu_E(\omega_\mathcal{F}\wedge \omega_{0})=\nu_E(\omega_{\mathcal{F}}) +  \lambda_0.
\end{equation}

\end{Remark}
Let us show that, if a foliation $\mathcal{F}$ satisfies that $\ell$-transversality property for an index $\ell$, then $\mathcal{F}$ satisfies transversality property for all the previous indices. First, we prove a technical lemma.
\begin{Lemma}\label{lemma_nu_h_omega}
For any $h \in \mathbb{C}\{x,y\}$ and $\ell \in \{1,2,\ldots,s\}$, we have
\begin{itemize}
  \item[(i)] if $\nu_C(h \omega_\ell) < nm$, then $\nu_C(h)=\nu_E(h)$;
  \item[(ii)] if $ u_i < \nu_C(h\omega_\ell)< nm $, then $\nu_E(h) > u_i - \lambda_\ell$ for any $-1 \leq i\leq s$.
\end{itemize}
\end{Lemma}
\begin{proof}
Let us prove (i). If $\nu_C(h \omega_\ell) < nm$, then $\nu_C(h) < nm-\lambda_\ell$. Hence, we have that $\nu_E(h) <nm$ and, by Remark~\ref{monomio-nm}, we obtain $\nu_E(h)=\nu_C(h)$.

From the hypothesis in statement (ii), we obtain
$$
u_i < \nu_C(h\omega_\ell) = \nu_C(h) + \nu_C(\omega_\ell)=\nu_E(h) + \lambda_\ell
$$
which gives $\nu_E(h) > u_i - \lambda_\ell$ as desired.
\end{proof}

\begin{Proposition}\label{Prop:l-transv-j-transv}
Let $\mathcal{F}$ be a totally $E$-dicritical foliation with $C$ as invariant curve and fix $\ell \in \{1,2,\ldots, s\}$. If $\mathcal{F}$ satisfies the $\ell$-transversality property, then $\mathcal{F}$ satisfies the $j$-transversality property for any $j \in \{1,2,\ldots,\ell\}$.
\end{Proposition}

In order to prove the proposition above, we will write the 1-form $\omega_\ell$ using Delorme's decomposition  given in Theorem~\ref{Th:descomp-Delorme}.
Given any $j \in \{0,1,\ldots,\ell-1\}$, we can write
\begin{equation}\label{eq:omega-ell}
\omega_\ell=\sum_{i=-1}^{j} h_i \omega_i
\end{equation}
with $\nu_C(h_j \omega_j)= \nu_C(h_k \omega_k)$ for some $k  \in \{-1,0,1,\ldots, j-1\}$ and  $\nu_C(h_j \omega_j)= \nu_C(h_k \omega_k)< \nu_C(h_i \omega_i)$ for any $i \neq k$, with $-1 \leq i \leq j-1$. Note also that $\nu_C(h_j \omega_j) <\nu_C(\omega_\ell)=\lambda_\ell$.

The proof of Proposition~\ref{Prop:l-transv-j-transv} follows from the inequalities given in the following lemma:
\begin{Lemma}\label{lemma-desigualdades}
Let $\mathcal{F}$ be a totally $E$-dicritical foliation, defined by $\omega=0$, with $C$ as invariant curve. Assume that the 1-form $\omega_\ell$ is written as in expression~\eqref{eq:omega-ell} with the properties above given by Delorme's decomposition. If the foliation $\mathcal{F}$ satisfies the $\ell$-transversality property  and the $p$-transversality property for $1 \leq p \leq j-1$, then the following inequalities hold
\begin{align}
   \nu_E(h_j \omega_j \wedge \omega) & \leq  \nu_E(h_k \omega_k \wedge \omega) \label{hj-hk}\\
  \nu_E(h_j \omega_j \wedge \omega) & <  \nu_E( \omega_\ell \wedge \omega) \label{hj-omega-ell}\\
  \nu_E(h_j \omega_j \wedge \omega) & < \nu_E(h_{i} \omega_{i} \wedge \omega) \label{hj-hi}
\end{align}
for $i \neq k$, with $-1 \leq i \leq j-1$.
\end{Lemma}
\begin{proof}
Since the foliation $\mathcal{F}$ satisfies the transversality property with $\mathcal{G}_\ell$ and $\mathcal{G}_k$, we have that
\begin{align}
  \lambda_\ell & = \nu_E(\omega \wedge \omega_\ell) - \nu_E(\omega),  \label{transv-ell}\\
  \lambda_k &  = \nu_E(\omega \wedge \omega_k) - \nu_E(\omega). \label{transv-k}
\end{align}
By Theorem~\ref{Th:nu_C}, the differential values $\lambda_j$ and $\lambda_i$ can be computed as
\begin{align}
  \lambda_j & =  \nu_E(\omega \wedge \omega_j) - \nu_E(\omega) + i_P(\mathcal{J}_{\mathcal{F},\mathcal{G}_j}',C') \label{lambda_j}, \\
  \lambda_i & =  \nu_E(\omega \wedge \omega_i) - \nu_E(\omega)   \label{lambda_i}
\end{align}
with $1 \leq i \leq j-1$.
Note that the hypothesis $\nu_C(h_j \omega_j)= \nu_C(h_k \omega_k)$   implies that
\begin{equation}\label{j-k-lambda-nuC}
  \nu_C(h_j) + \lambda_j = \nu_C(h_k) +\lambda_k.
\end{equation}
Moreover, since $ \nu_C(h_j \omega_j)= \nu_C(h_k \omega_k) < \nu_C(\omega_\ell)=\lambda_\ell <nm$, then by Lemma~\ref{lemma_nu_h_omega}, we obtain
\begin{equation}\label{ec:nu_C_E}
\nu_C(h_j)=\nu_E(h_j); \qquad \nu_C(h_k)=\nu_E(h_k).
\end{equation}
The properties of Delorme's decomposition imply that
\begin{equation}\label{lambda_i-lambda_j}
\nu_C(h_j) + \lambda_j   < \nu_C(h_i) + \lambda_i.
\end{equation}
Let us prove inequality~\eqref{hj-hk}:
\begin{align*}
  \nu_E(h_k \omega_k \wedge \omega) & =\nu_E(h_k) + \nu_E(\omega_k \wedge \omega) \\
   & \stackrel{\eqref{transv-k}}{=} \nu_E(h_k) + \lambda_k + \nu_E(\omega) \\
   & \stackrel{\eqref{j-k-lambda-nuC}}{=}  \nu_E(h_k) + \nu_C(h_j)+ \lambda_j - \nu_C(h_k) +  \nu_E(\omega) \\
    & \stackrel{\eqref{ec:nu_C_E}}{=} \nu_E(h_j)+ \lambda_j +\nu_E(\omega) \\
    & \stackrel{\eqref{lambda_j}}{=}  \nu_E(h_j) + \nu_E(\omega \wedge \omega_j) - \nu_E(\omega) + i_P(\mathcal{J}_{\mathcal{F},\mathcal{G}_j}',C') +\nu_E(\omega) \\
    & \geq \nu_E(h_j \omega \wedge \omega_j)
\end{align*}

Now we compute $\nu_E( \omega_\ell \wedge \omega)$ in order to obtain inequality~\eqref{hj-omega-ell}. We have
\begin{align*}
  \nu_E( \omega_\ell \wedge \omega) & \stackrel{\eqref{transv-ell}}{=} \lambda_\ell + \nu_E(\omega) \\
   & > \nu_C(h_j \omega_j) +  \nu_E(\omega) \\
   & = \nu_C(h_j) + \nu_C(\omega_j) + \nu_E(\omega) \\
   & \stackrel{\eqref{lambda_j}}{=} \nu_C(h_j)  + \nu_E(\omega \wedge \omega_j) - \nu_E(\omega) + i_P(\mathcal{J}_{\mathcal{F},\mathcal{G}_j}',C') + \nu_E(\omega) \\
   & \stackrel{\eqref{ec:nu_C_E}}{=} \nu_E(h_j)  + \nu_E(\omega \wedge \omega_j)+ i_P(\mathcal{J}_{\mathcal{F},\mathcal{G}_j}',C') \\
   & \geq \nu_E(h_j \omega \wedge \omega_j)
\end{align*}
Finally, let us prove inequality~\eqref{hj-hi}. We have to consider two cases: $\nu_C(h_i)<nm$ and $\nu_C(h_i) \geq nm$. If $\nu_C(h_i)<nm$, then $\nu_E(h_i)=\nu_C(h_i)$ and we have
\begin{align*}
  \nu_E(h_i \omega_i \wedge \omega) & =\nu_E(h_i) + \nu_E(\omega_i \wedge \omega) = \nu_C(h_i) +  \nu_E(\omega_i \wedge \omega)\\
    & \stackrel{\eqref{lambda_i}}{=} \nu_C(h_i) + \lambda_i + \nu_E(\omega) \\
  & \stackrel{\eqref{lambda_i-lambda_j}}{>} \nu_C(h_j) + \lambda_j + \nu_E(\omega) \\
  & \stackrel{\eqref{lambda_j}}{=} \nu_C(h_j)+ \nu_E(\omega \wedge \omega_j) - \nu_E(\omega) + i_P(\mathcal{J}_{\mathcal{F},\mathcal{G}_j}',C') + \nu_E(\omega) \\
  &  \stackrel{\eqref{ec:nu_C_E}}{=} \nu_E(h_j) + \nu_E(\omega \wedge \omega_j) + i_P(\mathcal{J}_{\mathcal{F},\mathcal{G}_j}',C')\\
  & \geq  \nu_E(h_j \omega \wedge \omega_j).
 \end{align*}
Consider now the case $\nu_C(h_i) \geq nm$. Note that we also have $\nu_E(h_i) \geq nm$ and hence $\nu_E(h_i) > \lambda_\ell > \nu_C(h_j \omega_j)$. Thus, we have
\begin{align*}
  \nu_E(h_i \omega_i \wedge \omega) & =\nu_E(h_i) + \nu_E(\omega_i \wedge \omega) \\
    & > \lambda_\ell  + \nu_E(\omega_i \wedge \omega) \\
   & > \nu_C(h_j \omega_j) + \nu_E(\omega_i \wedge \omega) \\
  & = \nu_C(h_j) + \lambda_j+ \nu_E(\omega_i \wedge \omega) \\
   & \stackrel{\eqref{lambda_j}}{=}  \nu_C(h_j) + \nu_E(\omega \wedge \omega_j) -\nu_E(\omega)  + i_P(\mathcal{J}_{\mathcal{F},\mathcal{G}_j}',C') + \nu_E(\omega_i \wedge \omega) \\
  &\stackrel{\eqref{ec:nu_C_E},\eqref{lambda_i}}{\geq} \nu_E(h_j) + \nu_E(\omega \wedge \omega_j) -\nu_E(\omega)  + i_P(\mathcal{J}_{\mathcal{F},\mathcal{G}_j}',C') + \lambda_i +\nu_E(\omega) \\
     & \geq   \nu_E(h_j \omega \wedge \omega_j).
\end{align*}
This ends the proof of the three inequalities.
\end{proof}
Let us explain how we obtain Proposition~\ref{Prop:l-transv-j-transv} from the above inequalities.
\begin{proof}[Proof of Proposition~\ref{Prop:l-transv-j-transv}] Let $\omega_\mathcal{F}$ be a 1-form such that $\mathcal{F}$ is defined by $\omega_\mathcal{F}=0$.
If $\ell=1$, the result holds. Let us prove the result by induction on $j$ with $1 \leq j < \ell$.
Assume that $\ell \geq 2$.  Let us prove that $\mathcal{F}$ satisfies the $1$-transversality property. By Delorme's decomposition (Theorem~\ref{Th:descomp-Delorme}), we can write
\begin{equation}\label{omega_ell-Delorme}
\omega_\ell = h_1 \omega_1 + h_0 \omega_0 + h_{-1}\omega_{-1}
\end{equation}
where there exists a unique index $k \in \{0,-1\}$ with
\begin{equation*}
  \nu_{C}(h_1 \omega_1)=\nu_{C}(h_k \omega_k) < \lambda_\ell .
\end{equation*}
For instance, we can assume that $k=0$ and hence we have that
\begin{equation*}
\nu_C(h_1) + \lambda_1 = \nu_C(h_0) + \lambda_0,
\end{equation*}
and by Lemma~\ref{lemma_nu_h_omega}, we get $\nu_E(h_1)=\nu_C(h_1)$  and $\nu_E(h_0)=\nu_C(h_0)$. Consequently, we obtain
\begin{equation}\label{Delorme_nu_E_1-0}
\nu_E(h_1) + \lambda_1 = \nu_E(h_0) + \lambda_0.
\end{equation}
By Theorem~\ref{Th:nu_C},   we get
\begin{align}
  \lambda_\ell & = \nu_E(\omega_{\mathcal{F}} \wedge \omega_\ell)-\nu_E(\omega_\mathcal{F}) \notag \\
  \lambda_1 & = \nu_E(\omega_{\mathcal{F}} \wedge \omega_1)-\nu_E(\omega_\mathcal{F}) + i_P(\mathcal{J}_{\mathcal{F},\mathcal{G}_1}',C') \label{ec:lambda1}
\end{align}
where the first equality follows from the $\ell$-transversality property of the foliation $\mathcal F$.
Moreover, we also have
\begin{equation}\label{transv-01}
\lambda_i= \nu_E(\omega_{\mathcal{F}} \wedge \omega_i)-\nu_E(\omega_\mathcal{F})
\end{equation}
for $i=-1,0$ by Remark~\ref{Rm:-1-0-transv}. From equation~\eqref{omega_ell-Delorme}, we obtain that
$$
  \omega_\ell \wedge \omega_\mathcal{F}= h_1 \omega_1 \wedge \omega_\mathcal{F} + h_0 \omega_0 \wedge \omega_\mathcal{F} + h_{-1}\omega_{-1}\wedge \omega_\mathcal{F}.
$$
By Lemma~\ref{lemma-desigualdades} applied to the expression of $\omega_\ell$ given in \eqref{omega_ell-Delorme}, we have the following inequalities
\begin{align*}
   \nu_E(h_1 \omega_1 \wedge \omega_\mathcal{F}) & \leq  \nu_E(h_0 \omega_0 \wedge \omega_\mathcal{F}) \\ 
  \nu_E(h_1 \omega_1 \wedge \omega_\mathcal{F}) & <  \nu_E( \omega_\ell \wedge \omega_\mathcal{F}) \\ 
  \nu_E(h_1 \omega_1 \wedge \omega_\mathcal{F}) & < \nu_E(h_{-1} \omega_{-1} \wedge \omega_\mathcal{F}) \\ 
\end{align*}
and we get that
$$
\nu_E(h_1 \omega_1 \wedge \omega_\mathcal{F}) =  \nu_E(h_0 \omega_0 \wedge \omega_\mathcal{F}).
$$
Consequently,
$$
\nu_E(h_1) + \nu_E(\omega_1 \wedge \omega_\mathcal{F}) =  \nu_E(h_0) + \nu_E( \omega_0 \wedge \omega_\mathcal{F}).
$$
and by equations~\eqref{ec:lambda1} and \eqref{transv-01}, we obtain
$$
\nu_E(h_1) +\lambda_1 - i_P(\mathcal{J}_{\mathcal{F},\mathcal{G}_1}',C') =  \nu_E(h_0) + \lambda_0.
$$
The equality given in~\eqref{Delorme_nu_E_1-0}, implies $i_P(\mathcal{J}_{\mathcal{F},\mathcal{G}_1}',C')=0$ and $\mathcal{F}$ satisfies the $1$-transversality property.

Let us prove the general case. By induction hypothesis assume now that the result holds for $j < \ell-1$. Let us prove that $\mathcal{F}$ satisfies the $(\ell-1)$-transversality property.

By Delorme's decomposition theorem (Theorem~\ref{Th:descomp-Delorme}), we can write $\omega_\ell$ as
\begin{equation}\label{omega_ell_caso_general}
  \omega_\ell = \sum_{i=-1}^{\ell-1} h_i \omega_i
\end{equation}
with $\nu_C(h_{\ell-1} \omega_{\ell-1})= \nu_C(h_k \omega_k)$ for some $k<\ell-1$ and $\nu_C(h_i \omega_i) > \nu_C(h_j \omega_j)$ for all $i \neq j,k$. Note also that
$\nu_C(h_{\ell-1} \omega_{\ell-1}) < \lambda_\ell=\nu_C(\omega_\ell)$ and hence $\nu_E(h_{\ell-1})=\nu_C(h_{\ell-1})$ and $\nu_E(h_k)=\nu_C(h_k)$ by Lemma~\ref{lemma_nu_h_omega}. Thus, we have
\begin{equation}\label{ec_nu_E-caso_general}
  \nu_E(h_{\ell-1})+ \lambda_{\ell-1} = \nu_E(h_k) + \lambda_k.
\end{equation}
By hypothesis, $\mathcal{F}$ satisfies the transversality property with $\mathcal{G}_\ell$ and $\mathcal{G}_k$, and hence we have the inequalities given in Lemma~\ref{lemma-desigualdades} applied to the decomposition given in \eqref{omega_ell_caso_general} which are
\begin{align*}
   \nu_E(h_{\ell-1} \omega_{\ell-1} \wedge \omega_\mathcal{F}) & \leq  \nu_E(h_k \omega_k \wedge \omega_\mathcal{F})  \\ 
  \nu_E(h_{\ell-1} \omega_{\ell-1} \wedge \omega_\mathcal{F}) & <  \nu_E( \omega_\ell \wedge \omega_\mathcal{F}) \\
  \nu_E(h_{\ell-1} \omega_{\ell-1} \wedge \omega_\mathcal{F}) & < \nu_E(h_{i} \omega_{i} \wedge \omega_\mathcal{F}) 
\end{align*}
for $i \neq k, {\ell-1}$, with $-1 \leq i < \ell-1$. Since
$$\omega_\ell \wedge \omega_\mathcal{F}  = \sum_{i=-1}^{\ell-1} h_i \omega_i \wedge \omega_\mathcal{F}$$
we conclude that
$$
\nu_E(h_{\ell-1} \omega_{\ell-1}\wedge \omega_\mathcal{F}) =\nu_E(h_k \omega_k \wedge \omega_\mathcal{F}).
$$
As in the previous case, taking into account that $\mathcal{F}$ satisfies the $k$-transversality property, we get
$$
\nu_E(h_{\ell-1}) + \lambda_{\ell-1} -  i_P(\mathcal{J}_{\mathcal{F},\mathcal{G}_{\ell-1}}',C')= \nu_E(h_k) + \lambda_k
$$
Then $i_P(\mathcal{J}_{\mathcal{F},\mathcal{G}_{\ell-1}}',C')=0$ as a consequence of equality~\eqref{ec_nu_E-caso_general}. This implies that $\mathcal{F}$ satisfies the $(\ell-1)$-transversality property as desired.
\end{proof}
Let us show how the transversality property allows to determine some analytic invariants of the invariant cusps of a totally dicritical foliation. More precisely, let us show that, if a totally $E$-dicritical foliation $\mathcal{F}$ satisfies the  $\ell$-transversality property, then the first  $\ell+2$ elements of the basis of the semimodule of differential values for any  cusp in $\operatorname{Cusp}_E(\mathcal{F})$ coincide.

\begin{Theorem}\label{th:valores-dif-tilde-C}
Let $\mathcal{F}$ be a totally $E$-dicritical foliation with $C$ as invariant curve and assume that $\mathcal{F}$ satisfies the $\ell$-transversality property for some $\ell \in \{1,2,\ldots,s\}$. Let $\widetilde{C}$ be any curve in $\operatorname{Cusp}_E(\mathcal{F})$ and $\mathcal{B}_{\widetilde{C}}=(\lambda_{-1}^{\widetilde{C}},\lambda_{0}^{\widetilde{C}},\lambda_{1}^{\widetilde{C}},\ldots,\lambda_{\ell}^{\widetilde{C}}, \lambda_{\ell+1}^{\widetilde{C}}, \ldots, \lambda_{\tilde{s}}^{\widetilde{C}})$ be the basis of the semimodule $\Lambda_{\widetilde{C}}$ of $\widetilde{C}$ with $\tilde{s}=s(\widetilde{C})$. Then $\tilde{s} \geq \ell$ and
$$
\lambda_{j}^{\widetilde{C}} = \lambda_{j}, \qquad \text{ for } j=-1,0,1,\ldots, \ell.
$$
In particular, $\Lambda_j^{\widetilde{C}}=\Lambda_j$ for $j=-1,0,1,\ldots, \ell$, and a minimal standard basis for the curve $\widetilde{C}$ is given by
$$
(\omega_{-1},\omega_0,\omega_1,\ldots,\omega_\ell,\omega_{\ell+1}^{\widetilde{C}}, \ldots, \omega_{\tilde{s}}^{\widetilde{C}} ).
$$
\end{Theorem}
\begin{proof} First note that $\lambda_{-1}^{\widetilde{C}}=\lambda_{-1}=n$ and $\lambda_{0}^{\widetilde{C}}=\lambda_{0}=m$ because $C$ and $\widetilde{C}$ are elements of $\operatorname{Cusp}(E)$. Consider $\omega_j$ an element of the minimal standard basis for $C$ with $1 \leq j \leq \ell$. By Theorem~\ref{Th:nu_C} and Proposition~\ref{Prop:l-transv-j-transv}, we have
$$
\lambda_j^{\widetilde{C}}=\nu_{\widetilde{C}}(\omega_j)=\nu_E(\omega_\mathcal{F}\wedge\omega_j)-\nu_E(\omega_\mathcal{F})=\nu_C(\omega_j)=\lambda_j
$$
Hence, $\lambda_i \in \Lambda_{\widetilde{C}}$ for $i=-1,0,1,\ldots, \ell$.

Let us prove that there is no $\tilde{\lambda} \in \Lambda_{\widetilde{C}} \setminus \Lambda_\ell$ with $\tilde{\lambda} < \lambda_\ell$. Assume that there exists such $\tilde\lambda$ and, without loss of generality, assume also that
$$
\tilde{\lambda}=\min \{ \lambda \in \Lambda_{\widetilde{C}} \ : \ \lambda \not \in \Lambda_\ell\}
$$
Denote $k=\max \{ j \ : \ \lambda_j < \tilde\lambda < \lambda_{j+1}\}$. Since $\tilde{\lambda} > m=\lambda_0$, then $k \geq 0$. In particular, we have that
$$
\Lambda_j^{\widetilde{C}}=\Lambda_j \text{ for } j=-1,0,1,\ldots, k,
$$
and consequently, by definition $u_i^{\widetilde{C}}=u_i$ for $i=1,\ldots, k,k+1$. Moreover, $t_1^{\widetilde{C}}=n+m=t_1$ and we have
$$t_{i+1}^{\widetilde{C}}=t_{i}^{\widetilde{C}}+ u_{i+1}^{\widetilde{C}}-\lambda_i^{\widetilde{C}}=t_{i+1} \quad \text{ for } i=1,\ldots, k.
$$
By the properties of the elements of the basis of the semimodule $\Lambda_{\widetilde{C}}$ of $\widetilde{C}$ given in Theorem~\ref{th_caracterizacion-omega-i},
the differential value $\tilde{\lambda}$ is equal to
$$
\tilde{\lambda}=\sup \{ \nu_{\widetilde{C}}(\omega) \ : \ \nu_E(\omega)=t_{k+1}^{\widetilde{C}}=t_{k+1}\}
$$
In particular we have that $\tilde{\lambda} \geq \nu_{\widetilde{C}}(\omega_{k+1})=\nu_C(\omega_{k+1})=\lambda_{k+1}$ which gives a contradiction with the fact $\tilde{\lambda} < \lambda_{k+1}$.
\end{proof}
\begin{Remark} The previous result shows that it is possible to determine part of the semimodule of differential values for the $E$-cusps of a totally dicritical foliation $\mathcal{F}$ satisfying the $\ell$-transversality property in the sense of the description given for the $E$-cusps of the foliations $\mathcal{G}_{\ell+1}$  (analytic semiroots of $C$) and $\widetilde{\mathcal{G}}_{\ell+1}$ in Section~\ref{sec:standard-systems} (see also \cite{Can-C-SS-2023,SS-2025}). In Section~\ref{sec:total-transv-property}, the $j$-transversality property of foliations  $\mathcal{G}_\ell$   and $\widetilde{\mathcal{G}}_\ell$ will be studied for $j<\ell$.
\end{Remark}

In the proof of Theorem~\ref{th:valores-dif-tilde-C}, we show that $\nu_C(\omega_i)=\nu_{\widetilde{C}}(\omega_i)$ for $-1 \leq i \leq \ell$  when $\widetilde{C} \in \operatorname{Cusp}_E (\mathcal{F})$ for a totally $E$-dicritical  foliation $\mathcal{F}$  satisfying the $\ell$-transversality property and the 1-forms $\omega_i$ belong to a minimal standard basis for $C$. Next proposition generalizes the above equality for any 1-form $\eta$ with divisorial value $\nu_E(\eta) \leq t_\ell$.

\begin{Proposition}\label{prop:nu_C-nu_tilde_C}
Let $\mathcal{F}$ be a totally $E$-dicritical foliation in $(\mathbb{C}^2,0)$ with $C$ as invariant curve. Assume that $\mathcal{F}$ satisfies the $\ell$-transversality property for some $\ell \in \{1,2,\ldots,s\}$. For any $1$-form $\eta$ with $\nu_E(\eta) \leq t_\ell$, we have that
$$\nu_C(\eta)=\nu_{\widetilde{C}}(\eta)
$$
for any $\widetilde{C} \in \operatorname{Cusp}_E(\mathcal{F})$.
\end{Proposition}
\begin{proof}
If $\eta$ is non-resonant, then $\nu_{\widetilde{C}}(\eta)=\nu_E(\eta)$ for any $\widetilde{C} \in \operatorname{Cusp}_E(\mathcal{F})$ by Remark~\ref{rm-nu_C-nu_E-no-resonante} and we obtain the result. Let us now assume that $\eta$ is resonant. By Lemma~\ref{lema_eta_resonant-omega_i}, there exist $i \in \{1,\ldots, \ell\}$ and $a,b \geq 0$ such that
$$\nu_E(\eta)=\nu_E(x^a y^b \omega_i), \quad \quad \nu_C(\eta) \leq \nu_C(x^a y^b \omega_i).$$
Take $i$ the maximum index satisfying the conditions above and write $\alpha=x^a y^b \omega_i$. Let us consider a 1-form $\omega_\mathcal{F}$ such that the foliation $\mathcal{F}$ is defined by $\omega_\mathcal{F}=0$. Since
$$\alpha \wedge \omega_\mathcal{F}= x^a y^b \omega_i \wedge \omega_{\mathcal{F}}
$$
we have that
$$
\nu_E(\alpha \wedge \omega_\mathcal{F})= \nu_E( x^a y^b) + \nu_E( \omega_i \wedge \omega_{\mathcal{F}}).
$$
By Proposition~\ref{Prop:l-transv-j-transv}, the foliation $\mathcal{F}$ satisfies the $i$-transversality property and hence $\nu_E( \omega_i \wedge \omega_{\mathcal{F}})= \nu_E(\omega_\mathcal{F}) + \nu_C(\omega_i)$. Consequently, we obtain that
\begin{equation}\label{eq:1-prop6-10}
\nu_E(\alpha \wedge \omega_\mathcal{F})= \nu_E( x^a y^b) + \nu_E(\omega_\mathcal{F}) + \nu_C(\omega_i) = \nu_E(\omega_\mathcal{F}) + \nu_C(\alpha).
\end{equation}
Let us denote $\nu_C(\alpha)=\lambda$ and we put $\omega=\omega_\mathcal{F}$ to simplify notations. We consider two cases:
\begin{description}
  \item[Case (i)] $\nu_C(\eta)=\nu_C(\alpha)$
  \item[Case (ii)] $\nu_C(\eta)<\nu_C(\alpha)$
\end{description}

\noindent{\bf Case (i).} Assume $\nu_C(\eta)=\nu_C(\alpha)=\lambda$.
Note that $\eta$ can be a non-saturated 1-form, hence we use Theorem~\ref{Th:nu_C} and Remark~\ref{Th:nu_C_1-forma-no-foliacion} to compute $\nu_C(\eta)$ and we get
\begin{equation}\label{eq:2-prop6-10}
\nu_C(\eta)=
\nu_E(\eta \wedge \omega)- \nu_E(\omega) + i_P(\mathcal{J}_{\omega,\eta}',C')
\end{equation}
where $P$ is the infinitely near point of $C$ in $E$.
From equations~\eqref{eq:1-prop6-10} and \eqref{eq:2-prop6-10}, we get
\begin{equation}\label{eq:nu_C_alpha}
\nu_C(\alpha) =
\nu_E(\alpha \wedge \omega)- \nu_E(\omega)= \nu_E(\eta \wedge \omega)- \nu_E(\omega) + i_P(\mathcal{J}_{\omega,\eta}',C') = \nu_C(\eta).
\end{equation}
If $i_P(\mathcal{J}_{\omega,\eta}',C')>0$, then
$$
\nu_E(\alpha \wedge \omega) >  \nu_E(\eta \wedge \omega).
$$
Take $\mu \in \mathbb{C}^*$ such that the 1-form $\theta=\alpha - \mu \eta$ satisfies that $\nu_C(\theta) > \nu_C(\alpha)$. Then
$$
\theta \wedge \omega    =  \alpha \wedge \omega     - \mu \   \eta  \wedge \omega.
$$
Since $\nu_E(\alpha \wedge \omega) >  \nu_E(\eta \wedge \omega)$,   we conclude that
\begin{equation}\label{eq:nu_E_theta}
\nu_E(\theta \wedge \omega) = \nu_E(\eta \wedge \omega).
\end{equation}
As before, we use Theorem~\ref{Th:nu_C} and Remark~\ref{Th:nu_C_1-forma-no-foliacion} to compute $\nu_C(\theta)=\nu_E(\theta \wedge \omega) - \nu_E(\omega)+i_P(\mathcal{J}_{\omega,\theta}',C')$. By equality~\eqref{eq:nu_E_theta} and the relationships given in~\eqref{eq:nu_C_alpha}, we get
$$
\nu_C(\theta) =\nu_C(\eta) - i_P(\mathcal{J}_{\omega,\eta}',C')+i_P(\mathcal{J}_{\omega,\theta}',C')
$$
and we deduce that
\begin{equation}\label{Ip_T}
i_P(\mathcal{J}_{\omega,\theta}',C') > i_P(\mathcal{J}_{\omega,\eta}',C').
\end{equation}
Note that the  curves $\mathcal{J}_{\omega,\theta}$, $\mathcal{J}_{\mathcal{F},\mathcal{G}_i}$ and $\mathcal{J}_{\omega,\eta}$ are given by
$$
\mathcal{J}_{\omega,\theta}=(T=0), \qquad \mathcal{J}_{\mathcal{F},\mathcal{G}_i}=(J_i=0), \qquad \mathcal{J}_{\omega,\eta}=(G=0)
$$
where $\omega  \wedge \theta=T dx \wedge dy$, $\omega  \wedge \alpha  = x^a y^b J_{i} dx \wedge dy$ and
$\omega  \wedge \eta =G dx \wedge dy$.
Moreover, we have
$$
T(x,y)=x^a y^b J_{i}(x,y) - \mu G(x,y).
$$
which implies
$$
i_P(\mathcal{J}_{\omega,\theta}',C') \geq \min \{i_P(\mathcal{J}_{\mathcal{F},\mathcal{G}_i}',C'), i_P(\mathcal{J}_{\omega,\eta}',C')\}
$$
and the equality holds if $i_P(\mathcal{J}_{\mathcal{F},\mathcal{G}_i}',C') \neq i_P(\mathcal{J}_{\omega,\eta}',C')$.
Since $\mathcal{F}$ satisfies the $i$-transversality property, then  $i_P(\mathcal{J}_{\mathcal{F},\mathcal{G}_i}',C')=0$ and we must have $i_P(\mathcal{J}_{\omega,\theta}',C') = 0$ which contradicts equation~\eqref{Ip_T}. It follows that $i_P(\mathcal{J}_{\omega,\eta}',C')=0$ and $\nu_E(\alpha \wedge \omega)=\nu_E(\eta \wedge \omega)$.

Consider now $\widetilde{C} \in \operatorname{Cusp}_E(\mathcal{F})$, with $\widetilde{C} \neq C$.
By Theorem~\ref{th:valores-dif-tilde-C}, we have that $\nu_{\widetilde{C}}(\alpha)=\nu_C(\alpha)$. Recall also that $\nu_C(\eta)=\nu_C(\alpha)$. Let us see that $\nu_{\widetilde{C}}(\eta)=\nu_C(\eta)$.

If $\nu_{\widetilde{C}}(\eta)>\nu_C(\eta)$, we can consider $\eta_1=\eta - \mu_1 \alpha$ such that $\nu_C(\eta_1)>\nu_C(\eta)=\nu_C(\alpha)$. We see that $\nu_{\widetilde{C}}(\eta_1)=\nu_{\widetilde{C}}(\alpha)=na + mb + \lambda_i=\lambda$. By Theorem~\ref{th_caracterizacion-omega-i}, we have that $\lambda \in \Lambda_{i-1}$.  As in the proof of Lemma~\ref{lema_eta_resonant-omega_i}, the condition $\lambda \in \Lambda_{i-1}$ implies the following possibilities:
\begin{itemize}
  \item either $\eta$ is reachable by $\omega_{i+1}$ which contradicts the maximality in the definition of $i$;
  \item or $\eta$ is reachable by $\tilde{\omega}_{i+1}$ and we have $\nu_E(\eta) \geq  \tilde{t}_{i+1} > t_{s+1} > t_\ell$ (by Lemma~\ref{lema_t_i_u_i}) against the hypothesis $\nu_E(\eta) \leq t_\ell$.
\end{itemize}
We deduce that $\nu_{\widetilde{C}}(\eta) \leq \nu_C(\eta)=\lambda$. Now, if $Q$ is the infinitely near point of $\widetilde{C}$ in $E$, by Theorem~\ref{Th:nu_C} and Remark~\ref{Th:nu_C_1-forma-no-foliacion}, we obtain
\begin{align*}
  \nu_{\widetilde{C}}(\eta) & = \nu_E(\eta \wedge \omega)- \nu_E(\omega) + i_Q(\mathcal{J}_{\omega,\eta}',\widetilde{C}')\\
   & \geq \nu_E(\eta \wedge \omega)- \nu_E(\omega) \\
   & = \nu_E(\alpha \wedge \omega)- \nu_E(\omega) = \lambda
\end{align*}
and then $\nu_{\widetilde{C}}(\eta)=\nu_C(\eta)=\lambda$. In particular, we obtain that $i_Q(\mathcal{J}_{\omega,\eta}',\widetilde{C}')=0$ for any curve $\widetilde{C} \in \operatorname{Cusp}_E(\mathcal{F})$. Hence, if $\eta$ is a saturated 1-form which defines a foliation $\mathcal{G}$, we deduce that $\mathcal{F}$ and $\mathcal{G}$ have the transversality property for the divisor $E$.

\medskip
\noindent{\bf Case (ii).} Assume $\nu_{C}(\eta)=\lambda'<\nu_C(\alpha)=\lambda$. By Remark~\ref{Th:nu_C_1-forma-no-foliacion}, the differential value $\nu_{C}(\eta)$ is given by
\begin{align*}
  \lambda'= \nu_{C}(\eta)& = \nu_E(\eta \wedge \omega)- \nu_E(\omega) + i_P(\mathcal{J}_{\omega,\eta}',{C}')  \\
   & < \nu_E(\alpha \wedge \omega)- \nu_E(\omega)=\nu_C(\alpha)=\lambda
\end{align*}
and we obtain
$$
\nu_E(\eta \wedge \omega)+i_P(\mathcal{J}_{\omega,\eta}',{C}') < \nu_E(\alpha \wedge \omega).
$$
Let us prove that $i_P(\mathcal{J}_{\omega,\eta}',{C}')=0$. If $i_P(\mathcal{J}_{\omega,\eta}',{C}')>0$, we consider the 1-form
$$
\theta=\alpha-\eta
$$
with $\nu_C(\theta)=\nu_C(\eta)=\lambda'$. As in Case (i), since $\nu_E(\eta \wedge \omega) < \nu_E(\alpha \wedge \omega)$, we get that
$$\nu_E(\theta \wedge \omega)=\nu_E(\eta \wedge \omega).
$$
Then, we have
\begin{align*}
  \lambda'=\nu_C(\theta) & = \nu_E(\theta \wedge \omega)- \nu_E(\omega) + i_P(\mathcal{J}_{\omega,\theta}',{C}')  \\
   &  = \nu_E(\eta \wedge \omega) - \nu_E(\omega) + i_P(\mathcal{J}_{\omega,\theta}',{C}') \\
   &  = \lambda' -i_P(\mathcal{J}_{\omega,\eta}',{C}') + i_P(\mathcal{J}_{\omega,\theta}',{C}')
\end{align*}
and hence we obtain
$$
i_P(\mathcal{J}_{\omega,\eta}',{C}') = i_P(\mathcal{J}_{\omega,\theta}',{C}').
$$
As in the previous case, we obtain that $i_P(\mathcal{J}_{\omega,\theta}',{C}')=0$ by the definition of $\theta$, the properties of the intersection multiplicity and the fact $i_P(\mathcal{J}_{\mathcal{F},\mathcal{G}_i},C')=0$ since $\mathcal{F}$ has the $i$-transversality property. Consequently, $i_P(\mathcal{J}_{\omega,\eta}',{C}') =0$.

Consider now $\widetilde{C} \in \operatorname{Cusp}_E(\mathcal{F})$, with $\widetilde{C} \neq C$. With a similar argument as in the previous case, we get that $\nu_{\widetilde{C}}(\eta) \leq \nu_{C}(\alpha)=\lambda$.

If $Q$ is the infinitely near point of $\widetilde{C}$ in $E$, we have
\begin{align*}
  \nu_{\widetilde{C}}(\eta) & = \nu_E(\eta \wedge \omega) -\nu_E(\omega)+ i_Q(\mathcal{J}_{\omega,\eta}',\widetilde{C}')\\
   &  \geq  \nu_E(\eta \wedge \omega) -\nu_E(\omega) = \nu_C(\eta)=\lambda '
\end{align*}
and then $\lambda'=\nu_{C}(\eta) \leq \nu_{\widetilde{C}}(\eta) \leq \lambda= \nu_C(\alpha)=\nu_{\widetilde{C}}(\alpha)$.

Note that although we chose the 1-form $\alpha$ applying Lemma~\ref{lema_eta_resonant-omega_i} for the curve $C$, this 1-form $\alpha$ also fulfills the properties for the curve $\widetilde{C}$ since $\nu_{\widetilde{C}}(\eta) \leq \nu_{\widetilde{C}}(\alpha)$.

Assume that $\nu_{\widetilde{C}}(\eta) < \nu_{\widetilde{C}}(\alpha)$. Following the arguments given at the beginning of Case (ii), we conclude that $i_Q(\mathcal{J}_{\omega,\eta},\widetilde{C}')=0$ and hence $\nu_{\widetilde{C}}(\eta)=\nu_C(\eta)$ as wanted.

If we assume that $\nu_{\widetilde{C}}(\eta) = \nu_{\widetilde{C}}(\alpha)$, we are in the situation considered in Case (i), but this implies that $\nu_{\widetilde{C}}(\eta)=\nu_{D}(\eta)$ for all $D \in \operatorname{Cusp}_E(\mathcal{F})$ and hence $\nu_{{C}}(\eta)=\nu_{\widetilde{C}}(\eta) = \nu_{\widetilde{C}}(\alpha)=\nu_C(\alpha)$ against the hypothesis in Case (ii).
\end{proof}
From the proof of the previous proposition we obtain
\begin{Corollary}
Let $\mathcal{F}$ be a totally $E$-dicritical foliation in $(\mathbb{C}^2,0)$ with $C$ as invariant curve. Assume that $\mathcal{F}$ satisfies the $\ell$-transversality property for some $\ell \in \{1,2,\ldots,s\}$. Let $\mathcal{G}$ be a  foliation  in $(\mathbb{C}^2,0)$ defined by a 1-form $\eta=0$ with $\nu_E(\eta)\leq t_\ell$, then $\mathcal{F}$ and $\mathcal{G}$ have the $E$-transversality property.
\end{Corollary}

As a consequence of the previous results we obtain that the transversality property is independent of the choice of the minimal standard basis for $C$.
\begin{Corollary}\label{Cor:indep-basis}
Let $\mathcal{F}$ be a totally $E$-dicritical foliation in $(\mathbb{C}^2,0)$ with $C$ as invariant curve. If $\mathcal{F}$ satisfies the $\ell$-transversality property relative to a minimal standard basis for $C$, then $\mathcal{F}$ satisfies the $\ell$-transversality property relative to all minimal standard basis for $C$.
\end{Corollary}

Note that the fact that $\mathcal{F}$ has the $E$-transversality property with a foliation $\mathcal{G}$ defined by a 1-form $\eta$ with $\nu_E(\eta)=t_\ell$ does not imply that $\mathcal{F}$ satisfies the $\ell$-transversality property as it is shown in the following example.
\begin{Example}
Let $\mathcal{F}$ be the foliation in $(\mathbb{C}^2,0)$ defined by $\omega_\mathcal{F}=0$ with
$$
\omega_\mathcal{F}=(x^6y-7y^6)dx + (6y^5x -x^7) dy.
$$
The curve $C$ given by $y^6-x^7-x^6y=0$ is an invariant curve of $\mathcal{F}$. Note that $C$ is a cusp with Puiseux pair $(7,6)$ and a Puiseux parametrization of $C$ is given by $\gamma(t)=(x(t),y(t))$ with
\begin{equation*}
\left\{
\begin{aligned}
  x(t) & = t^6 \\
  y(t) & = t^{7}+\tfrac{1}{6}t^{8}-\tfrac{1}{24}t^{9}+\tfrac{1}{81}t^{10}-\tfrac{91}{31104}t^{11}+ \cdots
\end{aligned}
\right.
\end{equation*}
Using the results and notations in \cite[Section 4]{Can-C-SS-2023} and Section~\ref{sec:E-dicritical},  the Newton polygon of $\mathcal{F}$ satisfies
$$\mathcal{N}(\mathcal{F}) \subset R^{6,7}(1,6)= \{(i,j) \ : \ i+j \geq 7 \} \cap \{(i,j) \ : \ 5i+6j \geq 41 \}$$ and $\operatorname{In}(\omega)= y^5(-7 y dx + 6 x dy)$ is a resonant 1-form, then the foliation $\mathcal{F}$ is a totally $E$-dicritical foliation, where $E$ is the cuspidal divisor in the minimal reduction of singularities of $C$.

The basis of the semimodule $\Lambda_C$ is equal to $\mathcal{B}=(6,7,15)$ and a standard basis for the curve $C$ is given by
\begin{equation}
\omega_{-1}=dx,\;\;\;\omega_{0}=dy,\;\;\;\omega_{1}=7ydx-\left(6x-\tfrac{1}{7}y\right)dy.
\end{equation}

Consider the foliation $\mathcal{G}$ defined by $\eta=0$ with
$$\eta=7ydx-6xdy.$$
The divisorial value of $\eta$ is equal to $\nu_E(\eta)=t_1=13$. Since
$\omega_{\mathcal{F}}\wedge \eta=x^{7}y dx\wedge dy,$
the strict transform of the Jacobian curve $\mathcal{J}_{\mathcal{F},\mathcal{G}}=(x^{7}y=0)$ by $\pi$ does not intersect $E$ and hence $\mathcal{F}$ and $\mathcal{G}$ have the $E$-transversality property.

Consider now the foliation $\mathcal{G}_{1}$ defined by $\omega_{1}=0$. We have
$$
\omega_{\mathcal{F}}\wedge \omega_{1} =  y(-y^{6}+x^{7}+\frac{1}{7}x^{6}y)dx\wedge dy
$$
and then the jacobian curve $\mathcal{J}_{\mathcal{F},\mathcal{G}_1}$ is given by $\mathcal{J}_{\mathcal{F},\mathcal{G}_1}=\mathcal{J}_{\mathcal{F},\mathcal{G}_1}^1 \cup \mathcal{J}_{\mathcal{F},\mathcal{G}_1}^2$ with $\mathcal{J}_{\mathcal{F},\mathcal{G}_1}^1=\{y=0\}$ and $\mathcal{J}_{\mathcal{F},\mathcal{G}_1}^2=\{-y^{6}+x^{7}+\frac{1}{7}x^{6}y=0\}$. Hence, the strict transform of $\mathcal{J}_{\mathcal{F},\mathcal{G}_1}^2$ by $\pi$ intersects the cuspidal divisor $E$ in a free point. Then $\mathcal{F}$ does not satisfies the 1-transversality property.

\end{Example}
\section{Total transversality property}\label{sec:total-transv-property}
In this section we introduce the notion of total transversality property for a foliation $\mathcal{F}$ which means that the foliations $\mathcal{F}$ and $\mathcal{G}_\ell$ have the transversality property for the cuspidal divisor for $\ell=1,\ldots,s$. We give  conditions in terms of the divisorial value of a 1-form defining $\mathcal{F}$ which allow to assure that $\mathcal{F}$ have the total transversality property (see Theorem~\ref{Th:total-transv-property}).

As in the previous section, consider a morphism $\pi: (M,D) \to (\mathbb{C}^2,0)$ obtained from a cuspidal sequence and let $E$ be the cuspidal divisor. We fix a curve $C \in \operatorname{Curv}(E)$ with Puiseux pair $(m,n)$. Let $\mathcal{B}=(\lambda_{-1},\lambda_0,\lambda_1,\ldots,\lambda_s)$ be the basis of the semimodule $\Lambda_C$ and fix a minimal standard basis $(\omega_{-1},\omega_0,\omega_1,\ldots,\omega_s)$ for $C$.
\begin{Definition}
Let $\mathcal{F}$ be a totally $E$-dicritical foliation in $(\mathbb{C}^2,0)$ with $C$ as invariant curve. We say that the foliation $\mathcal{F}$ has the {\em total transversality property\/} if $\mathcal{F}$ satisfies the $s$-transversality property.
\end{Definition}
As a consequence of Proposition~\ref{Prop:l-transv-j-transv}, if $\mathcal{F}$ has the total transversality property then $\mathcal{F}$ has the $\ell$-transversality property for $-1\leq \ell \leq s$.
We start this section proving that   the foliations $\mathcal{G}_{s+1}$ and $\widetilde{\mathcal{G}}_{s+1}$ satisfy the total transversality property (see Theorem~\ref{Th:transv-omega_s+1}), where  $\mathcal{G}_{s+1}$ and $\widetilde{\mathcal{G}}_{s+1}$ are defined by the 1-forms $\omega_{s+1}=0$ and $\widetilde{\omega}_{s+1}=0$ respectively.  First we prove a technical lemma.
\begin{Lemma}\label{lemma_t_i*_lambda_i}
Consider a cusp $C \in \operatorname{Curv}(E)$ and take $(\mathcal{E},\widetilde{\mathcal{E}})$ a standard system for the curve $C$ with $\mathcal{E}=(\omega_{-1},\omega_0,\omega_1,\ldots,\omega_s,\omega_{s+1})$ and $\widetilde{\mathcal{E}}=(\widetilde{\omega}_1,\ldots,\widetilde{\omega}_s,\widetilde{\omega}_{s+1})$. For any $\ell \in \{-1,0,1,\ldots, s\}$, we have
$$
\nu_E(\omega_i^* \wedge \omega_\ell)=t_i^* + \lambda_\ell, 
$$
where  $i > \ell$ and $1 \leq  i \leq s+1$, the 1-form $\omega_i^* \in \{\omega_i,\widetilde{\omega}_i\}$ and $t_i^*=\nu_E(\omega_i^*)$.
\end{Lemma}
\begin{proof}
By Remark~\ref{Rm:-1-0-transv}, any totally $E$-dicritical foliation satisfies the $\ell$-transversality property for $\ell=-1,0$. In particular, we have the equalities given in expression~\eqref{eq:Rm:-1-0-transv} which imply the result for all $1 \leq i \leq s+1$ and $\ell=-1,0$. Let us prove the result by induction on the pair $(\ell,i)$ with $1 \leq \ell < i$.

We start computing the divisorial value $\nu_E(\omega_2^* \wedge \omega_1)$. We write $\omega_2^*$ using the decomposition of  Delorme (see Theorem~\ref{Th:descomp-Delorme}) and we obtain
\begin{equation}\label{eq:delorme_omega_*_2}
\omega_2^*=h_1 \omega_1 + h_0 \omega_0 + h_{-1}\omega_{-1}
\end{equation}
with $\nu_C(h_1 \omega_1)=u_2^* <nm$ and there exists a unique index $k \in \{-1,0\}$ with
\begin{equation}\label{eq:h_1-h_k}
\nu_C(h_k \omega_k)= \nu_C(h_1 \omega_1) =t_2^*-t_1+\lambda_1.
\end{equation}
We assume that $k=0$ (the case $k=-1$ works in a similar way). Hence, we have $\nu_C(h_0 \omega_0) <nm$ and then $\nu_C(h_0)=\nu_E(h_0)$  by Lemma~\ref{lemma_nu_h_omega}.
From equation~\eqref{eq:delorme_omega_*_2}, we obtain that
$$
\omega^*_2 \wedge \omega_1 = h_0 \omega_0 \wedge \omega_1 + h_{-1}\omega_{-1}\wedge \omega_1.
$$
The equality in~\eqref{eq:h_1-h_k} implies that $\nu_E(h_0)=\nu_C(h_0)=t_2^*-t_1+\lambda_1 -\lambda_0$. Then we have
$$
\nu_E(h_0 \omega_0 \wedge \omega_1)=\nu_E(h_0) + \nu_E(\omega_0 \wedge \omega_1) = t_2^*-t_1+\lambda_1 -\lambda_0+ t_1+\lambda_0=t_2^*+\lambda_1.
$$
Let us compute now $\nu_E(h_{-1}\omega_{-1}\wedge \omega_1)$. We consider two cases: $\nu_E(h_{-1}) \geq nm$ and  $\nu_E(h_{-1}) <nm$. In the first case, by Lemma~\ref{lema_t_i_u_i}, we have  $\nu_E(h_{-1}) \geq nm > u_2^*$ and then
\begin{align*}
  \nu_E(h_{-1}\omega_{-1}\wedge \omega_1) & = \nu_E(h_{-1}) + \nu_E(\omega_{-1}\wedge \omega_1) \\
   &  \geq nm + \nu_E(\omega_{-1}\wedge \omega_1) \\
   & > u_2^* + t_1+\lambda_{-1} = t^*_2-t_1+\lambda_1+t_1+\lambda_{-1} \\
   & > t_2^*+\lambda_1.
\end{align*}
In the case, $\nu_E(h_{-1}) <nm$, we have
\begin{align*}
  \nu_E(h_{-1}\omega_{-1}\wedge \omega_1) & = \nu_C(h_{-1}) + \nu_E(\omega_{-1}\wedge \omega_1) \\
   &  > t_2^*-t_1+ \lambda_1 -\lambda_{-1} + t_1+\lambda_{-1} \\
   & = t_2^*+\lambda_1.
\end{align*}
From the computations above, we deduce that
$$
\nu_E(\omega_2^*\wedge \omega_1)=t_2^*+\lambda_1
$$
as wanted. Let us now assume by induction hypothesis that for any $\ell, k$ with $1 \leq \ell < k \leq i-1$ we have that
$$
\nu_E(\omega_k^* \wedge \omega_\ell)=t_k^* + \lambda_\ell
$$
and let us prove the general case: $\nu_E(\omega_i^*\wedge \omega_\ell)=t_i^*+\lambda_\ell$ for any $\ell$ such that $1 \leq \ell < i$. As before, we write $\omega_i^*$ using the decomposition of Delorme given in Theorem~\ref{Th:descomp-Delorme}
\begin{equation}\label{eq:Delorme_omega_i_*}
\omega_i^*=h_{\ell}\omega_{\ell} + h_k \omega_k + \sum_{{j=-1}\atop {j\neq k}}^{\ell-1} h_j \omega_j
\end{equation}
with $\nu_C(h_{\ell}\omega_{\ell})=\nu_C(h_k\omega_k)=t_i^*-t_{\ell}+\lambda_{\ell}$ and $\nu_C(h_j \omega_j) > \nu_C(h_{\ell}\omega_{\ell})$ for $j \neq \ell,k$. Thus, we obtain
\begin{equation}\label{eq:omega_i_*_wedge_omega_ell}
\omega_i^* \wedge \omega_\ell= h_k \omega_k \wedge \omega_\ell + \sum_{{j=-1}\atop {j\neq k}}^{\ell-1} h_j \omega_j \wedge \omega_\ell.
\end{equation}
Let us compute $\nu_E( h_k \omega_k \wedge \omega_\ell)$. By induction hypothesis, we have that $\nu_E(\omega_k \wedge \omega_\ell)=t_\ell + \lambda_k$ since $k < \ell$. By Lemma~\ref{lema-lambda-t-u}, the differential value $\nu_{C}(h_k \omega_k) < u_i^*$ and hence $\nu_E(h_k)=\nu_C(h_k)$ by Lemma~\ref{lemma_nu_h_omega}. Moreover,
since $\nu_C(h_k)=t_i^*-t_\ell+\lambda_\ell -\lambda_k$, we obtain
\begin{equation}\label{eq:nu:E_h_k_omegakell}
\nu_E( h_k \omega_k \wedge \omega_\ell)= \nu_E(h_k) + \nu_E(\omega_k \wedge \omega_\ell)= 
t_i^*+\lambda_\ell.
\end{equation}
We compute now $\nu_E(h_j \omega_j \wedge \omega_\ell)$ for $j\neq \ell,k$.  By induction hypothesis, we have that $\nu_E(\omega_j \wedge \omega_\ell)=t_\ell + \lambda_j$ since $j < \ell$. We consider now two cases: $\nu_E(h_j)\geq nm$ and $\nu_E(h_j)< nm$. In the first case, we have that $\nu_E(h_j)\geq nm  > u_{i}^*$ by Lemma~\ref{lema_t_i_u_i} and then
\begin{align*}
  \nu_E(h_j \omega_j \wedge \omega_\ell) & = \nu_E(h_j)+t_\ell+\lambda_j \\
   & >u_{i}^* + t_\ell+\lambda_j > u_{i}^* + t_\ell.
  \end{align*}
In the case $\nu_E(h_j)< nm$, we obtain $\nu_C(h_j)=\nu_E(h_j)$. Moreover, as a consequence of item (iv) in Theorem~\ref{Th:descomp-Delorme}, we have $u_i^* < \nu_C(h_j \omega_j)$ and then $\nu_E(h_j) > u_i^*- \lambda_j$. Consequently,
\begin{align*}
  \nu_E(h_j \omega_j \wedge \omega_\ell) & = \nu_E(h_j)+t_\ell+\lambda_j \\
   & >u_i^* -\lambda_j + t_\ell+\lambda_j \\
   & = u_i^*+t_\ell.
   \end{align*}
Let us show that $u_i^*+t_\ell \geq t_i^*+\lambda_\ell$. This inequality is equivalent to  $u_i^*-t_i^*\geq \lambda_\ell - t_\ell$. Note that $u_i^*-t_i^*=\lambda_{i-1}- t_{i-1}$, we have to prove that $\lambda_{i-1}- t_{i-1} \geq \lambda_\ell - t_\ell$ which is consequence of statement $(iii)$ in Lemma~\ref{lema_t_i_u_i} since $\ell \leq i-1$. We conclude that
\begin{equation}\label{eq:nu:E_h_j_omegajell}
\nu_E(h_j \omega_j \wedge \omega_\ell) > t_i^*+\lambda_\ell \qquad \text{ for } j \neq \ell,k.
\end{equation}
From expressions~\eqref{eq:omega_i_*_wedge_omega_ell}, \eqref{eq:nu:E_h_k_omegakell} and \eqref{eq:nu:E_h_j_omegajell}, we obtain that
$$\nu_E(\omega_i^* \wedge \omega_\ell)=t_i^* + \lambda_\ell
$$
and this finishes the proof of the lemma.
\end{proof}
As a consequence of the equality given in the previous lemma we obtain the following result
\begin{Theorem}\label{Th:transv-omega_s+1}
Consider a cusp $C \in \operatorname{Cusp}(E)$ and take $(\mathcal{E},\widetilde{\mathcal{E}})$ a standard system for the curve $C$ with $\mathcal{E}=(\omega_{-1},\omega_0,\omega_1,\ldots,\omega_s,\omega_{s+1})$ and $\widetilde{\mathcal{E}}=(\widetilde{\omega}_1,\ldots,\widetilde{\omega}_s,\widetilde{\omega}_{s+1})$.

For any $\ell \in \{1,\ldots, s, s+1\}$, the foliation $\mathcal{G}_\ell^*$ satisfies the $j$-transversality property for any $-1 \leq j <\ell$, where $\mathcal{G}_\ell^*$ is the foliation defined by $\omega_\ell^*=0$ with $\omega_\ell^* \in \{\omega_\ell,\widetilde{\omega}_\ell\}$ and $\omega_\ell^* \neq \widetilde{\omega}_1$.

In particular, the foliations $\mathcal{G}_{s+1}$ and $\widetilde{\mathcal{G}}_{s+1}$ satisfy the total transversality property.
\end{Theorem}
\begin{proof}
By Proposition~\ref{Prop:l-transv-j-transv}, it is enough to show that the foliation $\mathcal{G}_\ell^*$ satisfies the $(\ell-1)$-transversality property. We can assume that $\ell \geq 2$ since the other cases are consequence of Remark~\ref{Rm:-1-0-transv}.

Let $\mathcal{J}=\mathcal{J}_{\mathcal{G}_\ell^*,\mathcal{G}_{\ell-1}}$ be the jacobian curve of the foliations $\mathcal{G}_\ell^*$ and $\mathcal{G}_{\ell-1}$. If the foliations $\mathcal{G}_\ell^*$ and $\mathcal{G}_{\ell-1}$ do not have the $E$-transversality property, then the strict transform $\mathcal{J}'$ of the jacobian curve $\mathcal{J}$ by the morphism $\pi$ intersects the divisor $E$ at a point $P \in E$. Take $C_P \in \operatorname{Cusp}_E(\mathcal{G}_\ell^*)$ such that $P$ is the infinitely near point of $C_P$ in $E$. By Theorem~\ref{Th:separatrices-tilde}, we have that $\nu_{C_P}(\omega_{\ell-1})=\lambda_{\ell-1}$. On the other hand, we can compute $\nu_{C_P}(\omega_{\ell-1})$ using expression \eqref{ec:muC-muE-2-formas} given in Theorem~\ref{Th:nu_C} and taking into account that $\nu_E(\omega_\ell^* \wedge \omega_{\ell -1})=  t_\ell^*+\lambda_{\ell-1}$ by Lemma~\ref{lemma_t_i*_lambda_i}, we obtain
\begin{align*}
  \nu_{C_P}(\omega_{\ell-1})=\lambda_{\ell-1} & = \nu_E(\omega_\ell^* \wedge \omega_{\ell -1}) - \nu_E(\omega_\ell^*) + i_P(\mathcal{J}',C_P') \\
   & =t_\ell^*+\lambda_{\ell-1} - t_\ell^*  + i_P(\mathcal{J}',C_P') \\
   & =  \lambda_{\ell-1} + i_P(\mathcal{J}',C_P') > \lambda_{\ell-1}
\end{align*}
which gives a contradiction.
\end{proof}
The following example shows that there are foliations $\mathcal{F}$ defined by a 1-form $\omega_\mathcal{F}$ with $\nu_E(\omega_\mathcal{F}) \neq t_{s+1}^*$ and satisfying the total transversality property.
\begin{Example}\label{Ex:total-transv-nuE-distinto-t}
Let $C$ be the cusp with Puiseux pair $(9,4)$ given by the Puiseux parametrization $\varphi(t)=(x(t),y(t))$ with
\begin{equation}\label{eq:cusp4-9}
\left\{
\begin{aligned}
  x(t) & = t^4 \\
  y(t) & = t^9+t^{10} -\tfrac{1}{2}t^{11} + \tfrac{7}{8}t^{13}
\end{aligned}
\right.
\end{equation}
The basis of the semimodule $\Lambda_C$ is $\mathcal{B}=(4,9,14,19)$ and a standard basis for $C$ is given by $(dx,dy,\omega_1,\omega_2)$ where
\begin{equation}\label{base-standar-ex-9-4}
\omega_1=9ydx-4xdy; \qquad \omega_2=9xy dx - \left(4x^2-\tfrac{4}{9} y\right) dy.
\end{equation}
In Example~\ref{ex:apendice-ki} we compute the axes and critical values which are given by
\begin{equation*}
  \begin{aligned}
   u_1 & = t_1 = 13, \quad & u_2 & = 18, & \quad t_2 & = 17, \quad & u_3 & = 23,  \quad & t_3 & = 21, \\
   \tilde{u}_1 &  =\tilde{t}_1 = 36, \quad & \tilde{u}_2 & = 32, &\tilde{t}_2 & = 31, \quad &\tilde{u}_3 & =28,  \quad & \tilde{t}_3 &=26.
   \end{aligned}
\end{equation*}
Let $E$ be the cuspidal divisor of the minimal reduction of singularities of $C$ and consider   the  totally $E$-dicritical foliation $\mathcal{F}$ defined by $\omega_\mathcal{F}=0$ where
$$
\omega_\mathcal{F}=\left( 9x^3y + x y^2 + \tfrac{\partial h}{\partial x} \right) dx - \left(4 x^4+\tfrac{4}{3} y^2 - \tfrac{\partial h}{\partial y} \right) dy
$$
with $h \in \mathbb{C}\{x,y\}$ such that $C$ is an invariant curve of $\omega_\mathcal{F}$ and hence $\nu_E(h)\geq 28$. Note that
\begin{align*}
\omega_\mathcal{F}\wedge \omega_1 & =4 y^2( 3 y -x^2 ) dx \wedge dy + dh \wedge \omega_1, \\
\omega_\mathcal{F}\wedge \omega_2 & = \tfrac{112}{9}xy^3 dx \wedge dy + dh \wedge \omega_2.
\end{align*}
and then the foliation $\mathcal{F}$ satisfies the total transversality property. However, we have that $\nu_E(\omega_\mathcal{F})=25$ which is different from $t_3$ and $\tilde{t}_3$.
\end{Example}
Next result gives a sufficient condition in terms of the divisorial value of the 1-form defining a foliation to assure that a totally $E$-dicritical foliation satisfies the total transversality property.
\begin{Theorem}\label{Th:total-transv-property}
Consider a cusp $C \in \operatorname{Cusp}(E)$ and let $\mathcal F$ be a totally $E$-dicritical foliation with $C$ as invariant curve. If the divisorial value
$$
\nu_E(\omega_\mathcal{F})\in [t_{s+1},t_{s+1}+\tilde{t}_{s+1}-t_s),
$$
then the foliation $\mathcal{F}$ satisfies the total transversality property.
\end{Theorem}
Let us prove a technical lemma before giving the proof of the theorem above.

\begin{Lemma}\label{lemma_reachable_only_one}
Let $C$ be a cusp and consider a 1-form $\omega$ such that $C$ is an invariant curve of the foliation defined by $\omega=0$.
 If $\nu_E(\omega)\in [t_{s+1},t_{s+1}+\tilde{t}_{s+1}-t_s)$, then $\omega$ is reachable from $\omega_{s+1}$ or from $\widetilde{\omega}_{s+1}$ but not from both.
\end{Lemma}
\begin{proof}

From the definition of the critical values in Section~\ref{sec:cuspidal-semimodules}, we have that
$$
t_{s+1}+\tilde{t}_{s+1}-t_s=t_s+ n \ell_{s+1}^n +  m \ell_{s+1}^m,
$$
where $\ell_{s+1}^n$ and $\ell_{s+1}^m$ are  limits of $\Lambda_C$.
By Theorem~\ref{Th:Saito_basis}, there exist $A,B \in \mathcal{O}_{\mathbb{C}^2,0}$ such that
\begin{equation*}
\omega= A  \, \omega_{s+1} + B \, \widetilde{\omega}_{s+1}.
\end{equation*}
Let us see that $\nu_E(A   \, \omega_{s+1}) \ne \nu_E ( B \, \widetilde{\omega}_{s+1})$. The equality  $\nu_E(A   \, \omega_{s+1}) =\nu_E ( B \, \widetilde{\omega}_{s+1})$ implies that
\begin{equation}\label{eq:na-mb}
na + mb + t_{s+1}=na'+mb'+\tilde{t}_{s+1} < t_{s+1}+\tilde{t}_{s+1}-t_s
\end{equation}
for some $a,a',b,b' \in \mathbb{Z}_{\geq 0}$. Assume for instance that $t_{s+1}=t_{s+1}^n=t_s+n\ell^n_{s+1}$ and $\tilde{t}_{s+1}=t_{s+1}^m=t_s+m\ell^m_{s+1}$. Hence, from equation~\eqref{eq:na-mb}, we deduce that
$$
n(a + \ell_{s+1}^n) + mb =na'+m (b'+ \ell^m_{s+1}) < n \ell_{s+1}^n + m \ell^m_{s+1}
$$
By Corollary~\ref{corollary_limites}, we have that $n \ell_{s+1}^n + m \ell^m_{s+1}<nm$ and we deduce that
\begin{equation}\label{eq:a-b}
a + \ell_{s+1}^n= a'; \qquad b=b'+ \ell^m_{s+1}.
\end{equation}
Since $\nu_E(\omega) < t_{s+1} + \tilde{t}_{s+1}-t_s$, then we have
$$
n a + mb + t_{s+1} < t_{s+1} + \tilde{t}_{s+1}-t_s=t_{s+1}+m \ell_{s+1}^m
$$
and consequently $n a + mb < m \ell_{s+1}^m$. In particular, we obtain that $b < \ell_{s+1}^m$ against the expression of $b$ given in \eqref{eq:a-b}. Hence, we obtain that $\nu_E(A   \, \omega_{s+1}) \ne \nu_E ( B \, \widetilde{\omega}_{s+1})$.

Let us consider the case $\nu_E(A   \, \omega_{s+1}) < \nu_E ( B \, \widetilde{\omega}_{s+1})$. Then $\operatorname{In}(\omega)=\operatorname{In} (A \omega_{s+1})$ and $\omega$ is reachable from $\omega_{s+1}$. If $\omega$ were also reachable from $\widetilde{\omega}_{s+1}$, then it would exist a monomial $h$ with $\nu_E(A \omega_{s+1})= \nu_E(h \widetilde{\omega}_{s+1})$ but we have just seen that this equality is not possible. The other case $\nu_E ( B \, \widetilde{\omega}_{s+1}) < \nu_E(A   \, \omega_{s+1})$ works in a similar way.
\end{proof}

\begin{Remark}
Given a 1-form $\omega$ satisfying the hypothesis of the previous lemma, we have that
\begin{itemize}
  \item[(i)] if $\nu_E(\omega) \in [t_{s+1},\tilde{t}_{s+1})$, then $\omega$ is reachable from $\omega_{s+1}$;
  \item[(ii)] if $\omega$ is reachable from $\widetilde{\omega}_{s+1}$, then $\nu_E(\omega) \in [\tilde{t}_{s+1},t_{s+1}+\tilde{t}_{s+1}-t_s)$.
\end{itemize}
\end{Remark}
\begin{proof}[Proof of Theorem~\ref{Th:total-transv-property}]
Consider a totally $E$-dicritical foliation $\mathcal{F}$ defined by $\omega_\mathcal{F}=0$ with $C$ as invariant curve and such that
$\nu_E(\omega_\mathcal{F})\in [t_{s+1},t_{s+1}+\tilde{t}_{s+1}-t_s)$. By Theorem~\ref{Th:Saito_basis}, there exist $A,B \in \mathcal{O}_{\mathbb{C}^2,0}$ such that
\begin{equation}\label{ec:omega_F_s+1}
\omega_{\mathcal{F}}= A  \, \omega_{s+1} + B \, \widetilde{\omega}_{s+1}.
\end{equation}
By Lemma~\ref{lemma_reachable_only_one}, can assume that $\omega_\mathcal{F}$ is reachable from $\omega_{s+1}$ but not from $\widetilde{\omega}_{s+1}$, and hence $\nu_E(\omega_\mathcal{F})=\nu_E(A \omega_{s+1})$. The other case runs in a similar way.
Hence, we have that
$$
\nu_E(A \omega_{s+1}) < \nu_E(B \widetilde{\omega}_{s+1})
$$
and then
$$
\lambda_s + \nu_E(A \omega_{s+1}) < \lambda_s + \nu_E(B \widetilde{\omega}_{s+1})
$$
By Lemma~\ref{lemma_t_i*_lambda_i}, we deduce that
$$
\lambda_s = \nu_E(\omega_{s+1} \wedge \omega_s)-\nu_E(\omega_{s+1})=\nu_E(\widetilde{\omega}_{s+1} \wedge \omega_s)-\nu_E(\widetilde{\omega}_{s+1})
$$
and we obtain that
\begin{align*}
  \lambda_s + \nu_E(A \omega_{s+1}) &= \nu_E(\omega_{s+1} \wedge \omega_s)-\nu_E(\omega_{s+1}) +\nu_E(A \omega_{s+1}) \\ & = \nu_E(\omega_{s+1} \wedge \omega_s) + \nu_E(A) \\
   \lambda_s+\nu_E(B \widetilde{\omega}_{s+1}) & = \nu_E(\widetilde{\omega}_{s+1} \wedge \omega_s)-\nu_E(\widetilde{\omega}_{s+1}) + \nu_E(B \widetilde{\omega}_{s+1}) \\ & = \nu_E(\widetilde{\omega}_{s+1} \wedge \omega_s) + \nu_E(B).
\end{align*}
Consequently, we obtain that $\nu_E(A\omega_{s+1} \wedge \omega_s) < \nu_E(B \widetilde{\omega}_{s+1} \wedge \omega_s)$.
From this inequality and the expression given in~\eqref{ec:omega_F_s+1}, we get that
$$
\nu_E(\omega_\mathcal{F} \wedge \omega_s)=\nu_E(A\omega_{s+1} \wedge \omega_s) < \nu_E(B \widetilde{\omega}_{s+1} \wedge \omega_s).
$$
By hypothesis, $\nu_E(\omega_\mathcal{F}) =\nu_E(A \omega_{s+1})\in [t_{s+1},t_{s+1}+\tilde{t}_{s+1}-t_s)$, then the divisorial value of the coefficient $A$ satisfies $\nu_E(A)<\tilde{t}_{s+1}-t_s<nm$ and we get that the initial part $\operatorname{In}(A)$ is equal to a monomial $A_{\alpha,\beta} x^\alpha y^\beta$ with $A_{\alpha,\beta}$ a constant.

Hence, if the jacobian curve $\mathcal{J}_{\mathcal{F},\mathcal{G}_s}$ is defined by $J=0$  and the jacobian curve $\mathcal{J}_{\mathcal{G}_{s+1},\mathcal{G}_s}$ is given by $J_{s+1,s}=0$, we have that, up to multiply by a constant, the initial parts satisfy the following equality
\begin{equation}\label{eq:in-jacobian}
\operatorname{In}(J)=  x^\alpha y^\beta \operatorname{In}(J_{s+1,s}).
\end{equation}
By Theorem~\ref{Th:transv-omega_s+1}, the foliations $\mathcal{G}_{s+1}$ and $\mathcal{G}_s$ have the $E$-transversality property, then the strict transform of $\mathcal{J}_{\mathcal{G}_{s+1},\mathcal{G}_s}$ by $\pi$ does not intersect the divisor $E$. From the expression of the initial part of $J$ given in~\eqref{eq:in-jacobian}, we deduce that the strict transform of  $\mathcal{J}_{\mathcal{F},\mathcal{G}_s}$ by $\pi$ does not intersect the divisor $E$ and hence the foliation $\mathcal{F}$ satisfies the $s$-transversality property. Consequently, the foliation  $\mathcal{F}$ satisfies the total transversality property by Proposition~\ref{Prop:l-transv-j-transv}.
\end{proof}
Next example shows that the upper bound in the previous result is sharp.

\begin{Example}\label{ex:sharp-bound}
Consider the curve $C$ given by the parametrization $\varphi(t)=(x(t),y(t))$ given by the expression~\eqref{eq:cusp4-9} in Example~\ref{Ex:total-transv-nuE-distinto-t}. Let us consider the 1-forms
\begin{align*}
\omega_3 &=\left(9x^2y +28 y^2+10xy^2+\tfrac{\partial h_1}{\partial x}\right)dx - \left(4x^3+12xy+\tfrac{23}{6} y^2-\tfrac{\partial h_1}{\partial y}\right)dy \\
\widetilde{\omega}_3 &=\left(9xy^2 - \tfrac{28}{9} x^6+\tfrac{\partial h_2}{\partial x}\right)dx - \left(4x^2y-\tfrac{4}{9}y^2-\tfrac{\partial h_2}{\partial y}\right)dy
\end{align*}
with $h_1,h_2 \in \mathbb{C}\{x,y\}$ such that $C$ is an invariant curve of $\omega_3$ and $\tilde{\omega}_3$. Note that $\nu_E(\omega_3)=t_3=21$ and $\nu_E(\widetilde{\omega}_3)=\tilde{t}_3=26$.

Consider now the foliation $\mathcal{F}$ given by the 1-form $\omega=0$ where
$$
\omega=\left(9x^2y^2+y^3-3x^7+\tfrac{\partial h}{\partial x}\right)dx - \left(4x^3y+\tfrac{4}{3}x^6-\tfrac{\partial h}{\partial y} \right) dy
$$
with $h\in \mathbb{C}\{x,y\}$ such that $\varphi^* \omega=0$ which implies that $\nu_E(h)\geq 32$.
The divisorial value of $\omega$ is equal to  $\nu_E(\omega)=30=t_{3}+\tilde{t}_3-t_2=21+26-17$ and it can be checked that the 1-form $\omega$ is reachable by $\omega_3$ and $\widetilde{\omega}_3$.

Let $\omega_2$ be the element of the standard basis for $C$ given in~\eqref{base-standar-ex-9-4} in Example~\ref{Ex:total-transv-nuE-distinto-t}. The 2-form $\omega \wedge \omega_2$ is equal to
$$
\omega \wedge \omega_2= \left(\tfrac{4}{9}y^4+12x^9+\tfrac{32}{3}x^7y\right) dx \wedge dy + dh \wedge \omega_2
$$
and hence the jacobian curve $\mathcal{J}_{\mathcal{F},\mathcal{G}_2}$ intersects the cuspidal divisor $E$. Consequently, the foliation $\mathcal{F}$ does not satisfy the 2-transversality property and hence $\mathcal{F}$ does not have the total transversality property. Note that $\mathcal{F}$ satisfies the 1-transversality property since
$$
  \omega \wedge \omega_1=(-4xy^3+12x^8+12yx^6)dx \wedge dy +dh \wedge \omega_1.
$$
\end{Example}

\section{Families of curves with constant semimodule}\label{sec:constant-semimodule}
Let $\mathcal{F}$ be a cuspidal dicritical foliation with $E$ the dicritical divisor. In general, two curves $C, \widetilde{C} \in \operatorname{Cusp}_E(\mathcal{F})$ does not have the same semimodule of differential values.
As a consequence of the results in Section~\ref{sec:transv-polar}, we get that the first $\ell$ elements of the basis of $\Lambda_C$ and $\Lambda_{\widetilde{C}}$ are the same provided that $\mathcal{F}$ satisfies the $\ell$-transversality property, with $\ell \in \{1,\ldots,s\}$ (see Theorem~\ref{th:valores-dif-tilde-C}). This section is devoted to give conditions which assure that all the $E$-cusps of $\mathcal{F}$ have the same semimodule of differential values (see Theorem~\ref{th:lambda-cte}). We introduce the following definition.

\begin{Definition}
We say that a cuspidal dicritical foliation $\mathcal{F}$ is {\em $\Lambda$-constant} if for all curves $C,\widetilde{C} \in \operatorname{Cusp}_E(\mathcal{F})$ we have that $\Lambda_C=\Lambda_{\widetilde{C}}$.
\end{Definition}
Let us fix a cusp $C \in \operatorname{Cusp}(E)$ and take $(\mathcal{E},\tilde{\mathcal{E}})$ a standard system for $C$ with $\mathcal{E}=(\omega_{-1},\omega_0,\omega_1,\ldots,\omega_s,\omega_{s+1})$ and $\widetilde{\mathcal{E}}=(\widetilde{\omega}_1,\ldots,\widetilde{\omega}_s,\tilde{\omega}_{s+1})$. As before, we denote by $\mathcal{G}_{\ell}$ the foliation defined by $\omega_{\ell}=0$ and $\widetilde{\mathcal{G}}_{\ell}$ the foliation defined by $\widetilde{\omega}_{\ell}=0$ for $\ell \in \{1,2,\ldots,s,s+1\}$. From the results in Section~\ref{sec:standard-systems}, we know that
the foliation $\mathcal{G}_{s+1}$ is $\Lambda$-constant, but the foliation $\widetilde{\mathcal{G}}_{s+1}$ is not in general $\Lambda$-constant as it is showed in the following example.

\begin{Example}\label{ex-G-tilde-no-lambda-cte1}
Let $C$ be the curve defined by the parametrization $\varphi(t)=(x(t),y(t))$ with
\begin{equation*}
\left\{
\begin{aligned}
  x(t) & = t^4 \\
  y(t) & = t^{9}+t^{10}+\tfrac{19}{18}t^{11}
\end{aligned}
\right.
\end{equation*}
The basis of the semimodule $\Lambda_C$ is given by $\mathcal{B}=(4,9,14)$. The axes are equal to $u_1=13$,  $u_2=18$,   $\tilde{u}_1 =36$,   $\tilde{u}_2 =32$  and the critical values are given by $t_1=13$,  $t_2  = 17$,  $\tilde{t}_1  =36$,   $\tilde{t}_2  =31$.

We can take the element $\widetilde{\omega}_2$ of a standard system for $C$ equal to
$$\widetilde{\omega}_2=\left( y^2 - \tfrac{224}{81}x^5\right) (9 y dx  - 4x dy) + \left(x^7 -\tfrac{133}{54}y^3x \right) dx + \tfrac{148}{81} x^6 dy + dh
$$
with $h \in \mathbb{C}\{x,y\}$ such that $\varphi^* \widetilde{\omega}_2\equiv0$.
The curves  in $\operatorname{Cusp}_E( \widetilde{\mathcal G}_2)$ are the curves in the family of  cusps $\{C_a\}_{a \in \mathbb{C}^*}$ given by the parametrizations
$$
\varphi_a(t)=\left(t^4, at^9+\frac{t^{10}}{a^2} + \left(\frac{37a^4-18}{18a^5}\right) t^{11}+ \left(\frac{5-5a^4}{3a^8}\right)t^{12}+\cdots \right)
$$
Note that $C_1=C$.

If $a^4 \neq 1$,  a standard basis for a curve $C_a$ is given by $(dx,dy,\omega_1,\omega_2)$ where
$$\omega_1=9y dy - 4xdy; \qquad \omega_2=x \omega_1+\frac{4}{9a^4}ydy.$$
Hence, the basis of the semimodule of the   curve $C_a$ is $(4,9,14,19)$ and then $\Lambda_C \subsetneq \Lambda_{C_a}$.
\end{Example}

Next result gives a condition in terms of the divisorial value of a 1-form $\omega_\mathcal{F}$ defining a totally $E$-dicritical foliation $\mathcal{F}$ which allows to assure that all the $E$-cusps in $\operatorname{Cusp}_E(\mathcal{F})$ have the same semimodule of differential values.
\begin{Theorem}\label{th:lambda-cte}
Let $\mathcal{F}$ be a totally $E$-dicritical foliation in $(\mathbb{C}^2,0)$ with $C$ as invariant curve. If  $\mathcal{F}$ satisfies the total transversality property and
$$\nu_E(\omega_\mathcal{F}) \leq c(\Gamma_C)-c(\Lambda_C)+2(n+m)-1,$$
then the foliation $\mathcal{F}$ is $\Lambda$-constant.
\end{Theorem}
\begin{proof}
Let us consider any curve $C^* \in \operatorname{Cusp}_E(\mathcal{F})$. Since the foliation $\mathcal{F}$ satisfies the total transversality property,  Theorem~\ref{th:valores-dif-tilde-C} implies that the basis of the semimodule $\Lambda_{C^*}$ is given by
$$
\mathcal{B}_{C^*}=(\lambda_{-1},\lambda_0,\lambda_1,\ldots,\lambda_s, \lambda_{s+1}^{C^*},\ldots,\lambda_{s({C^*})}^{C^*}),
$$
where $\mathcal{B}=(\lambda_{-1},\lambda_0,\lambda_1,\ldots,\lambda_s)$ is the basis of the semimodule $\Lambda_C$.
By Theorem~\ref{th_caracterizacion-omega-i}, the element $\lambda_{s+1}^{C^*}$ can be computed as
$$
\lambda_{s+1}^{C^*}= \sup\{ \nu_{C^*}(\eta) \ : \ \nu_E(\eta)=t_{s+1}\}.
$$
Note that $u_{s+1}^{C^*}=u_{s+1}$ and $t_{s+1}^{C^*}=t_{s+1}=t_s+u_{s+1}-\lambda_s$ because these values are determined by $\Lambda_s^{C^*}=\Lambda_s^C=\Lambda_C$. Since $\nu_E(\omega_{s+1})=t_{s+1}$, then $\lambda_{s+1}^{C^*} \geq \nu_{C^*}(\omega_{s+1})$. Let us compute $\nu_{C^*}(\omega_{s+1})$. By Theorem~\ref{Th:nu_C}, we have that
$$
\nu_{C^*}(\omega_{s+1}) =\nu_E(\omega_{s+1}\wedge \omega_{\mathcal{F}}) -\nu_E(\omega_\mathcal{F}) + i_Q(\mathcal{J}_{\mathcal{F},\mathcal{G}_{s+1}}', {C^*}')
$$
where $Q$ is the infinitely near point of ${C^*}$ in $E$. By Theorem~\ref{Th:Saito_basis}, we can write
$$
\omega_{\mathcal{F}}=A \omega_{s+1} + B \widetilde{\omega}_{s+1}.
$$
Moreover, by Saito's criterion \cite{Sai-1980} (see Theorem~\ref{th:Saito-criterion}), we have that
$$
 \omega_{s+1} \wedge \widetilde{\omega}_{s+1} =u f dx \wedge dy
$$
where $u$ is unit and $f=0$ is a reduced equation of the curve $C$.
From the equalities above, we obtain that
\begin{align}
  \nu_{C^*}(\omega_{s+1}) & = \nu_E(B \omega_{s+1} \wedge \widetilde{\omega}_{s+1})-\nu_E(\omega_\mathcal{F}) + i_Q(\mathcal{J}_{\mathcal{F},\mathcal{G}_{s+1}}', {C^*}')   \notag\\
   & \geq \nu_E( B u f dx \wedge dy) - \nu_E(\omega_\mathcal{F})  \notag\\
   & =\nu_E(B) + nm +n+m- \nu_E(\omega_\mathcal{F}) \label{eq:omega_F}\\
   & \geq \nu_E(B) + nm +n+m-(c(\Gamma_C)-c(\Lambda_C)+2(n+m)-1)) \notag \\
   & = \nu_E(B) + c(\Lambda_C) \notag \\
   &  \geq c(\Lambda_C) \notag
  \end{align}
This implies that $\Lambda_{{C^*}}=\Lambda_{C}$ and we obtain the result.
\end{proof}

\begin{Remark}\label{rm:conductor-t_s+1}
By Lemma~\ref{lema_conductor_u_s+1} and Lemma~\ref{lemma:ui+ti}, we have that
$$c(\Lambda_C)=\tilde{u}_{s+1}  -n-m+1=nm-t_{s+1}+1.$$ Hence, we obtain
$$c(\Gamma_C)-c(\Lambda_C)+2(n+m)-1= t_{s+1} + n+m-1.$$
\end{Remark}
From Theorems~\ref{Th:total-transv-property} and \ref{th:lambda-cte}, we deduce
\begin{Corollary}
Let $\mathcal{F}$ be a totally $E$-dicritical foliation in $(\mathbb{C}^2,0)$ with $C$ as invariant curve.
 Denote $t_\star=\min\{\tilde{t}_{s+1}-t_s, n+m-1\}$. If $\nu_E(\omega_\mathcal{F}) \in [t_{s+1},t_{s+1}+t_\star)$, then the foliation $\mathcal{F}$ is $\Lambda$-constant.
\end{Corollary}
In order to prove the total transversality property in Theorem~\ref{Th:total-transv-property}, we need to consider two cases: either $\omega_\mathcal{F}$ is reachable by $\omega_{s+1}$ or by $\widetilde{\omega}_{s+1}$ but not from both as stated in Lemma~\ref{lemma_reachable_only_one}. Let us study these two cases separately.

\begin{Corollary}\label{Cor:alcanzable-omega-s+1}
Let $\mathcal{F}$ be a totally $E$-dicritical foliation in $(\mathbb{C}^2,0)$ with $C$ as invariant curve.
Assume that $\mathcal{F}$ satisfies the total transversality property and that $\omega_\mathcal{F}$ is reachable only from $\omega_{s+1}$. If the 1-form $\omega_{\mathcal{F}}$ satisfies that
\begin{equation}\label{eq:condicion-alcanzable-omega-s+1}
\nu_E(\omega_\mathcal{F} \wedge \widetilde{\omega}_{s+1}) -\nu_E(\omega_\mathcal{F} \wedge {\omega}_{s+1}) < n+m,
\end{equation}
then the foliation $\mathcal{F}$ is $\Lambda$-constant.
\end{Corollary}
Note that the condition given in equation~\eqref{eq:condicion-alcanzable-omega-s+1} is weaker than condition
\begin{equation}\label{eq:th-lambda-cte}
\nu_E(\omega_E)<t_{s+1}+n+m
\end{equation}
given in Theorem~\ref{th:lambda-cte}. In fact, if we write $\omega_{\mathcal{F}}= A \omega_{s+1} + B \widetilde{\omega}_{s+1}$, then
$$
\nu_E(\omega_\mathcal{F} \wedge \widetilde{\omega}_{s+1}) -\nu_E(\omega_\mathcal{F} \wedge {\omega}_{s+1}) = \nu_E(A)-\nu_E(B).
$$
Hence, the condition in \eqref{eq:condicion-alcanzable-omega-s+1} is equivalent to
$$
\nu_E(A) < \nu_E(B)+n+m.
$$
However, if $\omega_\mathcal{F}$ is reachable from $\omega_{s+1}$, then condition \eqref{eq:th-lambda-cte} implies $\nu_E(\omega_\mathcal{F})=\nu_E(A \omega_{s+1})<t_{s+1} +n +m$ which is equivalent to
$$
\nu_E(A) < n+m.
$$
The proof of Corollary~\ref{Cor:alcanzable-omega-s+1} is similar to the proof of Theorem~\ref{th:lambda-cte} using that $\nu_E(\omega_\mathcal{F})=\nu_E(A \omega_{s+1}) < t_{s+1} + \nu_E(B) + n +m$ in  inequality~\eqref{eq:omega_F}.

\medskip

Let us study what happens when $\omega_\mathcal{F}$ is reachable by $\widetilde{\omega}_{s+1}$. The condition over $\nu_E(\omega_\mathcal{F})$ given in Theorem~\ref{th:lambda-cte} implies that
\begin{equation}\label{eq:condicion-alcanzable-tilde-omega-s+1}
\tilde{t}_{s+1} \leq \nu_E(\omega_{\mathcal{F}}) < t_{s+1}+n+m.
\end{equation}
By Remark~\ref{rm:conductor-t_s+1}, the conductor $c(\Lambda_C)$ of the semimodule $\Lambda_C$ is equal to $c(\Lambda_C)=nm-t_{s+1}+1$. Moreover, by Lemma~\ref{lemma:ui+ti}, we have $u_{s+1}=nm+n+m-\tilde{t}_{s+1}$. Thus, the hypothesis in \eqref{eq:condicion-alcanzable-tilde-omega-s+1} imply
$$
u_{s+1} > nm-t_{s+1} = c(\Lambda_C)-1.
$$
Let us see that this condition implies that there is no element $\lambda_{s+1}^{C^*}$ in the basis $\mathcal{B}_{C^*}$  in the proof of Theorem~\ref{th:lambda-cte}. Since $\Lambda_{C^*}$ is an increasing semimodule, then $\lambda_{s+1}^{C^*} > u_{s+1}^{C^*}=u_{s+1} \geq c(\Lambda_C)$. Consequently, $\lambda_{s+1}^{C^*} \in \Lambda_C$ and hence $\Lambda_{C^*}=\Lambda_C$.

This proves the following result
\begin{Corollary}
Let $\mathcal{F}$ be a totally $E$-dicritical foliation in $(\mathbb{C}^2,0)$ with $C$ as invariant curve.
Assume that $\mathcal{F}$ satisfies the total transversality property and that $\omega_\mathcal{F}$ is reachable only from $\widetilde{\omega}_{s+1}$. If $\tilde{t}_{s+1} \leq t_{s+1}+n+m-1$,
then the foliation $\mathcal{F}$ is $\Lambda$-constant.
\end{Corollary}

In particular, we obtain that
\begin{Corollary}
If $\tilde{t}_{s+1} \leq t_{s+1}+n+m-1$, then the foliation $\widetilde{\mathcal{G}}_{s+1}$ is $\Lambda$-constant.
\end{Corollary}
Taking into account Corollary~\ref{Cor:alcanzable-omega-s+1}, we can construct families of $\Lambda$-constant foliations with
big multiplicity. Write the 1-form $\omega_\mathcal{F}$ in terms of the Saito basis $\{\omega_{s+1},\widetilde{\omega}_{s+1}\}$ as
$$ \omega_\mathcal{F}=A \omega_{s+1}+B \widetilde{\omega}_{s+1}
$$
with $\nu_E(A)-\nu_E(B)<n+m$.
If $\omega_\mathcal{F}$ is reachable only by $\omega_{s+1}$ and $\operatorname{In}(A)$ is a monomial, with the same arguments as in the proof of Theorem~\ref{Th:total-transv-property}, we can prove that $\mathcal{F}$ satisfies the total transversality property. Consequently we are in the conditions of Corollary~\ref{Cor:alcanzable-omega-s+1} and hence the foliation $\mathcal{F}$ is $\Lambda$-constant.
\begin{Example}
For each $j \in \mathbb{Z}_{\geq 1}$, consider the foliations $\mathcal{F}_j$ and $\widetilde{\mathcal{F}}_j$ defined by $\eta_j=0$ and $\widetilde{\eta}_j=0$ respectively, with
$$
\eta_j=A_j \omega_{s+1} + B_j \widetilde{\omega}_{s+1}; \qquad \widetilde{\eta}_j= \tilde{A}_j \omega_{s+1} + \widetilde{B}_j \widetilde{\omega}_{s+1},
$$
where $A_j=x^j$, $B_j= x^j + y^j$, $\tilde A_j=y^j$, $\widetilde{B}_j=y^j + x^{hj}$ and $h=\left[ \tfrac{m}{n}\right]+1$.

First note that $\nu_E(A_j)-\nu_E(B_j)=\nu_E(\tilde{A}_j)-\nu_E(\widetilde{B}_j)=0$. Using the results  in \cite[Section 4]{Can-C-SS-2023}, we can check that the foliations $\mathcal{F}_j$ and $\widetilde{\mathcal{F}}_j$ are totally $E$-dicritical foliations.
Moreover, for any $j$, the 1-form $\eta_j$ is only reachable by $\omega_{s+1}$ if either $s=0$ or $t_{s+1}=t_{s+1}^n$, and the 1-form $\widetilde{\eta}_j$ is only reachable by $\omega_{s+1}$ if either $s=0$ or $t_{s+1}=t_{s+1}^m$. In these cases, the  foliations $\mathcal{F}_j$ or $\widetilde{\mathcal{F}}_j$ are $\Lambda$-constant.
\end{Example}

\appendix

\section{Some properties of  increasing cuspidal semimodules}\label{sec:apendix-combinatoria}
Consider the semigroup $\Gamma=\langle n,m\rangle$ with $2 \leq n < m$ and $\gcd(n,m)=1$. Let $\Lambda$ be an increasing $\Gamma$-semimodule (that is, $\Lambda$ is a cuspidal semimodule) with basis $\mathcal{B}=(\lambda_{-1},\lambda_0,\lambda_1,\ldots,\lambda_s)$ and denote $\Lambda_i= \bigcup_{k=-1}^i (\lambda_k + \Gamma)$ for $i=-1,0,1,\ldots,s$. As in Section~\ref{sec:cuspidal-semimodules}, we can define the axes $u_i^n, u_i^m, u_i, \tilde{u}_i$ and critical values $t_i^n, t_i^m, t_i, \tilde{t}_i$ for the semimodule $\Lambda$ with $1 \leq i  \leq s+1$. Recall that the semimodule $\Lambda$ is increasing when $\lambda_i > u_i$ for $1 \leq i \leq s$.

From the definition of the axes, we get the following properties
\begin{Lemma}[\cite{Can-C-SS-2026}, Lemma 2.19]\label{lemma:u_i}
Let $\Lambda$ be a cuspidal semimodule. For any $0\leq i\leq s$, there exist two unique integer numbers $k_i^n$ and $k_i^m$ with $-1 \leq k_i^n, k_i^m \leq i-1$ and integers $b_{i+1} \geq 0$ and $a_{i+1} \geq 0$ such that
\begin{align*}
  u_{i+1}^n & = \lambda_i + n \ell_{i+1}^n = \lambda_{k_i^n} + m b_{i+1}, \\
  u_{i+1}^m & = \lambda_i + m \ell_{i+1}^m = \lambda_{k_i^m} + n a_{i+1}.
\end{align*}
\end{Lemma}
The integers $a_i$ and $b_i$ given in the previous result are called {\em colimits}. Recall that the integers $\ell_i^n, \ell_i^m$ above are called limits. Note that this lemma also holds for non-increasing cuspidal semimodules. Next result describe how to compute the integers $k_i^n$ and $k_i^m$ which are called {\em bounds}.
\begin{Lemma}[\cite{Can-C-SS-2026}, Lemma 2.27]\label{lemma:k_i}
Assume that $\Lambda$ is an increasing cuspidal semimodule. For any $1 \leq i \leq s$, we have
\begin{itemize}
  \item[(a)] if $u_i=u_i^n$, then $k_i^n=i-1$ and $k_i^m=k_{i-1}^m$;
  \item[(b)] if $u_i=u_i^m$, then $k_i^n=k_{i-1}^n$ and $k_i^m=i-1$.
\end{itemize}
\end{Lemma}
In the following example we show how the previous lemma helps in the computation of the axes.
\begin{Example}\label{ex:apendice-ki}
Let us compute the axes and critical values for the curve $C$ given in Example~\ref{Ex:total-transv-nuE-distinto-t} with basis $\mathcal{B}=(4,9,14,19)$. Note that $n=4$ and $m=9$, then $u_1=u_1^n=13$ and $\tilde{u}_1=u_1^m=36$.

Since $u_1=u_1^n$, then $k_1^n=0$ and $k_1^m=k_0^m=-1$. Therefore, to compute $u_2^n$ and $u_2^m$ we have to solve
$$
u_2^n=14 + 4 \ell_2^n=9 + 9b_2; \qquad u_2^m=14+ 9 \ell_2^m= 4+4a_2
$$
We get $\ell_2^n=b_2=1$ and $u_2^n=18$; $\ell_2^m=2$, $a_2=7$ and $u_2^m=32$. Then,  $u_2=18$ and $\tilde{u}_2=32$. We also obtain   the critical values $t_2=t_1+ 4 \ell_2^n=13+4=17$ and $\tilde{t}_2=t_2+ 9 \ell_2^m=13+18=31$.

Now, to compute $u_3$ and $\tilde{u}_3$, we have $k_2^n=1$ y $k_2^m=k_1^m=-1$ and hence we have to solve
$$
u_3^n=19 + 4 \ell_3^n=14 + 9b_3; \qquad u_3^m=19+ 9 \ell_3^m= 4+4a_3
$$
which gives $\ell_3^n=b_3=1$ and $u_3^n=23$; $\ell_3^m=1$, $a_3=6$ and $u_3^m=28$. Consequently, $u_3=23$, $\tilde{u}_3=28$, $t_3=t_2+4=21$ and $\tilde{t}_3=t_2+9=26$.
\end{Example}
Moreover, we recall a result which establishes a relationship between limits and colimits of the semimodule $\Lambda$.
\begin{Proposition}[\cite{Can-C-SS-2026}, Proposition 2.31]\label{prop:relaciones-limites-colimites}
Let $\Lambda$ be an increasing cuspidal semimodule. For any $1 \leq i \leq s$, we have
\begin{itemize}
  \item[(a)] if $k_i^n=i-1$, then $\ell_{i+1}^n + a_{i+1}=a_i$ and $\ell_{i+1}^m+b_{i+1}=\ell_i^m$;
  \item[(b)] if $k_i^m=i-1$, then $\ell_{i+1}^n + a_{i+1}=\ell_i^n$ and $\ell_{i+1}^m + b_{i+1}=b_i$.
\end{itemize}
\end{Proposition}
Let us prove some technical lemmas concerning a property of the limits needed in the proof of Lemma~\ref{lemma_reachable_only_one}.
\begin{Lemma}\label{lemma_limites}
Let $\Lambda$ be an increasing cuspidal semimodule.
For $1\leq i \leq s+1$, we have
$$t_i + n \ell_{i+1}^n + m \ell_{i+1}^m < \tilde{t}_i.$$
\end{Lemma}
\begin{proof}
We have two possibilities: either $u_i=u_i^n$ or $u_i=u_i^m$ which work in a similar way. We prove the result assuming that $u_i=u_i^n$ and hence we have $t_i=t_i^n$. By Lemma~\ref{lemma:k_i}, the bound $k_i^n$ is given by $k_i^n=i-1$ and  $u_{i+1}^n$ can be written as
$$
u_{i+1}^n=\lambda_i + n \ell_{i+1}^n = \lambda_{i-1} + m b_{i+1}.
$$
as shown in Lemma~\ref{lemma:u_i}. Then, we have that
\begin{align*}
  t_i + n \ell_{i+1}^n + m \ell_{i+1}^m & = t_i + \lambda_{i-1} -\lambda_{i} + m b_{i+1} + m \ell_{i+1}^m \\
   &  = t_i + \lambda_{i-1} -\lambda_{i} + m \ell_{i}^m
\end{align*}
where the last equality follows from the expression $b_{i+1}+ \ell_{i+1}^m =\ell_i^m$ given in Proposition~\ref{prop:relaciones-limites-colimites}. Since the semimodule $\Lambda$ is increasing, we have that $\lambda_i > u_i=\lambda_{i-1} +n \ell_i^n$ and we obtain that $\lambda_{i-1}-\lambda_i < -n\ell_i^n$. Moreover, since $t_i=t_i^n=t_{i-1} + n \ell_i^n$ and $\tilde{t}_i=t_i^m=t_{i-1}+m\ell_i^m$, we conclude that
\begin{align*}
  t_i + n \ell_{i+1}^n + m \ell_{i+1}^m  & = t_i + \lambda_{i-1} -\lambda_{i} + m \ell_{i}^m \\
   & < t_i-n\ell_i^n + m \ell_i^m=t_{i-1} + m \ell_i^m= \tilde{t}_i
\end{align*}
as desired.
\end{proof}
From the previous lemma we deduce the following result
\begin{Corollary}\label{corollary_limites}
Let $\Lambda$ be an increasing cuspidal semimodule with $\lambda_{-1}=n$ and $\lambda_0=m$.
We have
  $$
n \ell_i^n + m \ell_i^m <nm -n-m   \qquad \text{ for } \quad 2\leq i \leq s+1,
$$
and
$$ n \ell_1^n + m \ell_1^m <nm.
$$
\end{Corollary}
\begin{proof}
For $i\geq 2$, the result follows from the previous lemma and Lemma~\ref{lema_t_i_u_i} since
$$n \ell_i^n + m \ell_i^m < \tilde{t}_{i}-t_i \leq  \tilde{t}_1-t_1 =nm-n-m.$$

For $i=1$, we have that $u_1=n+m$ and hence $\ell_1^n=b_1=1$ whereas $\tilde{u}_1=nm$ and then $\ell_1^m=n-1$ and $a_1=m-1$. We obtain
$$ n \ell_1^n + m \ell_1^m=n+m(n-1)<nm.$$
\end{proof}
Moreover, we also have the following inequality
\begin{Lemma}\label{lema-lambda-t-u}
If $\Lambda$ is an increasing cuspidal semimodule, then
$$
\lambda_j + t_\ell^*-t_j < u_\ell^*
$$
with $1 < \ell  \leq s$ and $1 \leq j < \ell-1$.
\end{Lemma}
\begin{proof}
From the definition of the critical values we obtain that $t^*_\ell-t_{\ell-1}=u^*_\ell-\lambda_{\ell-1}$. We can write
\begin{align*}
  \lambda_j + t_\ell^*-t_j & =\lambda_j + t_\ell^* - t_{\ell-1} + \sum_{k=j}^{\ell-2} (t_{k+1}-t_{k})   \\
   & =  \lambda_j + (u_\ell^* - \lambda_{\ell-1}) +  \sum_{k=j}^{\ell-2}(u_{k+1} - \lambda_{k}) \\
   & =  u_\ell^*+ \sum_{k=j+1 }^{\ell-1}(u_k - \lambda_k) <  u_\ell^*
\end{align*}
where the last inequality follow from the fact $u_k - \lambda_k <0$ since $\Lambda$ is an increasing semimodule.
\end{proof}
As we state in Remark~\ref{Rm-def-critical-axes}, from the definition of the axes and critical values, we get
$$
u_i + \tilde{t}_i= \tilde{u}_i +t_i \qquad \text{ for } 1\leq i\leq s+1.
$$
In next result, we prove that the value above is constant for all index $i$.

\begin{Lemma}\label{lemma:ui+ti}
Let $\Lambda$ be an increasing cuspidal semimodule with $\lambda_{-1}=n$ and $\lambda_0=m$.  We have
\begin{itemize}
  \item[(i)] $\tilde{u}_1+t_1=nm+n+m$
  \item[(ii)] $\tilde{u}_{i+1} + t_{i+1}=\tilde{u}_i+t_i$  for any $1 \leq i \leq s$.
\end{itemize}
Consequently, we obtain $u_i + \tilde{t}_i= \tilde{u}_i +t_i =nm+n+m$ for all $1 \leq i \leq s+1$.
\end{Lemma}
\begin{proof}
Note that $\tilde{u}_1=nm$ and $t_1=n+m$ and hence $\tilde{u}_1+t_1=nm+n+m$ which gives the first assertion. In order to prove the equality in (ii),
we have to consider the following cases:
\begin{itemize}
  \item[(a)] $(u_i,\tilde{u}_i)=(u_i^n,u_i^m)$ and
  \begin{itemize}
    \item[(a-1)] $(u_{i+1},\tilde{u}_{i+1})=(u_{i+1}^n,u_{i+1}^m)$
    \item[(a-2)] $(u_{i+1},\tilde{u}_{i+1})=(u_{i+1}^m,u_{i+1}^n)$
  \end{itemize}
  \item[(b)] $(u_i,\tilde{u}_i)=(u_i^m,u_i^n)$ and
  \begin{itemize}
    \item[(b-1)]  $(u_{i+1},\tilde{u}_{i+1})=(u_{i+1}^n,u_{i+1}^m)$
    \item[(b-2)] $(u_{i+1},\tilde{u}_{i+1})=(u_{i+1}^m,u_{i+1}^n)$
  \end{itemize}
\end{itemize}
In case (a), we have that
\begin{equation*}
  \begin{aligned}
  u_i & = u_i^n= \lambda_{i-1} + n \ell_i^n = \lambda_{k_{i-1}^n} + m b_i \\
  \tilde{u}_i & = u_i^m =\lambda_{i-1} + m \ell_{i}^m =\lambda_{k_{i-1}^m} + n a_i.
\end{aligned}
\end{equation*}
Moreover, $k_i^n=i-1$ and $k_i^m=k_{i-1}^m$ by Lemma~\ref{lemma:k_i} and $a_{i+1}+\ell_{i+1}^n=a_i$ by Proposition~\ref{prop:relaciones-limites-colimites}.

In case (a-1), the hypothesis $(u_{i+1},\tilde{u}_{i+1})  = (u_{i+1}^n,u_{i+1}^m)$ gives
\begin{equation*}
  \begin{aligned}
  u_{i+1} & = u_{i+1}^n= \lambda_{i} + n \ell_{i+1}^n = \lambda_{k_{i}^n} + m b_{i+1}\\
  \tilde{u}_{i+1} & = u_{i+1}^m =\lambda_{i} + m \ell_{i+1}^m =\lambda_{k_{i}^m} + n a_{i+1}
\end{aligned}
\end{equation*}
and we deduce that
\begin{align*}
\tilde{u}_{i+1} + t_{i+1} & =\lambda_{k_{i}^m} + n a_{i+1}+ t_i + n\ell_{i+1}^n \\
& = \lambda_{k_{i-1}^m} + n(a_{i+1}+\ell_{i+1}^n) + t_i \\
& = \lambda_{k_{i-1}^m} + n a_i + t_i=\tilde{u}_i + t_i.
\end{align*}
Now, in case (a-2), we have that
\begin{equation*}
  \begin{aligned}
  u_{i+1} & =   u_{i+1}^m =\lambda_{i} + m \ell_{i+1}^m =\lambda_{k_{i}^m} + n a_{i+1} \\
  \tilde{u}_{i+1} &= u_{i+1}^n = \lambda_{i} + n \ell_{i+1}^n = \lambda_{k_{i}^n} + m b_{i+1}
\end{aligned}
\end{equation*}
and we obtain that
\begin{align*}
\tilde{u}_{i+1} + t_{i+1} & =\lambda_{i} + n \ell_{i+1}^n+ t_i + \lambda_{k_{i}^m} + n a_{i+1} - \lambda_i \\
& =\lambda_{k_{i}^m} + n (\ell_{i+1}^n + a_{i+1}) + t_i  \\
& = \lambda_{k_{i-1}^m} + n  a_i + t_i
 =\tilde{u}_i + t_i.
\end{align*}
The proof in case (b) works in a similar way.
\end{proof}

Moreover, we have the following property of the axes
\begin{Lemma}\label{lemma:ui-ui+1}
Let $\Lambda$ be an increasing cuspidal semimodule. For any $1 \leq i \leq s$, we have
$$\tilde{u}_i \in (u_{i+1} + \Gamma) \cup (\tilde{u}_{i+1}+\Gamma).$$
\end{Lemma}
\begin{proof}
By Lemma~\ref{lemma:k_i}, we have two possibilities $\tilde{u}_i=u_i^m$ and $k_i^n=i-1$ or $\tilde{u}_i=u_i^n$ and $k_i^m=i-1$. Both cases work in a similar way. Assume for instance that $\tilde{u}_i=u_i^m$ and $k_i^n=i-1$, hence
\begin{equation}\label{lemma:ui-ui+1-eq1}
\tilde{u}_i=u_i^m=\lambda_{i-1} + m \ell_i^n
\end{equation}
and by Lemma~\ref{lemma:u_i}, we can write
\begin{equation}\label{lemma:ui-ui+1-eq2}
u_{i+1}^n=\lambda_i + n \ell_{i+1}^n=\lambda_{i-1}+ m b_{i+1}.
\end{equation}
From Lemma~\ref{lema_t_i_u_i}, we obtain that $\tilde{u}_i>u_{i+1}^n$. Consequently, we deduce $\ell_i^n > b_{i+1}$ from the expressions given in equations~\eqref{lemma:ui-ui+1-eq1} and \eqref{lemma:ui-ui+1-eq2}. Hence, we can write
$$
\tilde{u}_i=u_{i+1}^n + m (\ell_{i}^n- b_{i+1}).
$$
The expression above gives the result since $u_{i+1}^n$ is equal to $u_{i+1}$ or to $\tilde{u}_{i+1}$.
\end{proof}
\subsection{Conductor of an increasing cuspidal semimodule}\label{ap:conductor}
 The objective of this section is to prove Lemma~\ref{lema_conductor_u_s+1} which states that the conductor $c(\Lambda)$ of an increasing cuspidal semimodule $\Lambda$ is equal to
\begin{equation}\label{eq:conductor-semimodulo}
c(\Lambda)=\tilde{u}_{s+1}-n-m+1.
\end{equation}
The computation of $c(\Lambda)$ is based on the results of \cite{Alm-M-2021} where the conductor $c(\Lambda)$ is computed in terms of the syzygy semimodule $\operatorname{Syz}(\Lambda)$ of $\Lambda$ and the generators of the semigroup $\Gamma$.

If $\Lambda$ is a cuspidal $\Gamma$-semimodule with basis $\mathcal{B}=(\lambda_{-1},\lambda_0,\lambda_1,\ldots,\lambda_s)$, the {\em syzygy}   $\operatorname{Syz}(\Lambda)$ of $\Lambda$ is the $\Gamma$-semimodule given by
$$
\operatorname{Syz}(\Lambda)= \bigcup_{-1 \leq i < j \leq s} ((\lambda_i + \Gamma) \cap (\lambda_j + \Gamma)).
$$
The semimodule $\operatorname{Syz}(\Lambda)$ consists of those elements in $\Lambda$ which admit more than
one presentation of the form $\lambda + \gamma$ with $\lambda \in \mathcal{B}$ and $\gamma \in \Gamma=\langle n,m\rangle$ (see \cite{Moy-U}). Moreover,  $\operatorname{Syz}(\Lambda)$ has a basis with $s+2$ elements (see  \cite[Theorem 4.3]{Moy-U}). Next result describes the basis of $\operatorname{Syz}(\Lambda)$ in terms of the axes of the semimodule $\Lambda$.
\begin{Proposition}\label{prop:base-siz}
Let $\Lambda$ be an increasing cuspidal $\Gamma$-semimodule with basis \linebreak $\mathcal{B}=(\lambda_{-1},\lambda_0,\lambda_1,\ldots,\lambda_s)$. Then
$$\mathcal{B}'=(u_1,u_2,\ldots, u_{s+1},\tilde{u}_{s+1})
$$
is the basis of $\operatorname{Syz}(\Lambda)$.
\end{Proposition}
\begin{proof}
Let $\Lambda'$ be the $\Gamma$-semimodule generated by $\mathcal{B}'$, that is,
$$
\Lambda'= \bigcup_{i=1}^{s+1}(u_i + \Gamma) \cup (\tilde{u}_{s+1} + \Gamma).
$$
By definition of the axes of $\Lambda$,  we have that $u_1,u_2,\ldots, u_{s+1},\tilde{u}_{s+1} \in \operatorname{Syz}(\Lambda)$ and hence $\Lambda' \subset \operatorname{Syz}(\Lambda)$.

Consider now $\lambda \in \operatorname{Syz}(\Lambda)$. We can write
$$
\lambda = \lambda_i + \gamma_i = \lambda_j + \gamma_j, \qquad \text{ with } \gamma_i, \gamma_j \in \Gamma, \ \ j < i.
$$
In particular, we obtain that $\lambda_i + \gamma_i \in \Lambda_{i-1}$. If we write $\gamma_i=na + mb$, we have that  $a \geq \ell_{i+1}^n$ or $b \geq \ell_{i+1}^m$ by Lemma~\ref{Lemma:a-b-l}. Let us assume that $a \geq \ell_{i+1}^n$ (the other case works in a similar way). Since $u_{i+1}^n=\lambda_i + n \ell_{i+1}^n$, we can write
$$
\lambda = \lambda_i + n a + m b = u_{i+1}^n + n (a-\ell_{i+1}^n)+mb.
$$
Let us see that $u_{i+1}^n \in \Lambda'$. If $u_{i+1}=u_{i+1}^n$ or $i=s$, then $u_{i+1}^n$ belongs to the basis $\mathcal{B}'$ and we have the result. Hence, we only have to prove that $\tilde{u}_{i+1} \in \Lambda'$ when $\tilde{u}_{i+1}=u_{i+1}^n$ and $i < s$, but this is consequence of the fact
$\tilde{u}_i \in (u_{i+1} + \Gamma) \cup (\tilde{u}_{i+1} + \Gamma)$ proved in Lemma~\ref{lemma:ui-ui+1}.
\end{proof}
In \cite[Theorem 1]{Alm-M-2021}, it is shown that the conductor $c(\Lambda)$ is given by
$$
c(\Lambda)=M-n-m+1
$$
where $M$ is the greatest element of the basis of $\operatorname{Syz}(\Lambda)$. Proposition~\ref{prop:base-siz} implies that $M=\tilde{u}_{s+1}$ and we obtain the expression of the conductor given in equation~\eqref{eq:conductor-semimodulo}. This finishes the proof of Lemma~\ref{lema_conductor_u_s+1}.

Note that the expression of the conductor given in equation~\eqref{eq:conductor-semimodulo} does not hold for non-increasing cuspidal semimodules as shown in the following example.
\begin{Example}
Consider the semigroup $\Gamma=\langle 5,11\rangle$ and the $\Gamma$-semimodule $\Lambda$ with basis $\mathcal{B}=(5,11,18,19)$. The axes of $\Lambda$ are given by
$$
\begin{aligned}
u_1& =16,  & \tilde{u}_1&=55; \\
 u_2& =u_2^n=22, \ &\tilde{u}_2&=u_2^m=40; \\
 u_3& = u_3^n=29, \   &\tilde{u}_3&=u_3^m=30,
\end{aligned}
$$
and hence $\Lambda$ is non-increasing since $u_2>\lambda_2$. Computing the elements of $\Lambda$ we obtain that $c(\Lambda)=18$ which does not coincide with the number $\tilde{u}_3-n-m+1=15$.
\end{Example}
\section{Delorme's decomposition}
In this appendix, we recall  Delorme's decomposition following the ideas of \cite{Del,Can-C-SS-2023,Can-C-SS-2026}. Let $C$ be a cusp with Puiseux pair $(m,n)$ with $2 \leq n <m$ and $\gcd(n,m)=1$. We will use the notations introduced in Section~\ref{sec:cuspidal-semimodules} for the cuspidal semimodule $\Lambda_C$.

Consider a standard system $(\mathcal{E},\widetilde{\mathcal{E}})$  for the curve $C$. Recall that
$\mathcal{E}=(\omega_{-1},\omega_0,\omega_1,\ldots, \omega_{s+1})$ is a extended standard basis for $C$ where the 1-forms $\omega_i$ satisfy
\begin{align*}
  \nu_C(\omega_i) & =\lambda_i \qquad \text{ for }i=-1,0,1,\ldots,s, \qquad \nu_C(\omega_{s+1})=\infty\\
  \nu_E(\omega_i) & =t_i \qquad \text{ for } -1 \leq i \leq s+1,
\end{align*}
and $\widetilde{\mathcal{E}}=(\widetilde{\omega}_1,\ldots,\widetilde{\omega}_s,\widetilde{\omega}_{s+1})$ is a family of 1-forms such that
$$
\nu_E(\widetilde{\omega}_i)=\tilde{t}_i \text{  and } \nu_C(\widetilde{\omega}_i)=\infty  \text{ for } i=1,\ldots,s, s+1
$$
 as explained in Section~\ref{sec:cuspidal-semimodules}. Let us state a decomposition result  for 1-forms  which generalizes Delorme's decomposition (see \cite{Del}) and the decomposition given in \cite[Theorem 8.5]{Can-C-SS-2023}. We denote by $(t_\ell^*,u_\ell^*)$ either $(t_\ell,u_\ell)$ or $(\tilde{t}_\ell,\tilde{u}_\ell)$. Moreover, we denote $k_i^*$ the bounds corresponding to the axes $u_{i+1}^*$, that is,
 $$k_{i}^*=\left\{
         \begin{array}{ll}
           k_i^n, & \hbox{ if } u_{i+1}^*=u_{i+1}^n\\
           k_i^m, & \hbox{ if }u_{i+1}^*=u_{i+1}^m
         \end{array}
       \right.
\qquad
 k_{i}=\left\{
         \begin{array}{ll}
           k_i^n, & \hbox{ if } u_{i+1}=u_{i+1}^n\\
           k_i^m, & \hbox{ if }u_{i+1}=u_{i+1}^m
         \end{array}
       \right.
  $$

\begin{Theorem}[Theorem 3.11, Corollary 3.14 in \cite{Can-C-SS-2026}]\label{Th:descomp-Delorme}
Let $\omega$ be a 1-form with $\nu_E(\omega)=t_{\ell}^*$ and $\nu_C(\omega)>u_{\ell}^*$ with $\ell \in \{1,2,\ldots,s,s+1\}$. Given any $j \in \{0,1,\ldots,\ell-1\}$, the 1-form $\omega$ can be written as
$$
\omega=\sum_{i=-1}^{j} h_i \omega_i
$$
with the following properties: if we denote $v_i^*=\nu_C(h_i \omega_i)$ for $i \in \{-1,0,1,\ldots,j\}$, then
\begin{itemize}
  \item[(i)] $v_j^*=\min \{\nu_C(h_i \omega_i) \ : \ -1 \leq i < j\}$;
  \item[(ii)] $v_j^*=\lambda_j + t_\ell^*-t_j$. In particular, in the case we take $j=\ell-1$, we get $v_{\ell-1}^*=\lambda_{\ell-1}+t^*_\ell-t_{\ell-1}=u_\ell^*$;
  \item[(iii)] if $j < \ell-1$, we have that  $\nu_C(h_j \omega_j)= \nu_C(h_{k_j} \omega_{k_j}) < \nu_C(h_i \omega_i)$ for any $i \neq k_j$, with $-1 \leq i \leq j-1$;
 \item[(iv)] if $j=\ell-1$, we have that $\nu_C(h_{\ell-1} \omega_{\ell-1})=\nu_C(h_{k_j^*} \omega_{k_j^*})=u_\ell^*$ and $\nu_C(h_{i} \omega_i)> u_\ell^*$ for $i \neq k_j^*$ and $-1 \leq i < \ell-1$.
\end{itemize}
 Moreover, we have that $\nu_E(\omega)=\nu_E(h_j \omega_j)<\nu_E(h_i \omega_i)$ for all $-1 \leq i <j$ and hence
 $\operatorname{In}(\omega)=\operatorname{In} (h_j \omega_j)$.
\end{Theorem}
\begin{Remark}
 Note that we have that $v_j^*<u_\ell^*$ by Lemma~\ref{lema-lambda-t-u} when $j < \ell-1$.
\end{Remark}

\end{document}